%% file: main.tex
\documentclass[twoside,11pt]{article}
\usepackage[preprint]{jmlr2e}

\usepackage{amsmath,amsthm,mathtools,fixmath}

\usepackage{url}
\numberwithin{equation}{section}
\usepackage{xcolor}
 \usepackage{indentfirst}
\usepackage{booktabs}
\usepackage{multirow}
\usepackage{tikz}
\usetikzlibrary{decorations.pathreplacing, decorations.pathmorphing, decorations.shapes}
\usetikzlibrary{arrows,calc}
\usepackage{mathrsfs,esint}
\usepackage{comment}
\usepackage{enumerate}
\usepackage{hyperref}    
\hypersetup{
     colorlinks = true,
     linkcolor = red,
     anchorcolor =green!50!black,
     citecolor = green!50!black,
     filecolor =green!50!black,
     urlcolor = green!50!black
}
\usepackage{cleveref} 
\usepackage{enumitem}
\usepackage{soul}

\usepackage{float}
\usepackage[ruled, vlined]{algorithm2e}

\usepackage{subcaption}
\usepackage{pgfplots}
\pgfplotsset{compat=newest, compat/show suggested version=false}
\usepackage{tikz}
\usetikzlibrary{intersections}
\usepackage{tikz-3dplot}
\tdplotsetmaincoords{60}{20}

\input{definitions}

\newcommand{\zs}[1]{{\color{blue!50!black} {\bf Zihan:} [#1] }}
\newcommand{\kp}[1]{{\color{purple} {\bf Konstantin:} [#1] }}
\newcommand{\xt}[1]{{\color{orange} {\bf Xiaochuan:} [#1] }}

\newcommand{\revision}[1]{{\color{black}{}#1}}

\usepackage{todonotes}

\usepackage{lastpage}
\jmlrheading{23}{2025}{1-\pageref{LastPage}}{6/05}{9/22}{21-0000}{Shao, Pieper And Tian}

\ShortHeadings{Solving Nonlinear PDEs with Sparse RBF Networks}{Shao, Pieper And Tian}
\firstpageno{1}

\begin{document}

\title{Solving Nonlinear PDEs with Sparse Radial Basis Function Networks}

\author{\name Zihan Shao \email z6shao@ucsd.edu \\
\addr Department of Mathematics\\
University of California, San Diego\\
La Jolla, CA 92093, United States
\AND
\name Konstantin Pieper \email pieperk@ornl.gov \\
\addr Computer Science and Mathematics Division\\
Oak Ridge National Laboratory\\
Oak Ridge, TN 37831, United States
\AND
\name Xiaochuan Tian \email xctian@ucsd.edu \\
\addr Department of Mathematics\\
University of California, San Diego\\
La Jolla, CA 92093, United States
}

\editor{My editor}

\maketitle

\begin{abstract}%  
We propose a novel framework for solving nonlinear PDEs using sparse radial basis function (RBF) networks. Sparsity-promoting regularization is employed to prevent over-parameterization and reduce redundant features. 
This work is motivated by longstanding challenges in traditional RBF collocation methods, along with the limitations of physics-informed neural networks (PINNs) and Gaussian process (GP) approaches, aiming to blend their respective strengths in a unified framework. 
The theoretical foundation of our approach lies in the function space of Reproducing Kernel Banach Spaces (RKBS) induced by one-hidden-layer neural networks of possibly infinite width. We prove a representer theorem showing that the sparse optimization problem in the RKBS admits a finite solution and establishes error bounds that offer a foundation for generalizing classical numerical analysis.  The algorithmic framework is based on a three-phase algorithm to maintain computational efficiency through adaptive feature selection, second-order optimization, and pruning of inactive neurons. Numerical experiments demonstrate the effectiveness of our method and highlight cases where it offers notable advantages over GP approaches. This work opens new directions for adaptive PDE solvers grounded in rigorous analysis with efficient, learning-inspired implementation.
\end{abstract}


\section{Introduction}

This work introduces a sparse neural network approach for solving nonlinear partial differential equations (PDEs).
An adaptive training process is introduced for shallow neural networks with sparsity-promoting regularization, where neurons are gradually added to maintain a compact network structure. 
The PDEs considered in this paper are defined in a bounded open set $D\subset \R^d$ and subject to ``boundary conditions'' on $\partial D$ of the following form:
\beq
\label{eq:main}
\left\{
\begin{aligned}
& \cE[u] (x)  = 0,\quad x \in D, \\
& \cB [u] (x) = 0,\quad  x \in \partial D.  
\end{aligned}
\right.
\eeq 
Here, $ \cE[u](x) $ and $ \cB[u](x) $ are defined as  
\beq
\label{eq:PDEform}
\cE[u](x) := E(x, u(x), \nabla u(x), \nabla^2 u(x)), \quad  
\cB[u](x) := B(x, u(x), \nabla u(x)).
\eeq
where $E$ is a real-valued function on $\overline{D} \times \R \times \R^d \times \R^{d\times d}$ and 
$B$ is a real-valued function on $\partial D \times \R \times \R^d $  that determine the partial differential equation on \(D\) and the boundary condition on \(\partial D\), respectively. 
We note that here $\cB$ is a general boundary operator, which may lead to boundary or initial value problems in the context of PDEs.
A typical example falling within this framework is a second-order semilinear elliptic PDE with Dirichlet boundary conditions. Additional assumptions on the PDEs and further examples are provided in \Cref{subsec:1.3}.

Classical numerical methods for solving PDEs, such as finite difference methods and finite element methods, have been extensively developed and broadly applied since the mid-20th century. Meshless methods represent a more recent development, with origins tracing back at least to the 1970s through the introduction of smoothed particle hydrodynamics for astrophysical simulations \citep{gingold1977smoothed,lucy1977numerical}. These methods gained increasing attention in the 1990s, leading to the emergence of various approaches, including the element-free Galerkin method, reproducing kernel particle methods, radial basis function (RBF) methods, and others \citep{belytschko1994element,buhmann2000radial,fasshauer1996solving,franke1998solving,liMeshfreeParticleMethods2007,liu1996overview,wendland2004scattered}. 
A notable connection exists between meshless methods and machine learning through the use of kernel techniques, which have achieved considerable success in various machine learning applications \citep{scholkopf2018learning}.  The link between kernel-based approaches in machine learning and numerical solutions of PDEs was insightfully highlighted in the survey by \citet{schaback2006kernel}.
Kernel-based approaches are particularly effective for linear problems, such as regression or numerical solutions of linear PDEs, where the underlying solvers depend on techniques from numerical linear algebra. The recent popularity of neural networks has shifted the focus toward inherently nonlinear techniques, even for linear problems \citep{weinan2018deep,goodfellow2016deep,raissi2019physics}. As a result, optimization has begun to replace traditional linear algebra as the central computational framework.
However, neural network-based approaches, such as physics-informed neural networks (PINNs) \citep{raissi2019physics}, present significant challenges in training, due to their sensitivity to hyperparameter tuning and the substantial computational cost associated with large, often over-parameterized architectures that may include many redundant features.
A notable recent development is the Gaussian process (GP) approaches \citep{batlle2025error,chen2021solving,chen2025sparse}, which extend kernel-based methods to solving nonlinear PDEs. 
The central idea of GP typically involves solving a linear system as an inner step combined with an outer optimization step. However, since they are reduced to classical kernel-based or RBF methods in linear cases, many of the limitations associated with traditional RBF approaches still apply.

This work is originally motivated by the aim of addressing some of the persistent limitations of traditional RBF collocation methods. 
While RBF collocation methods are well known for their advantages, such as being inherently meshfree and offering fast or even spectral convergence in certain cases, they also face significant challenges. A central difficulty lies in the choice of the kernel scaling parameter, which plays a critical role in the success of the method but lacks a universal selection criterion. Choosing this parameter typically involves a trade-off between approximation accuracy and the condition number of the resulting system, and there is no clear guideline for achieving an optimal balance. For large linear systems, RBF collocation often leads to extremely ill-conditioned systems. In their survey, \citet{schaback2006kernel} envisioned that one important direction for developments in kernel-based approximation should focus on adaptive algorithms that solve the problems approximately, thereby alleviating the issue of ill-conditioning in large systems. 
The proposed approach, which is broadly applicable, is an actualization of their concept of adaptivity by selecting only a small number of features through sparsity-promoting regularization. Importantly, the kernel scale parameter is also determined adaptively, eliminating the need for tuning and resolving the long-standing issue of scale selection.
This represents a new development, even in the context of kernel-based linear regression or function approximation.
Our work is also motivated by the goal of addressing some of the issues of PINNs and GP methods. In this sense, the proposed approach can be viewed as: (1) a PINN framework trained using a shallow and sparse neural network for more efficient optimization, and (2) an extension of GP methods that enables adaptive selection of kernel parameters.
We note that the proposed approach can be viewed as a relaxation of kernel collocation methods, as the empirical loss is constructed from pointwise evaluations of PDEs. This leads to an interpretation of the method as an ``adaptive collocation solver". The formulation is intentionally simple, allowing us to clearly demonstrate the potential of the idea, both in terms of theoretical foundation and practical performance. Lastly, the adaptivity in our approach lies in the choice of trial functions. Other extensions, such as alternative loss function formulations or adaptivity in the selection of test points, may serve as promising directions for future research.\\

\paragraph{\it Structure of this section.}  We begin in \Cref{subsec:1.1} with a summary of our approach and a highlight of the main contributions. \Cref{subsec:1.2} discusses related work, where we compare our method with existing frameworks. 
 %including RKHS-based approaches, Gaussian processes, random feature models, and physics-informed neural networks (PINNs)
 In \Cref{subsec:1.3}, we introduce the notation and assumptions used throughout the paper and provide representative examples of the PDE problems considered. Finally, \Cref{subsec:1.4} outlines the structure of the rest of the paper.

\subsection{Summary of approach and contributions}
\label{subsec:1.1}
A shallow neural network of $N$ neurons is a function defined as
\begin{equation}
\label{eq:discrete_network}
    u_{c, \omega} (x) = \sum_{n=1}^{N} c_n \varphi(x; \omega_n),
\end{equation}
where $\varphi$ is the feature function, $c= \{c_n\}_{n=1}^{N}\subseteq \R$ represents the outer weights, and  $\omega = \{\omega_{n}\}_{n=1}^{N} \subseteq \Omega$ denotes the inner weights, with $\Omega$ being a prescribed parameter space. $N$ is also referred to as the network width. 
To illustrate our approach, we consider Gaussian Radial Basis Function (RBF) networks, where the feature function $\varphi$ is defined by
\revision{
\[
    \varphi(x;\; y, \sigma) = \frac{\sigma^{s}}{\left(\sqrt{2\pi}\sigma\right)^d}\exp\left(-\frac{\|x - y\|_{2}^{2}}{2\sigma^{2}}\right). 
\]
}
In this case, the inner weights $\{ \omega_{n} = (y_n, \sigma_n) \in \R^{d} \times \R_{+} \}$ represent the centers and shapes of the Gaussian functions. We note that the feature function is the standard Gaussian probability density function multiplied by the scale-dependent weight \(\sigma^s\), so that \(\int_{\R^d} \varphi(x;\; y, \sigma) \de x = \sigma^s\).
\revision{A common approach in the RBF literature is to omit the scaling in front of the exponential entirely, which corresponds to a choice of \(s = d\) up to a constant pre-factor \((\sqrt{2\pi})^{-d}\). For a physics informed loss associated to a PDE it is important to consider general \(s\), and set \(s\) larger than \(d+k\), where $k$ denotes the highest order of the differential operators evaluated in $\cE[u]$ and $\cB[u]$. Together with the regularization of the outer weights, this ensures stability and generalization properties; see \Cref{sec:theory}.}

As is customary in neural network-based PDE approximation, we define a measure \(\nu_D\) on \(D\) and \(\nu_{\partial D}\) on its boundary and consider the squared residual loss function 
\begin{align}
    \label{eq:loss}
    L(u) &=\frac{1}{2} \norm{\cE[u] }^2_{L^2(\nu_D)} + \frac{\lambda}{2} \norm{\cB[u]}^2_{L^2(\nu_{\partial D})}
\end{align}
where \(\lambda\) is an appropriate penalty parameter. 
%\beq
%\tag{$P_{c,\omega}$}
%\label{eq:sparse_min_discrete}
%\min_{N, c, \omega} \frac{1}{2} L(u_{c, \omega}) + \alpha \|c\|_{1}.
%\eeq
In practice, we use the empirical loss function $\hat{L} = \hat{L}_{K_1, K_2}$ defined as 
\begin{equation}
\label{eq:loss_discrete}
    \revision{\hat{L}_{K_1, K_2}(u)}=  \frac{1}{2} \sum_{k=1}^{K_1} w_{1,k}(\cE[u](x_{1,k}))^2 + \frac{\lambda}{2} \sum_{k=1}^{K_2}w_{2,k}(\cB[ u](x_{2,k}) )^2, 
\end{equation}
where $\{x_{1,k}\}_{k=1}^{K_1} \subseteq D$ and $\{x_{2,k}\}_{k=1}^{K_2} \subseteq \partial D$ \revision{are quadrature or collocation points, and $\{w_{1,k}\}_{k=1}^{K_1}\subseteq \R^+$ and  $\{w_{2,k}\}_{k=1}^{K_2} \subseteq \R^+$ are the associated (positive) weights used to approximate the two integrals in~\eqref{eq:loss} and the total number of points is $K=K_1+K_2$. For instance, we can consider a finite number of randomly generated points or (quasi-)uniform grids; the precise assumptions are stated in Assumption~\ref{assu:measureconvergence}.}

% \kp{Note: I have added the assumption that all weights are positive. I think this is something necessary, although it causes a restriction in general. For PINNs as in our numerics this is always the case.}

The neural network training is based on a sparse minimization problem with an $\ell_{1}$ regularization term:
\begin{equation}
\tag{$P_{N}^{\text{emp}}$}
\label{eq:emperical_sparse_min_discrete}
    \min_{N, c, \omega }   \hat{L}_{K_1, K_2}(u_{c, \omega}) + \alpha \|c\|_{1}
\end{equation}

Unlike most existing approaches, we do not fix the network width $N$ a prior; instead, it is treated as part of the optimization problem. However, under this formulation, the existence of a minimizer for the empirical problem \eqref{eq:emperical_sparse_min_discrete} is unclear. Conceptually, without a bound on the network width $N$, we might get a network solution of infinite width. Hence, we generalize the discrete network representation \eqref{eq:discrete_network} to a continuous formulation by introducing the integral neural network \citep{bach:2017,bengio2005convex,pieper2022nonconvex,rosset2007}, where $u$ is defined by
\begin{equation*}
\revision{u_{\mu}(x)} = \int \revision{\varphi(x;\omega)}\de\mu(\omega), \quad \mu\in M(\Omega).
\end{equation*}
Here $M(\Omega)$ is the space of signed Radon measures endowed with the total variation norm $\|\cdot\|_{M(\Omega)}$.
\revision{We also note that $u_\mu(x)$ is defined pointwise as a continuous function for \(s \geq d\) and an appropriate function space is introduced in the following.}
% \kp{I think we can directly see that \(s \geq d\) is sufficient for continuity. The result from Meyer is essentially that this gives the space \(B^0_{1,1} \subset C(D)\)}
Associated to the infinite width networks, we consider the continuous optimization problems
\begin{align}
    &\min_{\mu\in M(\Omega)}  L(u_{\mu}) + \alpha \|\mu\|_{M(\Omega)}\label{eq:sparse_min}\tag{$P_{\mu}$}\\
    &\min_{\mu\in M(\Omega)}  \hat{L}_{K_1,K_2} (u_{\mu}) + \alpha \|\mu\|_{M(\Omega)}\label{eq:empirical_sparse_min}\tag{$P_{\mu}^{\text{emp}}$}
\end{align}
We note the above characterization degrades to the prior discrete case when $\mu$ is a finite linear combination of Dirac-deltas.

We will establish the existence of minimizers for both \eqref{eq:sparse_min} and \eqref{eq:empirical_sparse_min}. More significantly, we will show that \eqref{eq:empirical_sparse_min} possesses a finite solution, reducing the problem to the discrete formulation \eqref{eq:emperical_sparse_min_discrete}. Nevertheless,  solving \eqref{eq:emperical_sparse_min_discrete} remains nontrivial since $N$ cannot be optimized through gradient-based methods. We then employ an adaptive search of neurons, where neurons are dynamically added and removed to maintain a compact network structure while minimizing the regularized loss, \revision{which we outline below in \Cref{subsec:adaptive_methods}.}
%\revision{\st{It follows an iterative refinement process found in sparse approximation and boosting-style} methods~\citep{friedman2001greedy,needell2009cosamp}. \st{The results can accurately approximate the PDE solution with high sparsity.}}

We note that functions represented by integral neural networks form a Banach space
\[
\cV(D) = \left\{  \int_\Omega \varphi(x; \omega)\de\mu(\omega)  \;\big|\; \mu \in M(\Omega) \right\}
\]
equipped with norm
\[
\norm{u}_{\cV(D)}
= \inf \left\{\norm{\mu}_{M(\Omega)} \;\big|\; \mu \in M(\Omega)\colon \int_\Omega \varphi(\cdot; \omega)\de\mu(\omega)= u \text{ on } D\right\}.
\]
This is a {\it reproducing kernel Banach space} (RKBS) by definition in \cite{bartolucci2023understanding}, and $\varphi : D \to C_0(\Omega)$ is the associated feature map. Equivalently,  \eqref{eq:sparse_min} and~\eqref{eq:empirical_sparse_min} can
be reformulated as
\begin{align}
\min_{u \in \cV(D)}  L(u) + \alpha \| u \|_{\cV} \quad \text{and} \quad \min_{u \in \cV(D)} \hat{L}_{K_1, K_2}(u) + \alpha \| u \|_{\cV} . 
\end{align}
These formulations identify the natural function space for shallow neural networks, where regularization is imposed through the Banach norm. 

The space $\cV$ is also referred to as an integral RKBS in the literature \citep{spek2025duality}, and is closely related to variation spaces and Barron spaces \citep{e2022barron,parhi2021banach,parhi2025function,siegel2023characterization}.
\revision{For the Gaussian feature function~\eqref{eq:scaledGaussian} with scaling \(s\), the associated RKBS is closely related to the classical Besov space \(B^s_{1,1}\); see \Cref{subsec:neural_net}. A precise relationship between the RKBS, the variation space, and the Besov space for general radial basis kernels is established in \cite{shao2026sparse}.}
For related studies on RKBS-based learning, we refer the readers to \cite{kumar2024mirror,parhi2022near,parhi2025function} and the references therein.

%\xt{References added}

Based on the empirical problem \eqref{eq:emperical_sparse_min_discrete}, we propose an efficient numerical framework for solving nonlinear PDEs in the RKBS. The main contributions of our work are summarized as follows.

\begin{itemize}[left=0pt, label={\tiny$\bullet$}]
\item {\it An adaptive extension of kernel collocation methods}: The proposed scheme can be interpreted as an adaptive collocation solver, extending classical RBF-based methods while resolving key issues such as scale selection and large system ill-conditioning.
\item {\it A framework integrating ideas from PINNs and GP methods}:  Our approach can be viewed as a PINN-type training scheme using a shallow, sparse neural network for improved optimization and an extension of the GP mechanism with adaptive kernel parameter selection. 
\item {\it A representer theorem}: We prove a representer theorem showing that the solution to the \revision{continuous empirical problem \eqref{eq:empirical_sparse_min} admits a finite linear combination of features; as a result, the problem is practically reduced to the discrete empirical problem  \eqref{eq:emperical_sparse_min_discrete}.}  
\item {\it Provable convergence guarantees}:
\revision{We provide a rigorous theoretical analysis of the convergence of the solutions of the continuous empirical problem \eqref{eq:empirical_sparse_min} to the solution of the PDE in the limit $K\to\infty$ and $\alpha\to0$.}
%We present a rigorous theoretical analysis establishing error bounds for the proposed method, offering a foundation that suggests new directions for extending traditional numerical analysis frameworks. 
\item {\it An efficient computational framework}: We design a three-phase algorithm that maintains a compact network structure through adaptive neuron insertion (guided by dual variables), efficient optimization via a second-order Gauss-Newton method, and pruning of inactive neurons.
\item {\it Demonstrated numerical advantages}: Through numerical examples and supported discussions, we illustrate the advantages of our method compared to GP and PINN approaches.
\item {\it Broad applicability}: Although developed in the context of RBF networks, the proposed framework is compatible with general shallow neural networks and other activation functions.
\end{itemize}

To illustrate the effectiveness of our approach, we present a simple 1D example comparing our results with the results obtained with the GP method from~\citet{chen2021solving} ~\citep[using the code in the associated repository;][]{chen2024nonlinpdes}.
\Cref{fig:intro_sols} demonstrates that our method accurately recovers the solution using a small number of adaptively chosen kernels. Here, the stars indicate the RBF center and shape parameter, with larger stars corresponding to larger RBF bandwidths. In contrast, \Cref{fig:intro_GP_linesearch,fig:intro_GP_sols} show that the GP method is highly sensitive to the kernel shape parameter. For this example, accurate recovery using GP is only possible within a very narrow range of $\sigma$ and deviations lead to significant errors. 

\begin{figure}[t]
    \centering
    % First row (3 figures)
    \begin{subfigure}[t]{0.32\textwidth}
        \includegraphics[width=\textwidth]{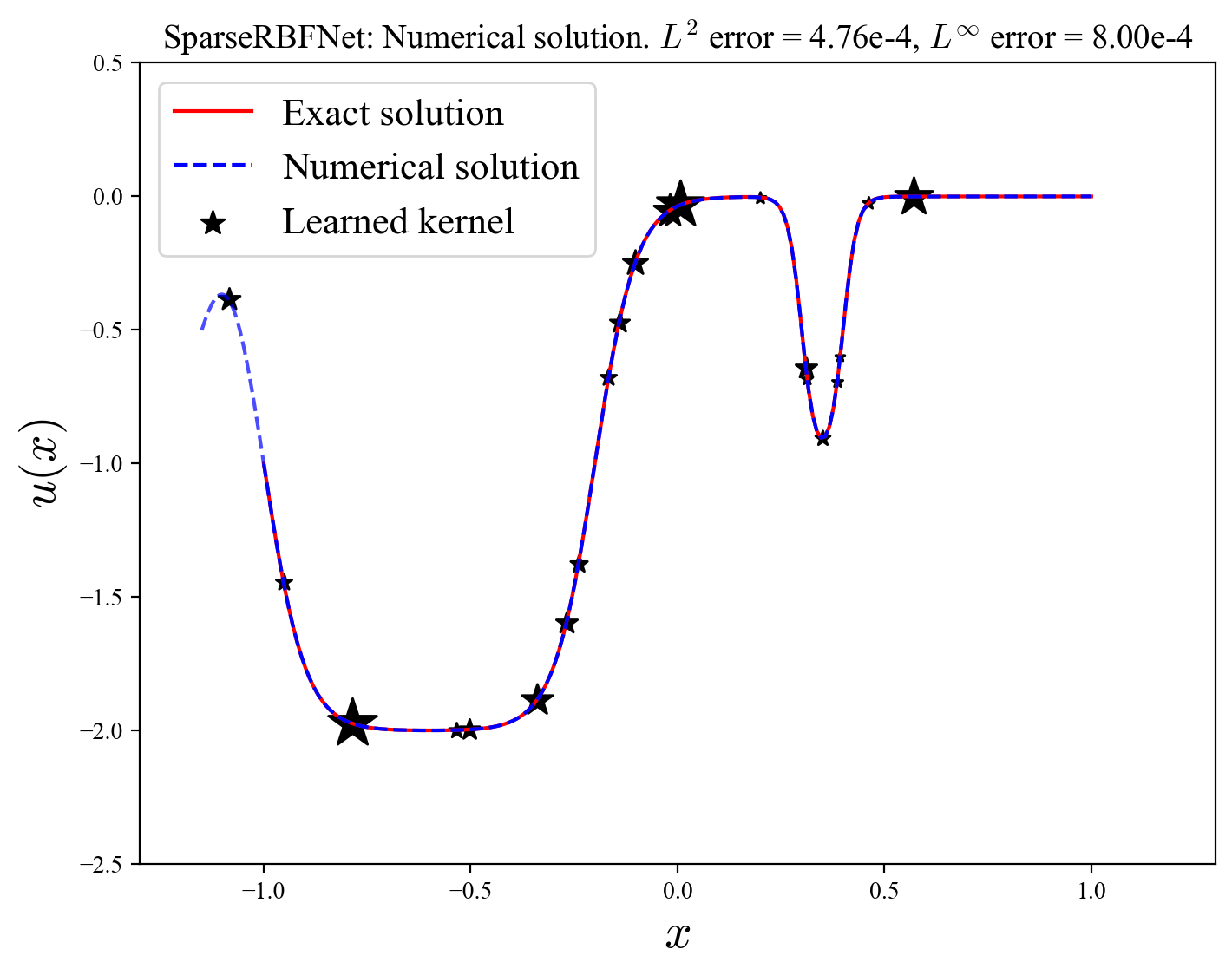}
        % \caption{Sparse RBFNet solution with $\alpha = 10^{-4}$, $\lambda = 100$.}
        \caption{}
        \label{fig:intro_sols}
    \end{subfigure}
    \hfill
    \begin{subfigure}[t]{0.32\textwidth}
        \includegraphics[width=\textwidth]{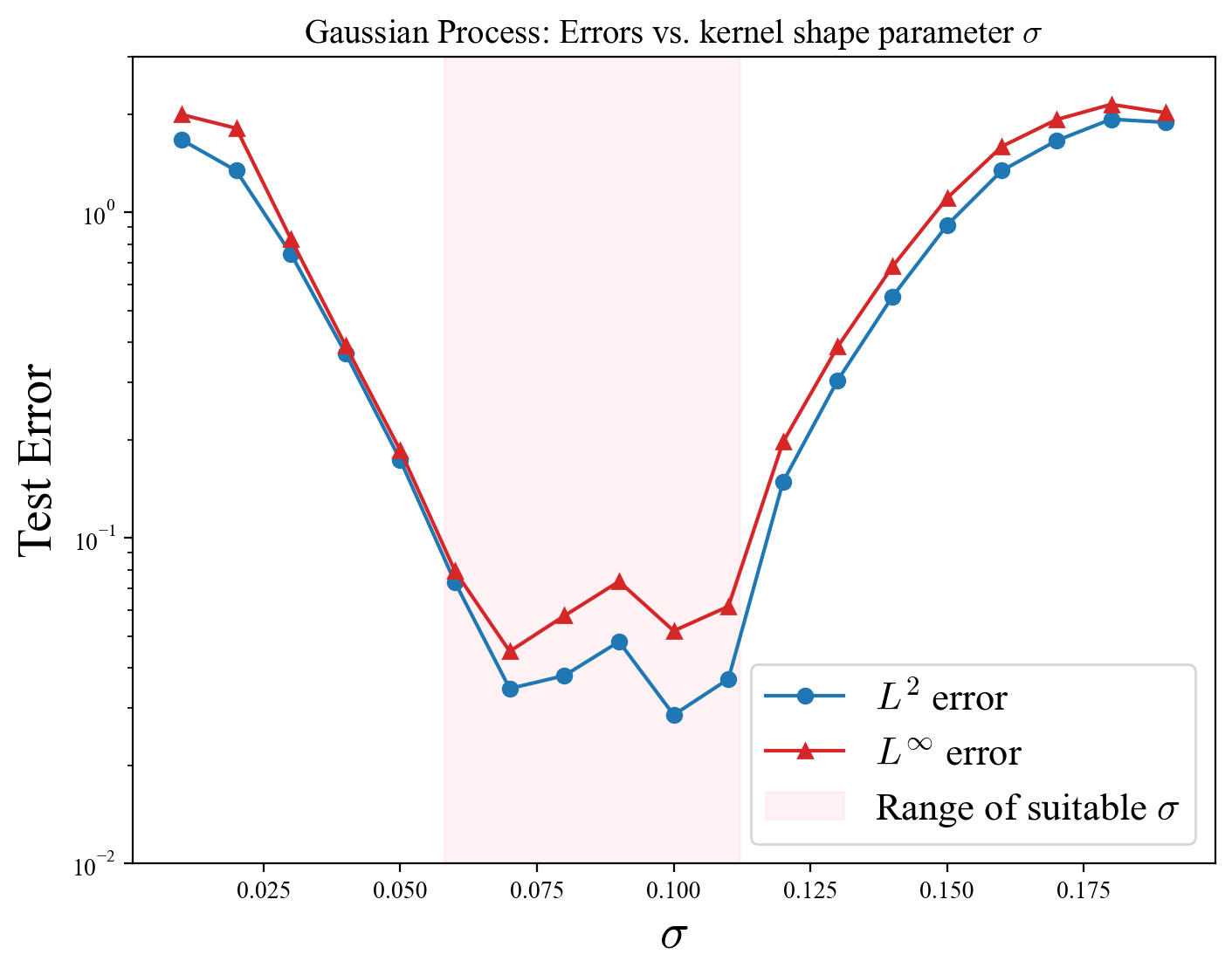}
        % \caption{Error curves for GP solutions with different kernel shape parameters.}
        \caption{}
        \label{fig:intro_GP_linesearch}
    \end{subfigure}
    \hfill
    \begin{subfigure}[t]{0.32\textwidth}
        \includegraphics[width=\textwidth]{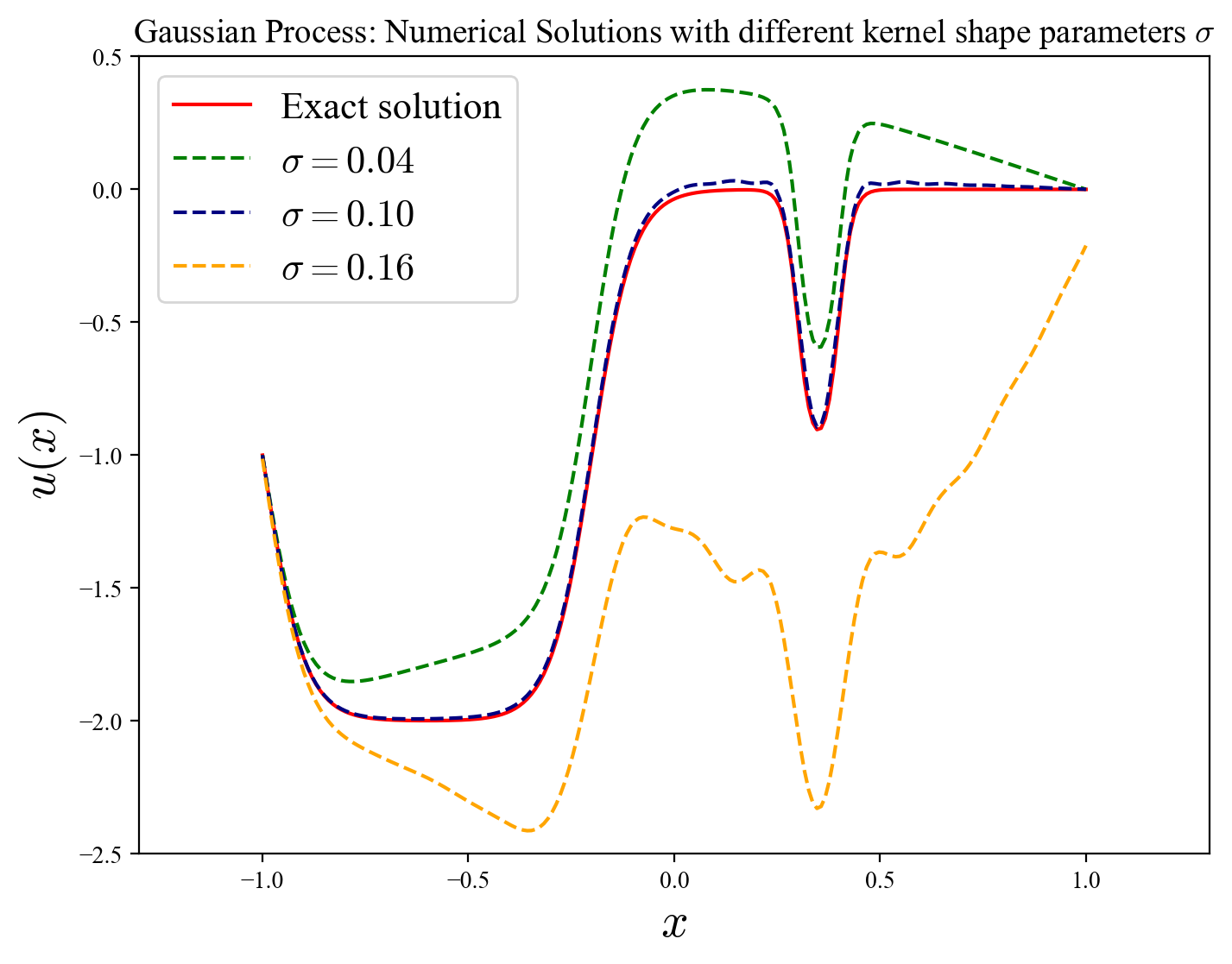}
        % \caption{GP solutions with different kernel shape parameters.}
        \caption{}
        \label{fig:intro_GP_sols}
    \end{subfigure}
    \caption{Numerical experiments on a 1D semilinear Poisson equation with Sparse RBFNet (our) and Gaussian Process. The exact solution $u$ is set to be 
    $u(x) = \tanh(10(x + 0.2)) - \tanh(10(x + 1.0)) + 0.5 \left( \tanh(30(x - 0.4)) - \tanh(30(x - 0.3)) \right)$. $K = 100$ grid collocation points are used for both methods; $L^{2}$ and $L^{\infty}$ error shown above are estimated at test grid of size $200$. (a) Sparse RBFNet solution with \revision{$\ell^{1}$-regularization parameter} $\alpha = 10^{-4}$, \revision{penalty weight} $\lambda = 100$; (b) Error curves for GP solutions with different kernel shape parameters; (c) GP solutions with different kernel shape parameters.}
    \label{fig:intro_semilinear}
\end{figure}

\subsection{Related studies}
\label{subsec:1.2}
Let us briefly mention some related methods and their conceptual connection to the proposed approach.
\subsubsection{The Hilbert space setting}
Given a probability measure $\rho$ over $\Omega$, one can define a space
\[
\cH(D) = \left\{ \int_{\Omega} \varphi(x ; \omega) c(\omega) \de\rho(\omega) \;|\; c \in L^2(\Omega, \de\rho)\right\}
\]
equipped with norm
\[
\norm{u}_{\cH} =  \inf \left\{\norm{c}_{L^2(\Omega, \de\rho)} \;|\; c \in L^2(\Omega, \de\rho)\colon \int_{\Omega} \varphi(\cdot ; \omega) c(\omega) \de\rho(\omega)  = u \text{ on } D \right\}.
\]
Given $u\in \cH$, the above infimum \revision{is attained by a unique minimizer 
$c_u \in L^2(\Omega, \de\rho)$} using the direct method in the calculus of variations. 
$\cH$ is a Hilbert space  equipped with inner product
\[
\pair{u, v}_{\cH}  := \pair{c_u, c_v}_{L^2(\Omega, \de\rho)}. 
\]
We also observe that $\cH$ is a reproducing kernel Hilbert space (RKHS) associated with the kernel \citep{bach:2017}
\beq
\label{eq:kernel}
k(x,x') = \int_{\Omega} \varphi(x; \omega) \varphi(x';\omega) \de\rho(\omega).
\eeq
This can be verified easily as follows.
For a fixed $x'\in D$, we have $k(\cdot, x') \in \cH$. If $\varphi(x', \cdot) \in L^2(\Omega,\de\rho) $ achieves $\norm{k(\cdot, x')}_{\cH} = \norm{\varphi(x', \cdot)}_{L^2(\Omega,\de\rho)}$,  then it follows that 
\[
\pair{u, k(\cdot, x)}_{\cH} = \pair{c_u, \varphi(x, \cdot) }_{L^2(\Omega, d\rho)} = \int_{\Omega} c_u(\omega) \varphi(x; \omega) \de\rho(\omega) = u(x),
\]
which verifies the reproducing property of the kernel.
In this Hilbert space framework, the minimization problem 
\beq
\label{eq:empirical_l2_min}
\min_{c \in L^2(\Omega,\de\rho)} \hat{L}_{K_1, K_2}(u_{c\de\rho}) + \frac{\alpha}{2} \norm{c}^2_{L^2(\Omega, \de\rho)}
\eeq
is equivalent to the minimization problem in the RKHS space:
\beq
\label{eq:empirical_RKHS_min}
\min_{u\in \cH} \hat{L}_{K_1, K_2}(u) + \frac{\alpha}{2} \norm{u}^2_{\cH}.
\eeq
To compare with our approach, we observe that
\[
\left(\int_{\Omega} |c(\omega)| \de\rho(w)\right)^2 \leq \int_{\Omega} |c(\omega)|^2 \de\rho(w),
\]
which naturally implies that $\cH \hookrightarrow \cV$. 
This embedding is in fact strict, and \citet{spek2025duality} show that $\cV$ coincides with the union of all such RKHS as $\rho$ ranges over the set of probability measures.
% \kp{
% The representer thorem in this case should show that the optimal solution can be written as
% \[
% u(x) = \sum_{i}\sum_{k} d_{i,k} \int_{\Omega} \varphi(x;\omega) \widetilde{\varphi}_{k,i}(\omega) \de \rho(\omega),
% \]
% see~\cite[``kernel vector field'']{batlle2025error})
% }

\subsubsection{Random feature models}
Random feature models offer a simpler alternative to neural network training. Instead of being trained, the internal weights in these models are 
sampled from a given probability distribution $\rho$, typically uniform or Gaussian, defined over the parameter space $\Omega$.
In a random feature model, one seeks a function expressed as
\[ 
u(x) = \sum_{n=1}^N c_n \varphi(x; \omega_n) , \quad  \{ \omega_n\} \sim \rho.
\]
Here, the inner weights $\{ \omega_n\}$ are sampled from $\rho$, and only the outer weights $\{ c_n\}$ are solved. 
Consequently, \Cref{eq:empirical_l2_min} becomes
\[
\min_{\{c_n \}_{n=1}^N} \hat{L}_{K_1, K_2}(u_{\{ c_n, w_n\}}) + \frac{\alpha N}{2} \sum_{n=1}^N |c_n|^2 . 
\]
This essentially corresponds to applying the kernel ridge regression method to solve a PDE, where the kernel is given by
\[
\hat{k} (x, x') = \frac{1}{N} \sum_{n=1}^N \varphi(x; \omega_n) \varphi(x';\omega_n)
\]
that approximates $k$ in \Cref{eq:kernel}. 
For a linear PDE problem, this leads to solving a linear system, which avoids the neural network training cost.
However, the random feature model does not adapt to the solution landscape and may require a large number of features. 
This can lead to prohibitive computational costs, especially in high-dimensional settings. For related discussions, see~\cite{RahimiRecht:08,zhang2023transnet,bach2023relationship,pieper2024nonuniform}. 
% \kp{added more references}

\subsubsection{The Gaussian process method} 
There has been a recent surge in using the Gaussian process (GP) method for solving PDEs \citep{batlle2025error,chen2021solving,chen2025sparse}.
The GP method proposed by \cite{chen2021solving} employs a Gaussian kernel with a fixed shape parameter and randomly chosen collocation points. Solving~\eqref{eq:empirical_RKHS_min} in the limiting case \(\alpha \to 0\) yields a variant of this GP method with the covariance kernel given by~\eqref{eq:kernel}, where the solution can be interpreted as a maximum a posteriori estimate under a GP prior, as shown in~\cite{chen2021solving}.
% \kp{Solving~\eqref{eq:empirical_RKHS_min} in the limiting case \(\alpha \to 0\) yields a variant of this GP method 
% with the covariance kernel given by~\eqref{eq:kernel}, since collocation conditions are combined with a regularization term related to a maximum likelihood estimate with respect to a GP prior, as in~\cite{chen2021solving}.}

However, the kernel~\eqref{eq:kernel} is only indirectly related to the feature function, and also depends on the density \(\de\rho\), we cannot directly compare the kernel used for the construction of the feature function and the covariance kernel resulting from this method. However, if we consider a density \(\de \rho\) that is concentrated on a single RBF shape parameter, but otherwise uniform over the kernel centers, we have a direct relation:
Consider the case where $\omega = y \in B_R \subset \R^d$ is in a ball of large radius \(R \gg 0\), $\sigma>0$ is fixed, and $\rho$ is a uniform probability measure on $B_R$. Then \Cref{eq:kernel} becomes
\[
\begin{split}
  k(x,x')
  &= \frac{\sigma^{2s}}{|B_R|} \int_{B_R} \sigma^{-2d} \hat{\varphi}((x-y)/\sigma) \hat{\varphi}((x'-y)/\sigma) \de y \\
  &\approx \frac{\sigma^{2s}}{|B_R|}  \int_{\R^d} \sigma^{-2d} \hat{\varphi}((x-y)/\sigma) \hat{\varphi}((x'-y)/\sigma) \de y  \\
  &=  \frac{\sigma^{2s}}{(\sqrt{2}\sigma)^d |B_R|}\hat{\varphi}\left(\frac{x-x'}{\sqrt{2}\sigma}\right)
\end{split}
\]
where $\hat{\varphi}$ denotes the standard Gaussian function in $\R^d$.  
This yields a Gaussian kernel with fixed scale parameter \(\sqrt{2}\sigma\).

\subsubsection{PINNs and other machine learning method for PDEs}

Our method is closely related to PINNs \citep{raissi2019physics,Shin_2023,LUO:2024,bonito2025convergenceerrorcontrolconsistent}, which also solve PDEs by minimizing a residual loss over a neural network ansatz. While these works provide convergence results for linear PDEs, they typically rely on large, fixed \revision{size} neural networks without incorporating sparsity or adaptivity, which are central to our approach.
Adaptive approaches are developed by \cite{cai2022self,liu2022adaptive} for linear PDEs using residual-based neuron enhancement, while our work differs by focusing on nonlinear PDEs and promoting sparsity through RKBS-based regularization.
% \zs{While convergence guarantees can be established for linear PDEs, these results typically rely on deep and over-parameterized networks, as is commonly used in numerical implementations. Also, generalization bound can be proven for two-layer neural networks coupled with Barron-type function spaces and path norm regularization.}
Radial basis function network based approaches related to PINNs have been developed by~\cite{ramabathiran2021spinn,wang2023solving, bai2023physics}.
  In particular, similar to our framework,~\cite{wang2023solving} propose a shallow RBF network trained with a sparsity-promoting $\ell_{1}$ regularized residual loss to obtain compact representation and better interpretability.
  However, a scaling of the feature functions in~\eqref{eq:discrete_network} that allows for stable computation of second derivatives for arbitrarily small length scale \(\sigma > 0\) has not been employed.
  This is a prerequisite for the analysis in terms of RKBS provided in our work, which enables theoretical guarantees in terms of generalization (error estimates), and adaptive training methods that do not require a large initial choice of the network width.

A key difference of our method compared to other approaches using (regularized) RBF networks is the optimization strategy. Instead of updating all parameters using gradient-based method, we employ a greedy, boosting-like approach that adaptively selects neurons while maintaining a simple network structure. To further enhance convergence and promote sparsity, we use a second-order optimization scheme. This helps avoid redundant neurons, which are often seen in first-order methods.
Although our theory is developed for Gaussian RBF, the approach still applies when we replace RBF with other activation functions commonly used in PINNs. 
Other related approaches, such as variational or weak formulation of loss functions \citep{yu2018deep,zang2020weak}, may be explored in the future.

It has been widely observed that the performance of PINNs is highly sensitive to the choice of penalty parameter $\lambda$ for enforcing boundary conditions~\citep{wang2021understanding, wang2022and}. As our method employs a similar loss function, it inherits this sensitivity as well. This motivates alternative treatment of boundary conditions, which will be discussed later in \Cref{subsec:bnd_treat} and \ref{subsec:num_bnd_treat}.
Related to this, a linear elliptic problem is studied by \citet{bonito2025convergenceerrorcontrolconsistent} for a class of collocation methods including PINNs formulations. Using the theory of optimal recovery and regularity theory in Besov scales, an improved PINNs loss is provided that allows for better theoretical guarantees when minimizing the loss over a (neural network) model class that approximates functions in Besov spaces. In our work, we directly regularize in a variation norm, which is related to a Besov norm, see Section~\ref{subsec:neural_net}, and we make similar observations on the practical and theoretical deficiency of the standard PINNs loss.

\revision{
\subsubsection{Adaptive methods and greedy algorithms}
\label{subsec:adaptive_methods}

Our method is closely related to a growing class of approaches that revisit shallow architectures through adaptive basis construction, including boosting-type procedures~\citep{friedman2001greedy,bengio2005convex}, Frank–Wolfe variants~\citep{bach:2017,LinearGCG:2024}, and greedy algorithms~\citep{needell2009cosamp,klusowski2018approximation,siegel2023greedy}. In these approaches, the network is constructed incrementally by repeatedly solving a nonconvex subproblem over the inner weight set to select and insert new units.
Our method distinguishes itself in several important aspects. First, in contrast to pure greedy type methods we include an $\ell_1$ sparsity-promoting regularization, which allows for pruning of unneeded features, promotes compact representations and mitigates overfitting; see \Cref{sec:experiments}.
In general, for convex loss functions, classical greedy or Franke--Wolfe style methods converge at a slow sub-linear rate, where the optimization error is proportional to one over the number of insertions.
Better resolution of the convex subproblems arising in those strategies results in provably linear convergence~\citep{flinth2019linear,PieperWalter:2021,LinearGCG:2024}. For this reason, we incorporate a semismooth Newton weight optimization strategy~\citep{Ulbrich:2011,ouyang2025trust} with inexact Gauss-Newton Hessian into the algorithm, enabling faster convergence once an appropriate parameter set has been identified.
Second, we combine the inner and outer weights in a combined update, similar to over-parametrized gradient based training methods~\citep{chizat2018global,chizat2022sparse} to mitigate the weight clustering observed in accelerated Frank--Wolfe methods that do not move the inner weights. However, in contrast to those methods the induced sparsity throughout the iterations enables joint optimization of both inner and outer weights with inexact second order updates.
Finally, rather than relying on an exact greedy selection step for the parameter insertion, which typically entails an expensive global optimization, we employ a relaxed randomized and computationally tractable strategy for updating the inner parameters.
We outline the proposed algorithm in \Cref{sec:algorithm}, but consider a detailed convergence analysis out of scope for the current work.
}

\subsection{Notation, assumptions and examples}
\label{subsec:1.3}
For the PDE problem defined by \eqref{eq:main} and \eqref{eq:PDEform}, we impose the following assumption throughout this paper.
\begin{assumption}
\label{assu:caratheodory}
Assume that the functions $E$ and $B$ satisfy the Carath{\'e}odory conditions:
\begin{align*}
(u,g,H) &\mapsto E(x, u, g, H) &&\text{is continuous for almost all } x \in D, \\
(u,g) &\mapsto B(x, u, g) &&\text{is continuous for almost all } x \in \partial D, \\
x &\mapsto E(x, u, g, H) &&\text{is bounded and $\nu_D$-measurable for all } (u,g,H) \in \R\times \R^{d} \times \R^{d^2}, \\
x &\mapsto B(x, u, g) &&\text{is bounded and $\nu_{\partial D}$-measurable for all } (u,g) \in \R\times \R^{d}.
\end{align*}
\end{assumption}

The assumptions stated above are sufficient to ensure the existence of minimizers for the continuous optimization problems, which will be established in \Cref{subsec:neural_net}.
We note, however, that these assumptions do not imply the existence of solutions to the PDE \eqref{eq:main}, which has to be argued for separately.
To analyze the convergence of our method, we require the PDEs in consideration to be well-posed.
To formalize this, let $\cU_o$, $\cU$, $\cF_o$ $\cF$ be Banach spaces in which smooth functions form a dense subset, with continuous embeddings $\cV\hookrightarrow\cU_o\hookrightarrow\cU$ and $\cF_o\hookrightarrow\cF$.  We define the combined PDE and boundary operator $\cR[u] := (\cE[u], \cB[u])$ and regard $\cR$ as a mapping from $\cU$ to $\cF$, with its restriction acting from $\cU_o$ to $\cF_o$:
\[
\cR: \cU \to \cF, \quad \cR|_{\cU_o} : \cU_o \to \cF_o.
\]
We think of $\cU_o$ and  $\cU$ as Sobolev spaces defined over the domain $D$ where the solution is sought.
Typically, $\cU$ represents the minimal regularity required for a weak solution, such as $ H^1(D)$ for second-order elliptic problems.  
The more regular space $\cU_o$, such as $H^2(D)$, may result from classical regularity theory under more regular data or boundary conditions.
The spaces $\cF_o$ and $\cF$ are given as direct sums of spaces on $D$ and $\partial D$. For example, $\cF$ may equal $ H^{-1}(D) \times H^{-1/2}(\partial D) $, combining the duals of $H^1(D)$ and its trace space. Similarly, $\cF_o$ may be a more regular variant, such as $L^2(D) \times L^2(\partial D)$, into which $\cR_{\cU_o}$ maps. 
We now reformulate  \eqref{eq:main} as the problem of finding $u \in \cU_o$ such that  
\beq
\label{eq:combinedPDE}
\cR [u] = 0 \quad \text{in } \cF_o. 
\eeq
We assume $\cF_o$ is chosen so that the zero element corresponds to functions that vanish almost everywhere on $D$.  We state the assumptions on well-posedness as follows.

\begin{assumption}[Well-posedness of the PDE]
\label{assu:wellposedness}
We assume the following:
\begin{enumerate}
\item Existence and uniqueness. There exists a unique solution $u \in \cU_o $ to the problem \eqref{eq:combinedPDE};
\item Continuity and boundedness. The maps $\cR$ and  $\cR_{\cU_o}$ are continuous around $u$ and satisfy a growth bound; that is, there exists $C>0$ such that
\begin{align*}
& \| \cR[v] \|_{\cF} = \|  \cR[u] - \cR[v] \|_{\cF} \leq C \| u - v\|_{\cU} \text{ and } \| \cR[v] \|_{\cF} \leq C \max(1, \| v \|_{\cU})  \quad \forall v \in \cU\\
& \| \cR[v] \|_{\cF_o} = \|  \cR[u] - \cR[v] \|_{\cF_o} \leq C \| u - v\|_{\cU_o} \text{ and } \| \cR[v] \|_{\cF_o} \leq C \max(1 , \| v \|_{\cU_o})  \quad \forall v \in \cU_o.
\end{align*}
 \item Stability of the solution. There exists $C>0$ such that
 \[
 \|  u - v\|_{\cU} \leq C  \|  \cR[u] - \cR[v] \|_{\cF} =  C  \|  \cR[v] \|_{\cF}  \quad \forall v\in \cU.
 \]
\end{enumerate}
\end{assumption}

Such well-posedness properties are well established in the classical theory for several important classes of PDEs; %, including linear elliptic and parabolic equations, semilinear equations with Lipschitz nonlinearities, and scalar conservation laws with entropy solutions
see, for example, \cite{brezis2011functional,caffarelli1995fully,dafermos2005hyperbolic,evans2022partial,gilbarg1977elliptic}. In general, we expect that $\cV \subsetneq  \cU_o $. However, a key advantage is that $\cV$ is dense in $\cU_o$. This property follows from the fact that the Gaussian kernel is a ``universal kernel'', meaning that finite linear combinations of its shifts can approximate any continuous function on a compact set arbitrarily well \citep{micchelli2006universal}. 
\Cref{assu:wellposedness} together with the density property allows our method to approximate solutions to \eqref{eq:main}. The corresponding analysis is presented in \Cref{subsec:convergence}.

Since many PDEs only depend on a subset of linear partial differential operators, we also introduce two linear operators
\begin{equation}
\label{eq:linearoperators}
\begin{aligned}
[\cL_E u](x) &= [c_i(x)u(x) + b_i(x)\cdot \nabla u(x) + \tr [A_i(x) \nabla^2 u(x)]_{i=1,\ldots,N_E} \\
[\cL_B u](x) &= [d_i(x)u(x) + e_i(x)\cdot \nabla u(x)]_{i=1,\ldots,N_B}
\end{aligned}
\end{equation}
with the property that
\begin{equation}
\label{eq:PDEoperatorform}
\cE[u] (x) = 
\hat{E}(x,[\cL_Eu](x)),\quad
\cB[u] (x) = 
\hat{B}(x,[\cL_Bu](x))
\end{equation}
for some nonlinear functions \(\hat{E} \colon \overline{D} \times \R^{N_E} \to \R\) and
\(\hat{B} \colon \partial D \times \R^{N_B} \to \R\).
This decomposition can be used to simplify the notation, improve the theoretical estimates below, and speed up some algorithms. 
However, in general, the linear operators could be set to simply evaluate all partial derivatives together with \(\hat{E} = E\) and \(\hat{B} = B\).
To formulate the first order optimality conditions of the optimization problems in \Cref{subsec:optimality}, we impose the following assumptions. 

\begin{assumption}
\label{assu:firstorderdiff}
\revision{Let \Cref{assu:caratheodory} hold for the residual functions \(\hat{E}\) and \(\hat{B}\) and let the coefficient functions $\{ c_i(x), b_i(x), A_i(x)\}_{i=1,\ldots,N_E}, \{ d_i(x),e_i(x) \}_{i=1,\ldots,N_B}$ be measurable and uniformly bounded.} 
In addition, the functions \( (x, l_E) \mapsto \hat{E}(x, l_E) \) and \( (x, l_B) \mapsto \hat{B}(x, l_B) \) are differentiable in their second arguments, with gradients uniformly bounded in \( x \in \overline{D} \) and \( x \in \partial D \), respectively. 
\end{assumption}

\noindent{\bf Example}
Using the notation introduced above, the semilinear Poisson equation with Dirichlet boundary conditions
\begin{equation*}
- \upDelta u + u^3 = f_{\mathcal{E}} \quad \text{in } D, \qquad
u = f_{\mathcal{B}} \quad \text{on } \partial D.
\end{equation*}
can be expressed as
\beq
\label{eq:semilinear}
\cE[u] = -\upDelta u + u^3 - f_{\cE},
\quad \cB[u] = u - f_{\cB},
\eeq
where we define \(E(x,u,g,H) = -\tr H + u^3 - f_{\cE}(x)\) and \(B(x,u,g) = u -  f_{\cB}(x)\). For the second description, we define the linear operators
\begin{align*}
\cL_E[u] &= [u, \upDelta u], \\
\cL_B[u] &= [u],
\end{align*}
together with \(\hat{E}(x,l_1,l_2) = -l_2 + l_1^3 - f_{\cE}(x)\) and \(\hat{B}(x,l_1) = l_1 - f_{\cB}(x)\).\\

%\zs{I want to change this to Dirichlet boundary  since that's what I'll use for the first proof-of-idea numerical example (Maybe combine example 1 and 3). The "mask" technique is working crazily well but that can only be applied to homogeneous Dirichlet BVP.}
%\kp{It is OK to change to Dirichlet, I have done so. The mask technique should apply to anay Dirichlet data \(f\) that can be extended to the whole domain \(f \colon \overline{D} \to \R\), by setting \(u(x) = f(x) + \psi(x)\cN(x)\), where \(\psi\) is the homogeneous Dirichlet mask function.}

% \noindent{\bf Example 3.}
% Space-time Burger's: we define the space time variable \(x = (t, \hat{x}) \in (0,T)\times(0,1) = D\) and
% \[
% \partial_t u + u \partial_{\hat{x}} u - \nu \partial^2_{\hat{x}} u = 0 \text{ on } D,
% \quad
% u = 0 \text{ on } (0,T)\times \{0,1\}
% \quad
% u = u_0 \text{ on } \{0\}\times (0,1).
% \]
% Here, we define 
% we define \(E(x,u,g,H) = g_1 + u g_2 - \nu h_{2,2}\) and \(B(x,u,g) = \mathbf{1}_{\{0,1\}}(\hat{x}) u + \mathbf{1}_{\{0\}}(t) (u - u_0(\hat{x}))\).
% For the linear operators, we can choose
% \begin{align*}
% \cL_E[u] &= [u, \partial_{\hat{x}} u, \partial_t u - \nu\partial_{\hat{x}}^2 u], \\
% \cL_B[u] &= [u],
% \end{align*}
% together with \(\hat{E}(x,l_1,l_2,l_3) = l_3 + l_1 l_2\) and \(\hat{B}(x,l_1) = \mathbf{1}_{\{0,1\}}(\hat{x}) l_1 + \mathbf{1}_{\{0\}}(t) (l_1 - u_0(\hat{x}))\).

\subsubsection{Notation summary} 
Throughout the paper, $\R$ denotes the set of real numbers, $\R_+$ the set of positive real numbers, $\N$ the set of positive integers, and $\N_0$ the set of non-negative integers.
For $m\in\N$,  $\R^m$ denotes the $m$-dimensional Euclidean space. $\Omega$ denotes the parameter domain for the integral neural network, while $D\subset \R^d$ is a bounded open domain where the PDE is defined.

We summarize the main notation below.
\begin{itemize}[left=0pt, label={\textbf{--}}]
    
    \item Function spaces on a set \( A \subset \mathbb{R}^m \)
    
    \begin{itemize}[left=1em, label={--}]
        \item \(C(A)\): Continuous functions on \(A\).
        \item \(C(\overline{A})\): Continuous functions on the closure \(\overline{A}\), with the supremum norm.
        \item \(C_0(A)\): Functions in \(C(\overline{A})\) that vanish on \(\overline{A} \setminus A\). 
        %Continuous functions vanishing at \(\overline{\Omega} \setminus \Omega\), i.e., for any \(\varepsilon > 0\), there exists compact \(K \Subset \Omega\) such that \(|f(x)| < \varepsilon\) for all \(x \notin K\). Equivalently, \(C_0(\Omega)\) consists of 
        \item \(C^k(\overline{A})\): Functions on \(\overline{A}\) with continuous derivatives up to order \(k\).
        \item \(C^{k,\gamma}(A)\): H\"{o}lder space of functions with \(k\) continuous derivatives whose \(k\)th derivatives are H\"{o}lder continuous with exponent \(\gamma \in (0,1]\).
        \item $H^s(A)=W^{s,2}(A)$: Sobolev space  with smoothness index $s$ and integrability index $2$. 
        \item $B^s_{p,q}(A)$: Besov space with smoothness index $s$, integrability $p$, and summability $q$.
    \end{itemize}
    
          \item \( M(\Omega) \): Space of signed Radon measures on \( \Omega \), equipped with the total variation norm. 
    Unless otherwise stated, \(M(\Omega)\) is endowed with the weak-$\ast$ topology induced by its duality with \(C_0(\Omega)\). 
    
      \item \( \mathcal{N} \mu(x) = u_\mu(x) \): Integral neural network defined by \( u_\mu(x) = \int_\Omega \varphi(x; \omega) \, d\mu(\omega)\), where $\varphi(x; \omega)$ is a Gaussian RBF with parameter $\omega\in \Omega$. 

   % \item \( \varphi(x; \omega) \): Gaussian radial basis function with scale parameter \( \sigma \), given by
    %\[
    %\varphi(x; \omega) = \frac{\sigma^s}{( \sqrt{2\pi} \sigma )^d} \exp\left( -\frac{\|x - y\|_2^2}{2 \sigma^2} \right), \quad \omega = (y, \sigma).
    %\]
    \item For two norms spaces $X$ and $Y$, $X\hookrightarrow Y$ denote a continuous embedding of $X$ into Y, i.e., $\| x\|_{Y} \leq C \| x\|_X$ for all $x\in X$. 
        \item \( \cF_o\hookrightarrow\cF,  \cU_o\hookrightarrow\cU \): Banach spaces of admissible input data and PDE solutions, respectively, each containing smooth functions as a dense subset. 
        \item \( \cR\): The combined PDE and boundary operator $(\cE, \cB) : \cU \to \cF$, with restriction $\cR|_{\cU_o}: \cU_o \to \cF_o$. 
  \item \( \mathcal{V}(D) \): Reproducing Kernel Banach Space (RKBS) of functions represented as \( \mathcal{N} \mu \) with \( \mu \in M(\Omega) \). \( \mathcal{V}(D)\) is a dense subset of  \(\mathcal{U}_o\) and $\cU$.
    % \item \( \mathcal{H}(D) \): RKHS associated with a positive-definite kernel on \( D \).
    \item \( \partial^\beta \): Multi-index partial derivative with $ \beta = (\beta_1, \dots, \beta_d) \in \N_0^d$.   
   % For \( \beta = (\beta_1, \dots, \beta_d) \in \mathbb{N}^d \),
    %\[
    %\partial^\beta = \frac{\partial^{|\beta|}}{\partial x_1^{\beta_1} \cdots \partial x_d^{\beta_d}}, \quad |\beta| = \sum_{i=1}^d \beta_i.
    %\]
    \end{itemize}

We also introduce some terminology used throughout the paper. By analogy with collocation methods, we refer to the set of quadrature points $\{x_{1,k}\} \cup \{x_{2,k}\} $, which define the empirical loss, as \emph{collocation points}, and denote their total number by $ K $. Note, however, that the PDE is not enforced to be exactly satisfied at these points. The functions $ \{ \varphi(x; \omega_n) \}_{n=1}^N $ are called \emph{feature functions}, and the associated parameters \( \{ \omega_n \}_{n=1}^N \) are referred to as \emph{inner weights} or \emph{kernel nodes}, where $N$ is the total number of feature functions.

\begin{comment}
\kp{
To impose mixed boundary conditions, we can relax the requirement that \(B\) is continuous in \(x\).
The same is valid for \(E\), which may allow for discontinuous diffusion and advection coefficients.
However in such settings the strong from of the equation may be contradictory at points of discontinuity, and we may have to think carefully about the strong form.
Mixed BC:
\[
\mathbf{1}_{\Gamma_D} (u - g_D) +
\mathbf{1}_{\Gamma_N} \left( \frac{\partial u}{\partial n} - g_N \right) +
\mathbf{1}_{\Gamma_R} \left( \frac{\partial u}{\partial n} + \gamma u - g_R \right),
\]
where $\Gamma_D, \Gamma_N, \Gamma_R$ are disjoint subsets of $\partial D$ with $\Gamma_D\cup\Gamma_N\cup\Gamma_R = \partial D$.  $g_D$, $g_N$, $g_R$ are continuous functions.}
\end{comment}

\subsection{Outline of the paper}
\label{subsec:1.4}
The rest of the paper is organized as follows. \Cref{sec:theory} develops the theoretical foundation of our approach. We introduce the integral neural network representation and its associated function space, and establish an existence result for the sparse minimization problem. We also provide a convergence analysis for the method applied to solving PDEs, derive a representer theorem ensuring finite representation, and present the optimality conditions and dual variables that underlie the design of our optimization algorithm.
\Cref{sec:algorithm} turns to the algorithmic framework. We introduce a three-phase algorithm, which includes a gradient boosting strategy for inserting kernel nodes, a semi-smooth Gauss–Newton method for optimizing parameters, and a node deletion step to maintain a compact network size. Some additional implementation components are also discussed.
\Cref{sec:experiments} presents numerical experiments that validate the effectiveness of the method. Finally, \Cref{sec:conclusion} concludes with a discussion of future directions.

\section{Theoretical framework}
\label{sec:theory}
\revision{
We begin with a brief overview of the main theoretical results in this section, together with high-level proof sketches.

\begin{itemize}[leftmargin=1.2em]
    \item \textbf{Existence (\Cref{thm:existence}, \Cref{subsec:neural_net}).}
    Under \Cref{assu:caratheodory} and the continuity of the feature map established in \Cref{prop:smoothness_feature_function} and \Cref{lem:continuity}, the continuous sparse minimization problem~\eqref{eq:sparse_min} and its empirical counterpart~\eqref{eq:empirical_sparse_min} admit global minimizers. The proof follows the direct method in the calculus of variations, using  weak-\(\ast\) compactness in \(M(\Omega)\) and weak-\(\ast\) lower semicontinuity of the objective.

    \item \textbf{Convergence and error estimates (\Cref{thm:convergence} and \Cref{thm:error_bound}, \Cref{subsec:convergence}).}
    Assuming the quadrature measures converge weakly and remain uniformly bounded (\Cref{assu:measureconvergence}), and the data-space embedding/interpolation condition (\Cref{assu:embedding}), empirical minimizers \(u_{K,\alpha}\) (as \(K\to\infty\)) admit limit points \(u_\alpha\) that satisfy an explicit residual-based error estimate in \(\cU\). 
    %Moreover, \(u_{K,\alpha}\to u_\alpha\) (up to subsequences) in \(C^1(\overline D)\). 
    We further show how to choose \(K\) and \(\alpha\) jointly to obtain an explicit approximation bound for \(u_{K,\alpha}\) relative to the PDE solution.

    \item \textbf{Finite representation/Representer theorem  (\Cref{thm:representer}, \Cref{subsec:representertheorem}).}
    For the empirical sparse problem~\eqref{eq:empirical_sparse_min}, there exists an optimal solution supported on finitely many atoms. The argument reduces the nonlinear objective to a constrained minimum-norm problem in \(M(\Omega)\) and then applies Carath\'eodory-type arguments to obtain a finite representation of solutions.
\end{itemize}

}

\subsection{Integral neural networks and the associated function space}
\label{subsec:neural_net}
We begin by recalling and formalizing the definition of the integral neural network introduced in the introduction. Specifically, we consider functions represented in the form
\beq
u_\mu (x)= \cN \mu(x) := \int_\Omega \varphi(x;\omega)\de \mu(\omega),
\eeq
where $\varphi$ is a Gaussian RBF with variable bandwidth \(\sigma\)
\beq
    \varphi(x;\omega) = \frac{\sigma^s}{\left(\sqrt{2\pi}\sigma\right)^d} 
    \exp\left(- \frac{\norm{x - y}^2_2}{2 \sigma^2}\right)
    \quad
    \text{where } \omega = (y,\sigma) \in \R^{d+1} .
\eeq
Introduce the standard Gaussian function on $\R^d$
\[
\hat{\varphi}(x') = \frac{1}{\left(\sqrt{2\pi}\right)^d} 
    \exp\left(- \frac{\norm{x'}^2_2}{2}\right).
\]
Then the scaled Gaussian can be written as
\beq
\label{eq:scaledGaussian}
\varphi(x; \omega) = \sigma^{s-d} \hat{\varphi}\left(\frac{x-y}{\sigma} \right).
\eeq
To define the parameter space $\Omega$, fix a maximum bandwidth \(\sigma_{\max}>0\) and let $D_{1} :=  \overline{D + B_{\sigma_{\max}}(0)}$. We then set
\[
\Omega =  D_1 \times (0, \sigma_{\max}].
\]

We now define the function space associated with the integral neural networks. 
Specifically, we consider the set of functions on $D$ that can be written as $\cN \mu$ for some signed Radon measure $\mu$.
This leads to a reproducing kernel Banach space (RKBS) \citep{bartolucci2023understanding,lin2022reproducing,zhang2009reproducing}, defined by
\beq
\cV(D) = \left\{ \cN \mu: D \to \R \;|\; \mu \in M(\Omega) \right\}
\eeq
and equipped with the norm
\beq
\norm{u}_{\cV(D)}
= \inf \{\norm{\mu}_{M(\Omega)} \;|\; \mu \in M(\Omega)\colon \cN \mu = u \text{ on } D \}. 
\eeq
Here \(M(\Omega)\) is the space of signed Radon measures on $\Omega$ equipped with the total variation norm
\[
\| \mu \|_{M(\Omega)} = \sup \left\{  \int_\Omega f \de \mu \; \colon f \in C_0(\Omega),\ \|f\|_\infty \leq 1 \right\}.
\]
\(M(\Omega)\) is also identified with the the dual space of $C_0(\Omega)$. Unless otherwise stated, \(M(\Omega)\) is endowed with the weak-$\ast$ topology induced by its duality with \(C_0(\Omega)\).
%\[
%C_0(\Omega) = \left\{ f \in C(\Omega) \mid \forall \epsilon > 0, \exists K \Subset \Omega \text{ such that } |f(x)| < \epsilon \text{ for all } x \notin K \right\}. 
%\]
%Equivalently, $C_0(\Omega)$ is the subspace  of the space of bounded continuous functions endowed with the sup norm
%\[
%C(\overline{\Omega})
%= \{ v \in C(\Omega) \;|\; \norm{v}_{C(\overline{\Omega})} = \sup_{x\in \overline{\Omega}} \abs{v(x)} < \infty \}
%\]
%that are zero on $\partial \Omega = \overline{\Omega} \setminus \Omega = D_1 \times \{0\}$.

\begin{proposition}
\label{prop:smoothness_feature_function}
For \(s \geq d + k + \gamma\), with \(k\in \N\cup\{0\}\) and \(\gamma \in (0,1]\), we have $\partial_x^\beta \varphi(x;\cdot) \in C_0(\Omega)$ uniformly for $x\in \overline{D}$ for and multi-index \(\beta\) with \(\abs{\beta} \leq k\) and $\varphi(\cdot;\omega) \in C^{k,\gamma}(D_1)$ uniformly for $\omega \in \Omega$.
\end{proposition}
\begin{proof}
\Cref{eq:scaledGaussian} implies
\[
\partial_x^\beta \varphi(x ; \omega) = \sigma^{s-d-\abs{\beta}} \partial_x^\beta \hat{\varphi}((x-y)/\sigma).
\]
Under the assumptions stated above, with \(\abs{\beta}\leq k\), this function can be shown to be uniformly H{\"o}lder continuous with index \(\gamma = \min\{1,s-d-k\}\). Moreover, for \(\omega = (y,\sigma)\) with \(\sigma \to 0\), it converges to zero at the rate \(\sigma^{s-d-\abs{\beta}}\).
\end{proof}
\begin{remark}
\begin{enumerate}
    \item Concerning the previous result, we note that for the integer cases of \revision{\(s\in \{d,d+1,d+2,\ldots\}\)} we cannot set \(k = s - d\), due to the requirement \revision{\(\gamma > 0\)}. Indeed, in the case \(s=d\) we only obtain that $\varphi(x;\cdot) \in L^\infty(\Omega)$ uniformly for $x\in \overline{D}$ and $\varphi(\cdot;\omega) \in L^{\infty}(D_1)$ uniformly for $\omega \in \Omega$ \revision{but \(\varphi\) cannot be extended continuously to the boundary \(\sigma = 0\)}.
    \item \revision{
    We emphasize that \Cref{prop:smoothness_feature_function} is indispensable for subsequent analysis in this section, including the existence of the minimizers and associated error estimates. Hence, the scale-dependent weight $\sigma^{s}$ with $s \geq d + k + \gamma$, in addition to the standard Gaussian normalization, is fundamental to the theoretical framework of our method.
    }
\end{enumerate}

\end{remark}
\begin{proposition}
\label{prop:neural_network_mapping}
    For \(s \geq d + k + \gamma\) with \(k\in \N\cup\{0\}\) and \(\gamma \in (0,1]\) the linear neural network operator is continuous:
    \[
    \cN \colon M(\Omega) \to C^{k,\gamma}(D_1).
    \]
    In particular, for \(s > d+2\) the required derivatives exist and are bounded in terms of the measure
    \[
    \norm{\cN \mu}_{C(\overline{D})} + \norm{\nabla_x \cN \mu}_{C(\overline{D})} +\norm{\nabla^2_x \cN \mu}_{C(\overline{D})}
    \leq C \norm{\mu}_{M(\Omega)}.
    \]
\end{proposition}

The above proposition implies the continuous embedding $\cV(D) \hookrightarrow C^{k,\gamma}(D)$. Moreover, the RKBS $\cV$  is closely related to a classical Besov space. 
In \Cref{app:besov},  we establish the embedding $B^{s}_{1,1}(\R^d)  \hookrightarrow \cV(\R^d)$, where a detailed proof is also provided.
\revision{In the context of wavelet analysis, \cite{meyer1992wavelets} gives a Besov space characterization of the Gaussian ``hump algebra'', which is closely related to the RKBS for \(s=d\) considered here. In our subsequent work~\cite{shao2026sparse}, we establish a precise Besov space characterization for a broad class of integral RKBSs related to scaled RBF kernels including the Gaussian considered here.
}
\revision{These refined characterizations together with the embedding properties of Besov spaces also show that the in the case \(s=d+k\) for integer \(k\), we still obtain continuously differentiable functions in the RKBS \(\cV\). However, for the purposes of generalization (in Section~\ref{subsec:convergence}), we need to quantify the degree of continuity, and consider \(\gamma > 0\).
}

We assume $s\geq d+2 + \gamma$ for $\gamma\in (0,1]$ for the rest of this paper. Therefore \revision{$\cV(D) \hookrightarrow C^{2,\gamma}(\overline{D})$.}
By \Cref{prop:smoothness_feature_function}, we observe that if the sequence of measures $\mu_k$ converges to $\mu$ in the weak-$\ast$ sense, then for any $x\in \overline{D}$, 
\beq
\label{eq:continuity}
\partial^\beta (\cN \mu_k) (x) = \int_{\Omega} \partial^\beta_x \varphi(x ; \omega) \de \mu_k(\omega) \to \int_{\Omega} \partial^\beta_x \varphi(x ; \omega) \de \mu(\omega) = \partial^\beta (\cN \mu) (x) ,
\eeq
for a multi-index $\beta$ with $|\beta|\leq 2$. 
This result leads to the following lemma, which is useful in establishing the existence theory for the minimization problems  \eqref{eq:sparse_min} and \eqref{eq:empirical_sparse_min}.
\begin{lemma}
\label{lem:continuity}
Assume $\mu_k  \stackrel{\ast}{\rightharpoonup} \mu $ in $M(\Omega)$, then we have 
\begin{align*}
&\cE[\cN \mu_k](x) \to \cE[\cN \mu](x)  \quad \forall x \in D, \\
& \cB[\cN \mu_k](x) \to \cB[\cN \mu](x)  \quad \forall x \in \partial D.
\end{align*}
\end{lemma}
\begin{proof}
    The result follows directly from~\eqref{eq:continuity} and \Cref{assu:caratheodory}.
 \end{proof}

\begin{theorem}[Existence of minimizers]
\label{thm:existence}
Suppose \Cref{assu:caratheodory} holds. 
The sparse minimization problem \eqref{eq:sparse_min} admits at least one global minimizer. Similarly, 
\eqref{eq:empirical_sparse_min} admits at least one global minimizer. 
\end{theorem}
\begin{proof}
We use direct methods in the calculus of variations to establish the results. Specifically, we prove the existence of minimizers for \eqref{eq:sparse_min}, with the result for \eqref{eq:empirical_sparse_min} following by a similar argument.  Denote $J(\mu) = L(u_\mu) + \alpha \norm{\mu}_{M(\Omega)}$.  Let $\mu^{(k)}\in M(\Omega)$ be a minimizing sequence, i.e., $J(\mu^{(k)}) \to \inf_{\mu\in M(\Omega)} J(\mu)$. Then since $\| \mu^{(k)}\|_{M(\Omega)}$ is uniformly bounded, there is a subsequence (without relabelling) that converges to $\bar{\mu}\in M(\Omega)$ in the weak-$\ast$ sense.
Now we want to show $J$ is weak-$\ast$ lower semicontinuous such that $\overline{\mu}$ is a minimizer of the problem. First of all, the lower semicontintuity of $\| \cdot \|_{M(\Omega)}$ under weak-$\ast$ convergence is obvious. Second,  we observe that for any multi-index $\beta \in \N^d$ with $|\beta |\leq 2$, $\partial^\beta_x\varphi(x; \cdot) \in C_0(\Omega)$ for a given $x\in \overline{D}$; see Proposition~\ref{prop:smoothness_feature_function}. Therefore, by \Cref{lem:continuity} we have  $\cE[u_{\mu^{(k)}}](x)\to \cE[u_{\bar{\mu}}](x)$ $\forall x\in D$, and $\cB[u_{\mu^{(k)}}](x)\to \cB[u_{\bar{\mu}}](x)$ $\forall x\in \partial D$. Finally, by Fatou's lemma
\[
L(u_{\bar{\mu}})\leq \liminf_{k}
\int_D \frac{1}{2}(\cE[u_{\mu^{(k)}}](x))^2 d\nu_D +\frac{\lambda}{2}   \int_{\partial D} (\cB[u_{\mu^{(k)}}](x)))^2 d\nu_{\partial D} = \liminf_{k} L(u_{\mu^{(k)}}).
\]
Therefore, $J$ is weak-$\ast$ lower semicontinuous and minimizers exist.
\end{proof}

\subsection{Functional analytical setup and assumptions}
\label{subsec:functional_setup}
Recall the conditions in \Cref{assu:wellposedness} on the well-posedness of the PDEs. 
We first present an example of a PDE for which these conditions are satisfied.
Consider the semilinear Poisson equation in \eqref{eq:semilinear}. For simplicity and without loss of generality, we consider the modified formulation
\[
\cE[u] = -\upDelta u + \phi_M(u) - f_{\cE},
\quad \cB[u] = u - f_{\cB},
\]
where $\phi_M \in C^\infty(\R)$ is monotone and equals  $u^3$ if $|u^3|<M$ and $\phi_M(u) = \text{sign}(u)\cdot 2M $ if $|u^3|>2M$. We assume $M>0$ is a sufficiently large constant. This simplification is justified by the maximum principle, which ensures that any classical solution $u$ to \eqref{eq:semilinear} satisfies the bound
$\sup_{D} |u|  \leq 2 \sup_{D} |\bar{u}| $ 
where $\bar{u}$ is a classical solution that solves the linear problem
\[
-\Delta \bar{u} - f_{\cE} = 0 \quad \text{in } D, \quad \bar{u} - f_{\cB} = 0 \quad \text{on } \partial D.
\]
A proof of this estimate can be found, for example, in \cite[Corollary 1.5]{taylor1996partial}. 
For the modified equation, under minimal regularity assumptions, we may take $\cU = H^1_0(D) + f_\cB$ and $\cF = H^{-1}(D)\times H^{1/2}(\partial D)$. 
Indeed, it is not hard to verify that  \Cref{assu:wellposedness}(2) holds for the pair $(\cU, \cF)$ due to the Lipschitz continuity of $\phi_M$, and  \Cref{assu:wellposedness}(3) is satisfied by the monotonicity of the operator $-\Delta u + \phi_M(u)$.  
More precisely, by assuming $f_\cE \in H^{-1}(D)$ and $f_\cB \in H^{1/2}(\partial D)$, \Cref{assu:wellposedness}(2) follows from the estimates
\begin{align*}
    &\norm{-\Delta (u-v)}_{H^{-1}(D)} = \sup_{w\in H^1_0(D)} \frac{\langle -\Delta (u-v) , w\rangle}{\norm{w}_{H^1}} \leq C \norm{u-v}_{H^1(D)}  \\
    & \norm{\phi_M(u) - \phi_M(v)}_{H^{-1}(D)}\leq C\norm{\phi_M(u) - \phi_M(v)}_{L^2(D)} \leq \tilde{C} \norm{u-v}_{L^2(D)} 
\end{align*}
along with the trace theorems for functions in $H^1(D)$. \Cref{assu:wellposedness}(3) is justified by
\[
\begin{split}
& \norm{-\Delta(u-v) + \phi_M(u) - \phi_M(v)}_{H^{-1}(D)} \\
\geq & \frac{\langle-\Delta(u-v) + \phi_M(u) - \phi_M(v), u-v \rangle}{\| u-v\|_{H^1(D)}} \geq \frac{\norm{\nabla(u-v)}_{L^2(D)}^2}{\| u-v\|_{H^1(D)} } \geq  C \norm{u-v}_{H^1(D)}, 
\end{split} 
\]
where we have used monotonicity of $\phi_M$ and Poincar\'e inequality for $u-v \in H_0^1(D)$.
Now the choice of $\cU_o$ and $\cF_o$ depends on the smoothness of the domain and the data.  For example, the classical elliptic regularity for Lipschitz domains allows us to choose 
$\cU_o = H^2(D)\cap \cU$ and $\cF_o = L^2(D) \times H^{3/2}(\partial D)$.  Under sufficient additional smoothness assumptions on the domain $D$ and the data $f_\cE$ and $f_\cB$, one may also take $\cU_o = H^k(D)\cap \cU$ with $\cF_o = H^{k-2}(D) \times H^{k-1/2}(\partial D)$ for any integer $k \geq 2$. 
Recall the embedding $ B^s_{1,1} \hookrightarrow \cV $ established in \Cref{app:besov}. By the embedding properties of Besov spaces~\citep[see, e.g.,][Proposition~4.6]{triebel2006}, for any $k>s$,
\[
H^k = B^k_{2,2} \hookrightarrow B^s_{1,1} \hookrightarrow \cV.
\]
In this case, the PDE solution $u\in \cU_o$ also lies in the RKBS $\cV$.
%\xt{Here is an example where the PDE solution lies in $\cV$. I did not quote the regularity theory \cite[Theorem 5.41]{sawano2018theory} since it is established for $F^{s}_{p,q}$ with $1<p,q$ (some embedding results are needed anyway for $B^s_{1,1}$).}

To reflect the definition of the loss function \eqref{eq:loss}, we introduce the weighted product space $L_\lambda^2 = L^2(D, \nu_D) \times  L^2(\partial D, \nu_{\partial D})$ with the norm
\[
\| e \|_{L_\lambda^2}^2  :=  \frac{1}{2}\| e_E\|_{ L^2(D, \nu_D) }^2 +  \frac{\lambda}{2} \| e_B\|_{ L^2(\partial D, \nu_{\partial D}) }^2
\]
for any $e = (e_E, e_B) \in L^2(D, \nu_D) \times  L^2(\partial D, \nu_{\partial D})$. 
We make the following assumptions on $\cF$ and $\cF_o$.  
\begin{assumption}
We assume that one of the following conditions holds:
\begin{enumerate}
\label{assu:embedding}
\item $\cF_o \hookrightarrow L^2_\lambda \hookrightarrow \cF$; 
\item $\cF_o \hookrightarrow \cF   \hookrightarrow L^2_\lambda$  and $\cF$ is an interpolation space in between  $L^2_\lambda$ and $\cF_o$, i.e.,  there exists $\theta \in (0,1)$ and $C>0$ such that
\[
\| e\|_{\cF} \leq C  \| e \|_{L^2_\lambda}^\theta  \|e \|_{\cF_o}^{1-\theta}. 
\]
\end{enumerate}
\end{assumption}

\begin{remark}
We note that the above assumption is made for notational simplicity. Ideally, the function spaces on $D$ and $\partial D$ should be treated separately, and either condition (1) or (2) may hold independently on each component. This is indeed the case for the earlier example with $\cF = H^{-1}(D)\times H^{1/2}(\partial D)$. 

Moreover, in certain cases, it is also beneficial to modify the form of the loss function depending on the structure of the underlying PDE. For example, in Dirichlet boundary value problems, incorporating derivatives of the boundary term into the loss function can enhance performance. Such extensions are discussed in the subsequent sections. Related ideas on improved phisics-informed loss functions can  be found in \citep{bonito2025convergenceerrorcontrolconsistent,ainsworth2024extended}.

\end{remark}

\subsection{Convergence analysis for the neural network solution}
\label{subsec:convergence}
In this subsection, we establish the convergence of our method in a general setting.

Denote $\nu_D^{K_1}:=\sum_{k=1}^{K_1} w_{1,k} \del_{x_{1,k}}$ and $\nu_{\partial D}^{K_2} := \sum_{k=1}^{K_2} w_{2,k} \del_{x_{2,k}}$,
\revision{which are nonnegative measures by construction}. 
The following assumption guarantees that the numerical quadrature used to define the empirical loss function \eqref{eq:loss_discrete} provides a consistent approximation of the continuous loss  \eqref{eq:loss}. 

\revision{
\begin{assumption}
\label{assu:measureconvergence}
We assume that the discrete measures $\nu_D^{K_1}$ and $\nu_{\partial D}^{K_2}$ converge (in the weak-\(*\) sense) to $\nu_{D}$ and $\nu_{\partial D}$, respectively, i.e., 
\[
\int_D  f \de \nu_D^{K_1} \to \int_D  f \de \nu_D , \quad \int_{\partial D}  g \de \nu_{\partial D}^{K_2} \to \int_{\partial D}  g \de \nu_{\partial D}  
\]
as $K_1\to \infty$ and $K_2\to \infty$ for any $f \in C(\overline{D})$ and $g \in C(\partial D)$. 
\end{assumption}
}
\revision{%
We also note that weak-\(\ast\) convergence in the above implies uniform boundedness of the non-negative measures $\{\nu_D^{K_1}\}$ and $\{\nu_{\partial D}^{K_2}\}$ (by setting \(f=1\) and $g=1$).
Recall the definition of $\cL_E$, $\cL_B$, $\hat{E}$ and $\hat{B}$ in \Cref{eq:linearoperators,eq:PDEoperatorform}.
For the rest of this subsection, we further strengthen Assumptions~\ref{assu:caratheodory} and \ref{assu:firstorderdiff} as follows.
\begin{assumption}
\label{assu:uniformcontinuity}
Let \Cref{assu:firstorderdiff} hold. In addition, the functions $\{A_i(x), b_i(x), c_i(x)\}_{i=1}^{N_E}$ and $\hat{E}(x, l_E)$ are uniformly \(\gamma\)-H{\"o}lder continuous on $\overline{D}$, and $\{d_i(x), e_i(x)\}_{i=1}^{N_B}$ and $\hat{B}(x, l_B)$ are uniformly \(\gamma\)-H{\"o}lder continuous on $\partial D$. 
\end{assumption}
}
\revision{%
\begin{remark}
Assumption~\ref{assu:uniformcontinuity} simplifies the subsequent analysis by ensuring the \(\gamma\)-H{\"o}lder continuity of the PDE-residual.
The arguments below remain valid under weaker assumptions for problems with discontinuous coefficients, such as piecewise continuity for the functions on subdomains. For simplicity of exposition, we omit such discussions. 
\end{remark}
The strengthened continuity allows us to show the following useful result.
}
\revision{%
\begin{lemma}
    \label{lem:uniformconvergence}
Let $\cE[u]$ and $\cB[u]$ be defined as in \eqref{eq:PDEoperatorform}. Under \Cref{assu:uniformcontinuity}, if $\mu_k  \stackrel{\ast}{\rightharpoonup} \mu $ in $M(\Omega)$ and \(\norm{\mu_k}_{M(\Omega)} \leq M < \infty\), then \(\cE[\cN \mu_k](x)\) and \(\cB[\cN \mu_k](x)\) are uniformly \(\gamma\)-H{\"o}lder continuous and
\begin{align*}
& \cE[\cN \mu_k](x) \to \cE[\cN \mu](x) \quad \text{uniformly on } D; \\
& \cB[\cN \mu_k](x) \to \cB[\cN \mu](x) \quad \text{uniformly on } \partial D.
\end{align*}\end{lemma}
\begin{proof}
First, \Cref{lem:continuity} gives the pointwise convergence of $\cE[\cN \mu_k](x)\to \cE[\cN \mu](x)$ for $x\in D$, and $\cB[\cN \mu_k](x)\to \cB[\cN \mu](x)$ for $x\in \partial D$. 
Further, $\sup_k \| \mu_k \|_{M(\Omega)} \leq M <\infty$ and \Cref{prop:neural_network_mapping} imply that $\sup_k \| \cN \mu_k\|_{C^{2,\gamma}(\overline{D})} < \infty$.
By the definition of $\cE$ and $\cB$ in \eqref{eq:PDEoperatorform}, and \Cref{assu:uniformcontinuity}, we know that $\cE[\cN \mu_k]$ and $\cB[\cN \mu_k]$ are uniformly bounded and equicontinuous on $D$ and $\partial D$, respectively.
In fact, a straightforward estimate shows that \(\norm{\cE[\cN \mu_k]}_{C^{0,\gamma}(\overline{D})} + \norm{\cB[\cN \mu_k]}_{C^{0,\gamma}(\partial D)} \leq C(1 + \norm{\cN \mu_k}_{C^{2+\gamma}})\).
%\kp{
%Sketch of verification:
%Consider for instance
%\begin{align*}
%\abs{\cE[u](x) - \cE[u](x')} &= \abs{\hat{E}[\cL_Eu(x), x] - \hat{E}[\cL_E u(x'), x']} \\
%& \leq \abs{\hat{E}[\cL_E u(x), x] - \hat{E}[\cL_E u(x), x']} + \abs{\hat{E}[\cL_E u(x), x'] - \hat{E}[\cL_E u(x'), x']} \\
%& \leq C \abs{x - x'}^\gamma + C_{\nabla_l\hat{E}}\abs{\cL_E u(x) - \cL_E u(x')} \\
%& \leq C \abs{x - x'}^\gamma + \sum_{g \in \{a_i,b_i,c_i\}_i} C_{\nabla_l\hat{E}} \norm{g}_{L^\infty}\abs{D(u(x) - u(x'))} + \norm{g(x) - g(x')} \norm{Du}_{L^\infty} \\
%&\leq C(1 + \norm{u}_{C^{2+\gamma}}) \abs{x - x'}^\gamma),
%\end{align*}
%using the uniform Lipschitz continuity (uniform boundedness of the gradients) of \(\hat{E}\) in the first argument, and the postulated H{\"o}lder continuities (where \(D\) is either \(1, \nabla\) or \(\nabla^2\)).
%}
The uniform convergence then follows from the Arzelà–Ascoli theorem. 
\end{proof}
}

%By \Cref{assu:wellposedness} and the discussion after it, we may assume without loss of generality that for a given input function $f\in \cF$, the unique solution $u^\ast: = u[f]$ lies in $\cV$ and is represented by $\cN \mu^\ast$ for $\mu^\ast \in M(\Omega)$. 

By \Cref{assu:wellposedness}(1), we denote the unique solution to \eqref{eq:main} by $u\in \cU_o$. 
Let $u^\ast = \cN \mu^\ast \in \cV$ be an approximation of $u$ in $\cU_o$. By 
the continuity of $\cR|_{\cU_o}$ in \Cref{assu:wellposedness}(2), it follows that 
\[
\| \cR[u^\ast]  \|_{\cF_o}  = \| \cR[u] -\cR[u^\ast]  \|_{\cF_o}\leq C_1 \| u - u^\ast\|_{\cU_o}. 
\]
The right-hand side in the above can be made arbitrarily \revision{small} by density of $\cV$ in $\cU_o$. 
In the following, we derive an error estimate for the neural network solution. 

\begin{theorem}
\label{thm:convergence}
Let $u\in \cU_{o}$ denote the unique solution to \eqref{eq:main} and \revision{$u^\ast =\cN \mu^\ast \in \cV$ be an approximate solution.}
For a given regularization parameter $\alpha>0$ and collocation number $K = K_1 + K_2 >0$, let  $u_{K,\alpha}  =\cN \mu_{K,\alpha}$ be a global minimizer of the empirical problem \eqref{eq:empirical_sparse_min}. 
\revision{Let Assumptions~\ref{assu:embedding}, ~\ref{assu:measureconvergence} and~\ref{assu:uniformcontinuity} hold and denote by
\[
\eta(\mu^*,\alpha) = \left(\| u - u^\ast\|_{\cU_o}^2 + \alpha \| \mu^\ast \|_{M(\Omega)}\right)^{1/2}
\]
the combined approximation and regularization error.}
As $K_1, K_2 \to \infty$, any limit point $\mu_\alpha$ of $\mu_{K,\alpha}$ in $M(\Omega)$ defines a function $u_\alpha = \cN \mu_\alpha$ that satisfies 
\revision{
\[
\| u_{\alpha} -  u \|_{\cU} \leq C  \eta(\mu^*,\alpha),
\]
}%
if \Cref{assu:embedding}(1) holds, or 
\revision{
\[
\| u_{\alpha} - u\|_{\cU} \leq C \max\left(\alpha^{- (1-\theta)} \eta(\mu^*,\alpha)^{2(1-\theta) + \theta}, \eta(\mu^*,\alpha)^{\theta} \right), 
\]
}%
if \Cref{assu:embedding}(2) holds, where $C>0$ is a generic constant independent of $\alpha$, $\mu^\ast$, $u$ and $u^\ast$. Moreover, the convergence $u_{K, \alpha} \to u_\alpha$ (up to a subsequence) holds in $C^2(\overline{D})$.
\end{theorem}

\begin{remark}
\label{rmk:convergence}
\Cref{thm:convergence}  yields the following implications. 
 \begin{itemize}[left=1em]
\item[(1)] The above estimate indicates that the approximation error arises from two sources. The first term reflects the error due to approximating the exact solution $u\in \cU_o$ using elements \revision{from a norm ball in the} space $\cV$, while the second term characterizes the error due to the regularization.
\item[(2)] Using the above estimate, one can select the appropriate $\alpha$ such that $u_\alpha$ converges to the true solution $u$ in $\cU$.
In the first case,  by density of $\cV$ in $\cU_0$, one can select $u^\ast = u^\delta \in \cV$ such that $ \| u - u^\delta\|_{\cU_o}^2 = \delta$. Let $\alpha = \min( \delta/ \norm{u^\delta}_\cV, \delta )$, we observe that $\alpha \to 0$ as $\delta \to 0$.  Then, by applying the estimate from \Cref{thm:convergence}.
 \[
 \| u_{\alpha} -  u \|_{\cU} \leq C(2\delta)^{1/2},
 \]
 which implies $u_\alpha \to u $ in $\cU$ as $\alpha\to0$.
 The second case is more delicate. Without additional assumptions on the function spaces, selecting $\alpha$ that guarantees convergence is challenging. As a simple illustration, we consider a setting where the space $\cV$ contains the solution $u\in \cU_o$.  This inclusion can be justified in certain cases, such as the example provided in \Cref{subsec:functional_setup}. In this simple case, let $u = \cN \mu^\ast$. Then
\[
\| u_{\alpha} - u\|_{\cU} \leq C\alpha^{\theta/2}  \max\left(  \|\mu^\ast \|_{M(\Omega)}^{(1-\theta) + \theta/2},  \|\mu^\ast \|_{M(\Omega)}^{\theta/2} \right), 
\]
 which implies $u_\alpha \to u $ in $\cU$ as $\alpha\to0$.
 \item[(3)] In the cases where $\cF$ does not contain $L^2_\lambda$ and $\cV$ is strictly smaller than $\cU_o$, one possible remedy is to design an appropriate loss function such that the corresponding data space is continuously embedded in $\cF$. Revisiting the example with $\cF = H^{-1}(D)\times H^{1/2}(\partial D)$, it suffices to define the loss function as
\[
    L(u) =\frac{1}{2} \norm{\cE[u] }^2_{L^2(\nu_D)} + \frac{\lambda}{2} \norm{\cB[u]}^2_{H^1(\nu_{\partial D})},
\]
with the corresponding empirical loss function. In this case, we can establish similar error estimates. In \Cref{sec:experiments}, we present numerical experiments comparing the effects of using an $L^2$ boundary loss versus an $H^1$ boundary loss. 
\end{itemize}

\end{remark}

\begin{proof}
We now prove \Cref{thm:convergence}.
Denote $\hat{J}(\mu) = \hat{L}(\cN\mu) + \alpha \norm{\mu}_{M(\Omega)}$.
Since $u_{K,\alpha}  =\cN \mu_{K,\alpha}$ is a global minimizer of \eqref{eq:empirical_sparse_min}, it satisfies 
\[
\alpha \|  \mu_{K,\alpha}\|_{M(\Omega)} \leq \hat{J}(\mu_{K,\alpha}) \leq \hat{J}(\mu^\ast) =  \hat{L}(u^\ast)+ \alpha \|\mu^\ast \|_{M(\Omega)} \quad \text{for all } K>0,
\]
where $\hat{L}(u^\ast) = \hat{L}_{K_1, K_2}(u^\ast)$ is given by 
\[
\hat{L}_{K_1, K_2}(u^\ast)  =\frac{1}{2} \int_{D} (\cE[ u^\ast](x))^2 \de\nu_D^{K_1}+ \frac{\lambda}{2}   \int_{\partial D} (\cB[ u^\ast](x))^2 \de\nu_{\partial D}^{K_2}.
\]
This term converges to $L(u^\ast)$ as $K_1, K_2 \to \infty$, and is therefore uniformly bounded in $K$.  
This implies that the sequence $\{ \mu_{K,\alpha} \}_K \subset M(\Omega)$ is bounded.  By the Banach-Alaoglu theorem, there exists a subsequence (not relabeled) and a $  \mu_{\alpha}\in M(\Omega)$ such that $\mu_{K,\alpha} \stackrel{\ast}{\rightharpoonup} \mu_{\alpha}$ as $K\to \infty$.
By \Cref{lem:uniformconvergence} $u_{K,\alpha}$ converges to $u_{\alpha} = \cN \mu_{\alpha}$ in $C^2(\overline{D})$ as $K\to\infty$.

%By \Cref{eq:continuity}, $u_{K,\alpha}$ and all its partial derivatives up to second order converge pointwise to the corresponding derivatives of $u_{\alpha}$ as $K\to\infty$. Therefore, the functions $u_{K,\alpha}$ and all their first-order derivatives are equicontinuous, and the Arzelà–Ascoli theorem tells us that $u_{K,\alpha}$ converges to $u_{\alpha}$ in $C^1(\overline{D})$ as $K\to\infty$. 

Now we show the estimate for $u_{\alpha}$.
\revision{%
Notice that 
\[
\begin{split}
\hat{L}(u_{K,\alpha}) - L(u_{\alpha}) & = \int_\Omega \left( (\cE[u_{K, \alpha}])^2  -(\cE[u_{ \alpha}] )^2\right)\de\nu_D^{K_1}   + \int_{\partial \Omega} \left( (\cB[u_{K, \alpha}])^2 -(\cB[u_{ \alpha}])^2\right) \de\nu_{\partial D}^{K_2}   \\
& + \int_\Omega (\cE[u_{ \alpha}])^2 \left(\de\nu_D^{K_1} -\de\nu_D \right)  + \int_{\partial \Omega}(\cB[u_{ \alpha}] )^2\left(\de\nu_{\partial D}^{K_2} - \de\nu_{\partial D}  \right) .
\end{split}
\]
The first two terms in the above converge to zero by the uniform convergence established in \Cref{lem:uniformconvergence} and a uniform bound on the total variation norm of the empirical measures implied by weak-\(*\) convergence from Assumption~\ref{assu:measureconvergence} and the last two terms converge to zero. 
Therefore $\lim_{K\to\infty}\hat{L}(u_{K,\alpha}) = L(u_{\alpha})$.
}
Finally, by the lower semicontinuity of the norm $\| \cdot \|_{M(\Omega)}$ under weak-$\ast$ convergence, we conclude
\[
L(u_{\alpha})\leq J(u_{\alpha})\leq \liminf_{ K\to\infty} \hat{J}(\mu_{K,\alpha}) \leq \liminf_{ K\to\infty} \hat{J}( \mu^\ast) = L(u^\ast) +  \alpha \|\mu^\ast \|_{M(\Omega)}. 
\]
The estimate above gives us
\beq
\label{eq:residualestimate}
\| \cR[u_{\alpha}]\|_{L^2_\lambda}^2 \leq \| \cR[u^\ast]\|_{L^2_\lambda}^2  +  \alpha \|\mu^\ast \|_{M(\Omega)}. 
\eeq

We now utilize \Cref{eq:residualestimate} to show the final result. First, we assume \Cref{assu:embedding}(1) holds. Then, combining this with with \Cref{assu:wellposedness}(3), we have
\[
\|u_{\alpha}-u \|_{\cU} \leq C_2 \| \cR[u_{\alpha}]\|_{\cF}^2  \leq  C_2 \| \cR[u_{\alpha}]\|_{L^2_\lambda}^2.
\]
On the other hand, we use \Cref{assu:embedding}(1) and \Cref{assu:wellposedness}(2) to obtain
\[
\| \cR[u^\ast]\|_{L^2_\lambda}^2  +  \alpha \|\mu^\ast \|_{M(\Omega)} \leq C \| \cR[u^\ast]\|_{\cF_0}^2  +  \alpha \|\mu^\ast \|_{M(\Omega)}   \leq CC_1 \| u - u^\ast\|_{\cU_0}^2  +  \alpha \|\mu^\ast \|_{M(\Omega)}. 
\]
Together, these inequalities yield the desired result. Now we assume \Cref{assu:embedding}(2) holds, then the right-hand side of \eqref{eq:residualestimate} is estimated in the same way. For the left-hand side, we use the interpolation equality and \Cref{assu:wellposedness}(2), 
\[
 \| \cR[u_{\alpha}]\|_{\cF}^{2/\theta}  \leq C  \| \cR[u_{\alpha}]\|_{L^2_\lambda}^{2}  \| \cR[u_{\alpha}]\|_{\cF_o}^{2(1-\theta)/\theta} \leq  C  \| \cR[u_{\alpha}]\|_{L^2_\lambda}^{2}  \max(1, \| u_{\alpha}\|_{\cU_o}^{2(1-\theta)/\theta})  
\]
Note that 
\[
  \| u_{\alpha}\|_{\cU_o} \leq   \| u_{\alpha}\|_{\cV}  = \| \mu_{\alpha}\|_{M(\Omega)} \leq \frac{1}{\alpha} L(u^\ast ) + \| \mu^\ast\|_{M(\Omega)} =  \frac{1}{\alpha}\| \cR[u^\ast]\|_{L^2_\lambda}^2  +   \|\mu^\ast \|_{M(\Omega)} . 
\]
Together, we have
\[
\begin{split}
 \| u_{\alpha} - u\|_{\cU}^{2/\theta} &\leq \| \cR[u_{\alpha}]\|_{\cF}^{2/\theta}\\
 &\leq  \max\left(\alpha^{- 2(1-\theta)/\theta} \left(  \| \cR[u^\ast]\|_{L^2_\lambda}^2  + \alpha  \|\mu^\ast \|_{M(\Omega)}\right)^{2(1-\theta)/\theta + 1}, \| \cR[u^\ast]\|_{L^2_\lambda}^2  + \alpha  \|\mu^\ast \|_{M(\Omega)}\right)\\
 &\leq C \max\left(\alpha^{- 2(1-\theta)/\theta} \left( \| u-u^\ast\|_{\cU}^2  + \alpha  \|\mu^\ast \|_{M(\Omega)}\right)^{2(1-\theta)/\theta + 1}, \| u-u^\ast\|_{\cU}^2  + \alpha  \|\mu^\ast \|_{M(\Omega)} \right), 
 \end{split}
\]
which leads to the desired result.
\end{proof}

% \kp{
% To ensure that we can take the simultaneous limit $\alpha_K \to 0$ as $K\to \infty$, we have to ensure some condition of the form
% \[
% \alpha_K \geq C(K),
% \]
% where \(C(K) \to 0\) for \(K \to \infty\);
% cf.~\cite{Hofmann_2007}. The problem is that we may have a well-posed problem, but a potentially a mismatch of spaces and that we need an a~priori bound on the ``generalization error''~\eqref{eq:generalization} (cf~\cite{bach:2017}, denoted estimation error), which requires a robust bound of \(\norm{\mu}\).
% For   second point, I think we can for some results assume that \(u \in \cV\), as an additional regularity assumption (without assuming that \(\cU = \cV\)). However, for real examples this will probably require very smooth domains with very smooth data (more than Lipschitz domain and \(f\in C(D)\)...). However, I do not think this alone will allow us to send \(\alpha \to 0\) for finite \(K\), without some additional work.
% }

The above theorem describes the limiting behavior of the neural network solution as the number of collocation points $K$ tends to infinity. In practice, however, we are also interested in the behavior of the solution for finite $K$. More importantly, given a regularization parameter $\alpha$, we aim to identify a suitable range of $K$ that prevents overfitting. 
\revision{%
For a neural network solution $u_\mu$, we interpret generalization as the property that the true loss $L(u_\mu)$ is close to the empirical loss $\hat L(u_\mu)$ computed from a finite set of collocation points, 
particularly when the network is trained to achieve a small empirical loss.}
This can be analyzed as follows. We first define a discrete version of the weighted norm $\| \cdot\|_{L^2_\lambda}$: 
\[
\norm{e}_{2,\lambda,K}^2 :=
    \frac{1}{2} \sum_{k=1}^{K_1} w_{1,k}(e_E(x_{1,k}))^2 + \frac{\lambda}{2} \sum_{k=1}^{K_2}w_{2,k}(e_B(x_{2,k}) )^2
\]
\revision{%
Now, we leverage that the discrete measures \(\nu^{K_1}_D\) and \(\nu^{K_2}_{\partial D}\) associated to this norm converge according to \Cref{assu:measureconvergence}.
This convergence can be quantified for any residual associated to a trained approximate solution \(u_{K,\alpha} = \cN(\mu_{K,\alpha})\) by leveraging the norm bound and the improved continuity from \Cref{lem:uniformconvergence}. 
}
\revision{
\begin{lemma}
\label{lem:generalization}
Given $\delta, M>0$, there exists $K = K(\delta, M) \in \N$, such that for any $v\in \cV$ with $\| v\|_{\cV} \leq M$
\[
 \left| \hat{L}(v)  - L(v) \right| \leq  \delta
\]
where $\hat{L} = \hat{L}_{K_1, K_2}$ and $K = K_1 + K_2$. 
\end{lemma}
\begin{proof}
We first claim that
\begin{equation}
\label{eq:measureconvergence}
\norm{\nu_D - \nu_D^{K_1}}_{C^{0,\gamma}(\overline{D})^*}
=
\sup_{f: \norm{f}_{C^{0,\gamma}(\overline{D})} \leq 1} \abs*{\int_D f \de \nu_D - \int_D f \de \nu^{K_1}_D}
\to 0
\end{equation}
for \(K_1 \to \infty\) and similarly for \(\nu_{\partial D}\).
Given this, define $e = \cR[v]$, we have
\[
\begin{aligned}
    \abs*{\int_D e^2(x) \de\nu_D - \int_D e^2(x) \de\nu^{K_1}_D}
&\leq \norm{e^2}_{C^{0,\gamma}(\overline{D})} \norm{\nu_D - \nu_D^{K_1}}_{C^{0,\gamma}(\overline{D})^*}\\
&\leq C \norm{e}^2_{C^{0,\gamma}(\overline{D})} \norm{\nu_D - \nu_D^{K_1}}_{C^{0,\gamma}(\overline{D})^*}.
\end{aligned}
\]
and the same holds for \(\nu_{\partial D}\). The above estimate together with the fact that
$
\norm{e}_{C^{0,\gamma}} =\norm{\cR[v]}_{C^{0,\gamma}} \leq C \norm{v}_\cV$ imply the desired result. 

To prove the convergence \eqref{eq:measureconvergence}, we argue by contradiction. If it fails,  then there exists $\epsilon > 0$ and a sequence \(f_{K_1}\) in the unit ball of \(C^{0,\gamma}\) for all \(K_1\) such that 
\[
\abs*{\int_D f_{K_1} \de \nu_D - \int_D f_{K_1} \de \nu^{K_1}_D} \geq \epsilon > 0
\]
for all \(K_1\). By Arzel\`a--Ascoli, up to a subsequence, \(f_{K_1} \to \bar{f}\) uniformly on \(\overline{D}\). 
We now test the convergence of $\nu_D^{K_1}$ to $\nu_{D}$ against $\bar{f}$:
\[
\int_D \bar{f} \de \nu^{K_1}_D - \int_D \bar{f} \de \nu_D
= \int_D f_{K_1} \de \nu^{K_1}_D + \int_D \bar{f} -f_{K_1} \de \nu^{K_1}_D - \int_D f_{K_1} \de \nu_D - \int_D \bar{f} -f_{K_1} \de \nu_D.
\]
The second and fourth term converge to zero, but the remaining two terms are always \(\epsilon\) apart, and thus \(\nu_D^{K_1}\) does not converge to \(\nu_D\) in the weak-$\ast$ sense, which contradicts \Cref{assu:measureconvergence}.
\end{proof}
}

\begin{theorem}
\label{thm:error_bound}
Let $u$ and $u_{K,\alpha}$ denote the functions given in \Cref{thm:convergence}. For any $\delta > 0$, there exist $u^\delta \in \cV$ and  parameters $\alpha = \alpha(\delta) \to 0$ and $K = K(\delta) \to \infty$ as $\delta\to0$ such that 
\[
\| u_{K,\alpha} - u \|_\cU \leq C \delta^{1/2},
\]
if \Cref{assu:embedding}(1) holds, or 
\[
\| u_{K,\alpha} - u \|_\cU \leq C \delta^{\theta/2} \max(1,  \| u^\delta\|_{\cV}^{1-\theta}), 
\]
if \Cref{assu:embedding}(2) holds, where $C>0$ is a generic constant. 
%independent of $\alpha$, $K$, $u_{K,\alpha}$ and $u$.  

%\xt{Here we encountered the same issue to show convergence for the second case. Basically, we can't guarantee $ \delta^{\theta/2} \| u^\delta\|_{\cV}^{1-\theta}$ converges to zero without finer estimates.}
%\kp{As far as I understand this result is proved but we can not simplify it further without imposing additional assumptions in the second case. I think this is fine, since the first setting is the more intuitive one, where \(\cF\) is larger than \(L^2(D) \times L^2(\partial D)\). However, these estimates may warrant a discussion on how to choose these spaces for a simple example involving a Laplacian operator. For Neumann BC, we will de able to obtain solutions in \(H^{3/2}(D)\), but for Dirichlet only in \(H^{1/2}(D)\). As we can see from the theory, trying to upgrade this with the second result can be come difficult if \(\cU_0\) is far from \(\cV\). Another approach may be to improve the loss function, and for instance measure the residual on the boundary in \(H^1(\partial D)\) (which is easily possible). This idea comes from this preprint~\cite{ainsworth2024extendedgalerkinneuralnetwork}. I also think it would be good to connect this discussion to the numerical instabilities we observed in Figure~\ref{fig:three-wide}. I believe this is connected.}
\end{theorem}
\begin{remark}
    As noted in \Cref{rmk:convergence}, without further assumptions on the function spaces, the second case guarantees convergence if we know that $\cV$ contains the exact solution $u\in \cU_o$, which ensures that the $\| u^\delta\|_{\cV}$ remains uniformly bounded. Once again, an alternative remedy is to design a loss function that ensures the corresponding data space is continuously embedded in $\cF$. Under such construction, convergence can be established in a similar way without additional assumptions.
\end{remark}
\begin{proof}
First, by density of $\cV$ in $\cU_o$, for any $\delta>0$, there exists $u^\delta \in \cV$ such that 
\[
L(u^\delta) = \norm{\cR[u^\delta]}_{L^2_\lambda}^2 \leq C \| u - u^\delta\|_{\cU_o} = \delta. 
\]
Using the above lemma, there exists $K = \revision{K(\delta, 3\| u^\delta\|_{\cV})}$ such that 
\[
\revision{
\left| \hat{L}(v) - L(v) \right| \leq \delta
}
\]
for all $v\in \cV$ satisfying $\| v\|_{\cV} \leq  3 \| u^\delta\|_{\cV}$.
Letting $\alpha = \delta/ \| u^\delta \|_{\cV}$, and noting that $u_{K, \alpha}$ is a global minimizer of \eqref{eq:empirical_sparse_min}, we obtain
\[
\alpha \| u_{K, \alpha}\|_{\cV} + \hat{L}(u_{K,\alpha}) \leq \alpha \| u^\delta \|_{\cV} + \hat{L}(u^\delta )\leq   \alpha \| u^\delta \|_{\cV} + \revision{L(u^\delta)+ \delta} = 3\delta. 
\]
This implies $\| u_{K, \alpha}\|_{\cV}  \leq 3  \| u^\delta\|_{\cV}$, and thus
\[
L(u_{K,\alpha}) \leq \revision{ \hat{L}(u_{K,\alpha}) + \delta \leq 4\delta}.
\]
Under \Cref{assu:embedding}(1), it follows that
\[
\| u_{K,\alpha} - u\|_{\cU} \leq C \| \cR[u_{K,\alpha}]\|_{\cF} \leq C \| \cR[u_{K,\alpha}]\|_{L^2_\lambda}  \leq C \left( L(u_{K,\alpha}) \right)^{1/2} \leq C \delta^{1/2}. 
\]
On the other hand, if  \Cref{assu:embedding}(2) holds, then
\[
 \| \cR[u_{K, \alpha}]\|_{\cF}^{2/\theta}  \leq C  \| \cR[u_{K, \alpha}]\|_{L^2_\lambda}^{2}  \| \cR[u_{K, \alpha}]\|_{\cF_o}^{2(1-\theta)/\theta} \leq  C  L(u_{K,\alpha}) \max(1, \| u_{K, \alpha}\|_{\cU_o}^{2(1-\theta)/\theta})  
\]
Moreover, since
\[
  \| u_{K,\alpha}\|_{\cU_o} \leq   \| u_{K, \alpha}\|_{\cV} \leq 3 \| u^\delta\|_{\cV} ,
\]
we obtain
\[
\| u_{K,\alpha} - u \|_\cU^{2/\theta} \leq C \delta \max(1,  \| u^\delta\|_{\cV}^{2(1-\theta)/\theta}),
\]
which leads to the desired result. 
\end{proof}

\revision{
\begin{remark}
The estimates in \Cref{lem:generalization} and \Cref{thm:error_bound} can be made more specific when considering a specific quadrature rule.
For example, suppose the quadrature points are quasi-uniform, with fill distances $h_D$ and $h_{\partial D}$ for the points on $D$ and $\partial D$, respectively. 
The quadrature weights can then be chosen as the measures of the corresponding Voronoi cells. In this setting, one can show that 
\[
\left| \norm{e}_{L^2_\lambda}^2 - \norm{e}_{2,\lambda,K}^2  \right| \leq C \left( h_D^\gamma + \lambda\, h_{\partial D}^{\gamma} \right) \| e\|_{C^{0,\gamma}(\overline{D})} \norm{e}_{C(\overline{D})}
\leq C \left( h_D^\gamma + \lambda\, h_{\partial D}^{\gamma} \right) \| e\|^2_{C^{0,\gamma}(\overline{D})} ,
\]
where $C$ is a generic constant.
For $e = \cR[v]$ and quasi-uniform quadrature points, this implies 
\[
\left|L(v) - \hat L(v)\right|\leq C \left( K_1^{-{\gamma}/{d}} + \lambda\, K_2^{-{\gamma}/{(d-1)}}\right) \norm{v}^{2}_\cV . 
\]
  The constant  $K = K(\delta, 3\| u^\delta\|_{\cV})$ in the proof of \Cref{thm:error_bound} can also be chosen explicitly in this case. Indeed, by the above estimate, one can take $K_1^{-{\gamma}/{d}} \approx \lambda\, K_2^{-{\gamma}/{(d-1)}}$ to balance the errors from the interior and the boundary. Then it suffices to choose $K_1 \approx \left(\|u^\delta\|^{2}_{\cV}/\delta\right)^{d/\gamma}$ and $K_2 \approx \left(\|u^\delta\|^{2}_{\cV}/(\lambda\delta)\right)^{(d-1)/\gamma}$. 
\end{remark}
}

\subsection{Representer theorem}
\label{subsec:representertheorem}
To further simplify the notation, we introduce the combined collocation set for the residual \(\{x_{1,k}\}\cup \{x_{2,k}\} \in \overline{D} = D \cup \partial D\), and the combined set of quadrature weights \(\{w_{1,k}\}\cup \{\lambda w_{2,k}\}\).
Then we introduce a new global index \(k = 1,\ldots, K_1 + K_2 = K\), such that \(x_k = x_{1, k}\)  and \(w_k = w_{1, k}\) for \(k \leq K_1\), while for $k > K_1$, we set \(x_{k} = x_{2, k-K_1}\) and \(w_{k} = \lambda w_{2, k-K_1}\).
With this, we 
introduce the combined linear operator evaluations as
\[
l_k[u_\mu] = \begin{cases}
[\cL_E u_\mu ](x_{k}) &\text{for } k \leq K_1, \\
[\cL_B u_\mu ](x_{k}) &\text{for } k > K_1.
\end{cases}
\]
Moreover we define the \(k\)-th residual function as
\beq
\label{eq:nonlin_resid}
r_k(l_k) = \begin{cases}
\hat{E}(x_{k},l_k) &\text{for } k \leq K_1, \\
\hat{B}(x_{k},l_k) &\text{for } k > K_1.
\end{cases}
\eeq
This allows us to rewrite the regularized discrete residual minimization problem~\eqref{eq:empirical_sparse_min} as
\beq
\label{eq:finite_min_compact}
\min_{\mu \in M(\Omega)} \frac{1}{2}\sum_{k=1}^{K} w_k (r_k(l_k[u_\mu]))^2 +
\alpha \norm{\mu}_{M(\Omega)}
\eeq

We proceed to prove the representer theorem, which fundamentally relies on variants of the Carathéodory theorem \citep{dubins1962extreme,klee1963theorem}.  A key ingredient in the argument is the characterization of the extreme points of certain sets arising in the sparse minimization problem. Recall that an extreme point of a convex set is a point that cannot be expressed as a nontrivial convex combination of two distinct elements in the set. 
To that end, we first summarize some standard properties of closed balls in $M(\Omega)$ in the following lemma. 
\begin{lemma}
\label{lem:sublevelset}
Define $B_t :=  \{\mu\in M(\Omega) \colon \|\mu\|_{M(\Omega)} \leq t\}$.
The following statements hold. 
\begin{enumerate}
\item $B_t$ is compact in the weak-$\ast$ topology on $M(\Omega)$. 
\item The extreme points of $B_t$ are given by $\pm t \delta_\omega$ for $\omega\in \Omega$. 
\end{enumerate}
\end{lemma}
The first statement in the above comes from the Banach-Alaoglu theorem and the lower semicontinuity under the weak-$\ast$ convergence. The second statement is also well-known, see e.g., \cite{bartolucci2023understanding,bredies2020sparsity}.

Let $\bar{N} = N_E K_1 + N_B K_2$. Using the notation introduced above, we view the collection $l[u_\mu] := (l_k[u_\mu])_{k=1}^{\bar{N}}$ as a vector in $\R^{\bar{N}}$.
Given a vector $\vec{l}\in \R^{\bar{N}}$, we consider the following minimization problem:
\beq
\label{eq:constrainedmin}
\min_{\mu \in M(\Omega)}  \| \mu\|_{M(\Omega)} \quad \text{subject to} \quad  l[u_{\mu}] =\vec{l}.
\eeq

\begin{lemma}
\label{lem:constrainedmin}
Assume that the solution set $S$ of \eqref{eq:constrainedmin} is nonempty. Then $S$ is a compact convex subset of $M(\Omega)$ equipped with the weak-$\ast$ topology.  As a consequence, $S$ contains at least one extreme point. 
\end{lemma}
\begin{proof}
First of all, by the linearity of the constraint, it is straightforward that $S$ is convex. By the lower semicontinuity of the total variation norm and \Cref{lem:continuity}, $S$ is also closed. Finally, since the optimization problem minimizes the total variation norm subject to a constraint, any solution must lie in a norm sublevel set $B_t$ for some $t>0$. Hence, $S\subset B_t$ is a closed subset of a compact set and is therefore compact. Lastly, by the Krein-Milman theorem \citep{conway1994course}, $S$ contains at least one extreme point. 
\end{proof}

The representer theorem can now be proved with known arguments ~\citep[cf., e.g.,][Section~3.3.2]{boyer2019representer}.

\begin{theorem}[Representer theorem (finite representation)]
\label{thm:representer}
The problem~\eqref{eq:empirical_sparse_min} possesses a finite solution
\[
\bar{\mu} = \sum_{n=1}^{N} c_n \delta_{\omega_n},
\quad
\cN \bar{\mu} = \sum_{n=1}^{N} c_n \varphi(\cdot;\omega_n)
\]
where the number of atoms satisfies \(N \leq \bar{N} = N_E K_1 + N_B K_2\).
\end{theorem}
\begin{proof}
According to Theorem~\ref{thm:existence}, a (not necessarily finite) solution exists. However, standard representer theorems derived for convex problems are not immediately applicable~\citep[see, e.g.,][Theorem 3.9]{bartolucci2023understanding}, due to the nonlinearity~\eqref{eq:nonlin_resid}. We argue by showing that any solution of~\eqref{eq:finite_min_compact} also solves a simpler convex problem. For that we define the optimal linear operator vector as
\(\bar{l} = l[u_{\bar{\mu}}] \in \R^{\bar{N}}\)
It is easy to see that \(\bar{\mu}\) is also a minimizer of problem~\eqref{eq:constrainedmin} with data \(\vec{l} = \bar{l}\)
% the following convex problem
%\[
%\min \{ \norm{\mu}_{M(\Omega)} \;|\;
% \mu \in M(\Omega);\, l[u_{\mu}] = \bar{l} \},
%\]
and conversely, any solution to this problem is also a solution of \eqref{eq:empirical_sparse_min}.  This problem is convex and, by applying \cite[Theorem 3.1]{boyer2019representer}, together with \Cref{lem:sublevelset,lem:constrainedmin}, we conclude that there exists a solution expressible as a convex combination of at most $\bar{N}$ extreme points of the norm ball \( B_{t^\ast} \), where \( t^\ast \) is the optimal value of \eqref{eq:empirical_sparse_min}.
This yields the desired expression for $\bar{\mu}$ by \Cref{lem:sublevelset}(2). 
\end{proof}

%\begin{remark}
%A similar argument can also be used for the \(\delta\)-constrained problem~\eqref{eq:delta_constrained_problem}, even if \(\delta = 0\).
%\kp{However, a set of conditions that ensures that this problem has admissible points with residual \(\leq \delta\) needs to be added to guarantee that solutions exist in the first place.}
%\xt{Shall we remove this remark?}
%\end{remark}

\subsection{Optimality conditions and dual variables}
\label{subsec:optimality}
First order optimality conditions for this problem can be derived.
\begin{proposition}
\label{prop:first_order_nec}
Any solution \(\bar{u}\) with \(\bar{u} = u_{\bar{\mu}}\) to the problem~\eqref{eq:finite_min_compact} fulfills the following first order necessary conditions:
\[
    \sum_{k=1}^{K} w_k r_k(l_k[\bar{u}])(\nabla_l r_k(l_k[\bar{u}]))^T l_k[\cN(\mu - \bar{\mu})] +
    \alpha (\norm{\mu}_{M(\Omega)} - \norm{\bar{\mu}}_{M(\Omega)}) \geq 0,
\]
for all \(\mu \in M(\Omega)\). Equivalently, it fulfills the conditions
\[
    \sum_{k=1}^{K} w_k r_k(l_k[\bar{u}]) (\nabla_l r_k(l_k[\bar{u}]))^T l_k[u - \bar{u}] +
    \alpha (\norm{u}_{\cV(D)} - \norm{\bar{u}}_{\cV(D)}) \geq 0,
    \]
for all \(u \in \cV(D)\).
\end{proposition}
To further interpret these conditions, we use the well-known characterization of the sub-differential
\begin{align*}
\partial \norm{\bar{\mu}}_{M(\Omega)}
&= \{v \in C_0(\Omega) \;|\; \norm{\mu}_{M(\Omega)} - \norm{\bar{\mu}}_{M(\Omega)} \geq \pair{\mu - \bar{\mu}, v} \text{ for all } \mu \in M(\Omega)\} \\
&= \{v \in C_0(\Omega) \;|\; \norm{v}_{C(\Omega)} \leq 1 , v = \sgn \bar{\mu} \text{ on } \supp \bar{\mu} \}
\end{align*}
to write the optimality conditions as
\[
-\bar{p} \in \alpha \partial \norm{\bar{\mu}}_{M(\Omega)}
\]
for some appropriate dual variable \(\bar{p} \in C_0(\Omega)\), such that
\[
\pair{\bar{p}, \mu - \bar{\mu}} =
\sum_{k=1}^{K} w_k r_k(l_k[\bar{u}])(\nabla_l r_k(l_k[\bar{u}]))^T l_k[\cN(\mu - \bar{\mu})]
\]
To characterize \(\bar{p}\), we require the dual operator of \(\cN\) and the linear PDE operator in an appropriate sense. However, this can get very technical due to the fact that this dual operator would have to take inputs from the dual space of \(C^{2}(\overline{D})\), which is a space of distributions~\citep[cf.][section~5.2]{wachsmuth2024nogap}.
To simplify the presentation, we identify \(C^{2}(\overline{D})\) with  a subspace of \(C(\overline{D}, \R \times \R^d \times \R^{d^2})\), using the embedding
\[
\iota: C^{2}(\overline{D}) \to C(\overline{D}, \R \times \R^d \times \R^{d^2}),
\quad\iota v(x) = (v(x), \nabla v(x), \nabla^2 v(x)) \quad \text{for all } x \in \overline{D}.
\]
Combining the neural network and this embedding yields
\[
[\iota \cN \mu](x) = \int_\Omega (\varphi(x;\omega), \nabla_x \varphi(x,\omega), \nabla_x^2\varphi(x,\omega)) \de \mu(\omega).
\]
Then, we realize that \(l_k[\cN \mu]\) corresponds to an evaluation of this integral operator together with an inner product in the vector space \(\R \times \R^d \times \R^{d^2} = \R^{1 + d + d^2}\): 
%\xt{In your original definition of $l_k[\cN \mu]$, a set of coefficients $(c_i, b_i, A_i)_{i}$ is involved. I guess you tried to make the setup most general, which could include examples like the Monge-Ampere equation (which involves $\det(\nabla^2 u)$). }\kp{I think the example of Monge-Ampere was already covered by your initial formulation, I just renamed \(P\) to \(E\) and relaxed the differentiability with respect to \(x\), to cover discontinuous coefficients/forcing terms and general mixed BC. The formulation with the \(i\)'s is there to get a better estimate for the number of atoms. But I forgot to include the multiple \(i\) below, and have fixed that. }
For any \(x_k \in D\) we have for the \(i\)-th entry \(i = 1,\ldots,N_E\) of $l_k[\cN \mu]$ we have
\begin{align*}
l_{k,i}[\cN \mu] &=
(c_i(x_k),b_i(x_k),A_i(x_k)) : [\iota \cN \mu](x_k) \\
&= \int_\Omega c_i(x_k)\varphi(x_k;\omega) + b_i(x_k) \cdot \nabla_x \varphi(x_k,\omega) + \tr(A_i(x_k) \nabla_x^2\varphi(x_k,\omega)) \de \mu(\omega),
\end{align*}
and similarly for \(x_k \in \partial D\). 
Define the function under the integral as
\[
\widetilde{\varphi}_{k,i}(\omega) :=
\begin{cases}
c_i(x_k)\varphi(x_k;\omega) + b_i(x_k) \cdot \nabla_x \varphi(x_k,\omega) + \tr(A_i(x_k) \nabla_x^2\varphi(x_k,\omega)) & \text{for } k \leq K_1 \\
d_i(x_k)\varphi(x_k;\omega) + e_i(x_k) \cdot \nabla_x \varphi(x_k,\omega) & \text{for } k \leq K_1.
\end{cases}
\]
With this notation, it is easy to see that
\begin{equation}
\label{eq:dual_variable}
\bar{p}(\omega) =
\sum_{k=1}^{K} w_k r_k(l_k[\bar{u}]) (\nabla_l r_k(l_k[\bar{u}]))^T \widetilde{\varphi}_k(\omega)
\end{equation}
defines the dual variable for the optimality conditions.
%\kp{If this notation remains, we can introduce \(\iota\) and the extended kernel \(\widetilde{\varphi}\) already earlier.}
%\xt{This notation is only used once. Maybe we can leave it as it is?}

With this preparation, we can derive the first order necessary conditions in the form of a support condition on the dual variable.
\begin{proposition} 
The first order necessary conditions derived in Proposition~\ref{prop:first_order_nec} can be expressed with the dual variable~\eqref{eq:dual_variable} as the subdifferential inclusion $-\bar{p} \in \alpha \partial \norm{\bar{\mu}}_{M(\Omega)}$ or the support conditions 
\[
\abs{\bar{p}(\omega)} \leq \alpha
\quad \bar{p} = - \sgn \bar{\mu} \text{ on } \supp \bar{\mu}.
\]
\end{proposition}
The main practical use of the dual variable is that it can be computed for any given \(u\) and provides a way to check for non-optimality. If a node is found that violates the bound on the dual, this provides a descent direction to further decrese the regularized loss. 
\begin{proposition}[Boosting step]
\label{prop:boosting}
Define the dual variable \(p[u]\) associated to the variable \(u = \cN \mu\) as
 \begin{equation}
p[u](\omega) :=
\sum_{k=1}^{K} w_k r_k(l_k[u]) (\nabla_l r_k(l_k[u]))^T \widetilde{\varphi}_k(\omega).
\end{equation}
Let \(\hat{\omega}\) be a coefficient with \(\abs{p[u](\hat{\omega})} > \alpha\). Then, the Dirac delta function at \(\hat{\omega}\) with negative sign of \(p[u](\hat{\omega})\) associated solution perturbation
\[
\hat{u} := \cN(-\sgn(\abs{p[u](\hat\omega)}) \delta_{\hat\omega}) = -\sgn(\abs{p[u](\hat\omega)})  \varphi(\cdot;\hat\omega) 
\]
is a descent direction for~\eqref{eq:finite_min_compact} at \(u = \cN \mu\).
\end{proposition}

%The next proposition gives a concrete interpretation of the dual variable for discrete $u_{c, w}$ in \eqref{eq:emperical_sparse_min_discrete}. 

\begin{remark} 
\label{rmk:discrete_dual}
By direct calculation, the dual variable $ p[u](\omega)$ admits the alternative expression
    \begin{equation}
        p[u](\omega) = \left. \frac{d}{d\tau} \left[ \hat{L} \left( u + \tau\varphi(\cdot; \omega) \right) \right] \right|_{\tau=0} ,\quad \omega\in \Omega.
    \end{equation}
This characterizes $p[u](\omega)$ as the directional derivative of the loss along the candidate feature function $\varphi(\cdot; \omega)$, quantifying the potential decrease in the loss if the node $\omega$ were added.    
This expression provides an intuitive explanation for Proposition \ref{prop:boosting}:  the directional derivative must be large enough (in absolute value) to outweigh the corresponding increase in the regularization penalty.
This observation forms the cornerstone of our optimization strategy, where kernel nodes are inserted and deleted dynamically. 
\end{remark}

\section{Algorithmic framework}
\label{sec:algorithm}

% \zs{Roughly define the theory behind optimization algorithm and define some key variables. The idea is to put a short overview of the algorithm in the main body and put the detailed version in Appendix.}

% Components of the algorithm:
% \begin{itemize}
%     \item Approximation of Hessian, Gauss-Newton method
%     \item Node insertion (Gradient boosting)
%     If the differential operator $\cP$ is nonlinear, we use its Jacobian to approximate $\cP$ at that state.
%     \item 
% \end{itemize}

We now present the algorithmic framework for solving PDEs by optimizing the empirical problem (\ref{eq:emperical_sparse_min_discrete}). Using the notation introduced earlier, we rewrite the optimization problem in the simplified form: 
\begin{equation}
\label{eq:discrete_sparse_min}
    \min_{c, \omega} \hat{L}(u_{c, \omega}) + \alpha\|c\|_{1} = \min_{c, \omega} \hat\ell(c,\omega) + \alpha \|c\|_{1}
\end{equation}
where the reduced loss is defined as
\[
\hat\ell(c,\omega) :=
\frac{1}{2} R(c,\omega)^{T} W R(c,\omega).
\]
Here the residue vector $R(c,\omega)\in \R^K$ is defined componentwise by  $[R(c, \omega)]_{k} = r_k(l_k(u_{c, \omega}))$. \revision{The weight matrix  $W = \text{diag}(w_{1, 1}, \cdots, w_{1, K_{1}}, \lambda w_{2, 1}, \cdots,  \lambda w_{2, K_{2}}) \in \R^{K\times K}$ is determined by chosen quadrature rule and the penalty parameter $\lambda$ and is fixed during optimization}. 
To incorporate the network width $N$ into the optimization process, we introduce insertion and deletion steps for kernel nodes, applied respectively before and after optimizing $(c, \omega)$ with fixed $N$.  
We therefore propose a three-phase algorithm that is executed consecutively and iteratively:
\begin{itemize}
\item Phase I inserts kernel nodes based on the gradient boosting strategy;
\item  Phase II optimizes $(c, \omega)$ with fixed  $N$ using a semi-smooth Gauss-Newton algorithm; 
\item Phase III removes kernel nodes whose associated outer weights are zero.
\end{itemize} 
The full procedure is presented in Algorithm \ref{alg:main}, with additional implementation details provided in \Cref{app:algo}. 

\begin{algorithm}[t]
    \SetAlgoLined
    \KwIn{Initial network $u^{(0)}$(defined by $c^{(0)}$, $\omega^{(0)}$) with width $N^{(0)}$}
    \KwOut{Trained network.}
    
    \While{$ t <  T $ }{
        \textbf{Phase I:} \textit{Kernel nodes insertion.}\\
        \Indp Sample (uniformly) $\{\hat{\omega}_{m}\}_{m=1}^{M} \subseteq\Omega$, let $\hat{p} =\max_{\hat{\omega}_{m}}|p(\hat{\omega}_{m})[u^{(t)}]|$ and  $\omega_{N^{(t)} + 1} = \arg\max_{\hat{\omega}_{m}}|p(\hat{\omega}_{m})[u^{(t)}]|$\\
        Compute $p_0 = \max_{1\leq n\leq N^{(t)}}|p(\omega^{(t)}_{n})[u^{(t)}]|$\\
        
        \If{$\hat{p} > \max \{\alpha, p_0\}$}{
            $N^{(t+1/2)} = N^{(t)} + 1$, $c^{(t+1/2)} = c^{(t)}\cup \{0\}$, $\omega^{(t+1/2)} = \omega^{(t)} \cup \{\omega_{N^{(t)} + 1}\}$
        }
        \Else{
            $N^{(t+1/2)} = N^{(t)}$, $c^{(t+1/2)} = c^{(t)}$, $\omega^{(t+1/2)} = \omega^{(t)}$
        }
        \Indm

        \textbf{Phase II:} \textit{Optimize $c^{(t+1/2)}, \omega^{(t+1/2)}$ simultaneously with a semi-smooth Gauss-Newton algorithm}\\
        \Indp Obtain $J_{R}$ and then $G$, $DG$ based on \eqref{eq:opt_root_finding}, \eqref{eq:opt_Hessian}, \eqref{eq:opt_hessian_approx}.\\
        Compute search direction $(z_{c}, z_{\omega}) = z = -DG^{-1}G$. \\ Perform a line search of optimal step size and update $c^{(t+1/2)}$, $\omega^{(t+1/2)}$.\\
        \Indm

        \textbf{Phase III:} \textit{Kernel nodes  deletion.}\\
        \Indp Remove nodes with $c^{(t+1/2)}_n = 0$ (resulting in $c^{(t+1)}, \omega^{(t+1)}$ with width $N^{(t+1)}$) \\ 
        \Indm
        
        $t = t + 1$
    }
    \caption{Main Algorithm}
    \label{alg:main}
\end{algorithm}

In the following subsections, we present the details of the kernel node insertion step, the semi-smooth Gauss–Newton algorithm, and other practical components of the proposed framework.

\subsection{Kernel nodes insertion via gradient boosting} 
In this subsection, we describe Phase I in the three-phase algorithm, which inserts new kernel nodes to enhance the approximation capacity of the model.
While adding more kernel nodes can improve the model's approximation capability, it also increases the computational cost of the optimization in Phase II. In the worst case, newly inserted nodes may be pruned in Phase III, rendering the added computational cost wasted. Therefore, new nodes should only be introduced when they are expected to yield the steepest descent in the loss function, and when further optimization over existing nodes is no longer effective.

According to \Cref{prop:boosting} and \Cref{rmk:discrete_dual}, the dual variable $p[u](\omega)$  provides a good estimate of the potential reduction in the objective function if a new kernel node is added at location $\omega$. In practice, we uniformly sample candidate locations from $\Omega$ and select the one with the largest $|p[u](\omega)|$ that also satisfies $|p[u](\omega)| > \alpha$. We then compare it with the dual variables of existing kernel nodes.  The selected node, denoted $\omega_{N+1}$, is inserted into the model $u = \sum_{i=1}^{N}c_{n} \varphi(\cdot; \omega_{n})$ if 
\begin{equation}
    \label{eq:insertion_threshold}
    |p[u](\omega_{N+1})| > \max_{1\leq n \leq N}  |p[u](\omega_{n})|  
\end{equation}
The newly added node is initialized with coefficient $c_{N+1} = 0$.

We note that this insertion strategy is inherently greedy, as it targets the direction of steepest local descent in the objective. It is therefore reasonable to incorporate heuristic modifications  to improve performance of the algorithm in practice. Those are discussed in detail in  \Cref{app:algo}.

\subsection{Semi-smooth Gauss-Newton}
The backbone of the three-phase algorithm is the optimization of $(c, \omega) \in \R^{N} \times \R^{N(d+1)} = \R^{N(d+2)}$ in Phase II. The first order optimality condition yields 
\begin{equation}
    \label{eq:opt_subdiff_inclusion}
    - \nabla \hat{\ell}(c, \omega) = - J_{R}(c, \omega)^{T}WR(c, \omega) \in \alpha \partial_{c, \omega} \|c\|_{1}
\end{equation}
where $J_{R}(c, \omega):= (\nabla_c{R}(c, \omega), \nabla_{\omega}R(c, \omega)) \in \R^{K \times N(d+2)}$ 
is the Jacobian and $\partial_{c, \omega}\|c\|_{1} := \partial_{c} \|c\|_{1} \times \{0_{\omega}\in \R^{N(d+1)}\}$ is the extended subdifferential. Instead of using a standard (proximal) gradient descent method, we apply a second order semismooth Newton-type method.  

To this end, we utilize the proximal operator of $\alpha \|c\|_{1}$, defined as
\begin{equation}
\begin{aligned}
\begin{aligned}
    \text{Prox}_{\alpha \|\cdot\|_1}(x) &= \argmin_{y} \left[ \frac{1}{2} \|x - y\|_{2}^{2} + \alpha \|y\|_{1} \right] \\
    &= \left( \text{sign}(x_i) \max\{|x_i| - \alpha, 0\} \right)_{i=1}^{n}.
\end{aligned}
\end{aligned}
\end{equation}
For notational simplicity, we denote $ \text{Prox} :=  \text{Prox}_{\alpha \|\cdot\|_1}$ in remainder of the text.
Using the proximal operator, the subdifferential inclusion \eqref{eq:opt_subdiff_inclusion} is equivalent a root-finding problem: 
find $(q, \omega)$ such that  $c = \text{Prox}(q)$ and 
\begin{equation}
\label{eq:opt_root_finding}
    G(q, \omega) = (q - \text{Prox}(q), 0) + \nabla \hat{\ell}(\text{Prox}(q), \omega)  = 0.  
\end{equation}
$G$ is referred to as the Robinson normal map, introduced by~\cite{robinson1992normal}. 
To address the non-differentiability introduced by the proximal operator, the generalized derivative of $G$ is calculated in a semismooth context \citep{pieper2015finite}, given by
\begin{equation}
\label{eq:opt_Hessian}
    DG(q, \omega) = (\text{Id} - DP(q, \omega)) + \nabla^{2}\hat{\ell}(\text{Prox}(q), \omega) DP(q, \omega)
\end{equation}
where $DP = (D\operatorname{Prox},\ \operatorname{Id})$ is the semismooth derivative of the map $(q, \omega) \rightarrow (\text{Prox}(q), \omega)$. 
The root-finding problem \eqref{eq:opt_root_finding} is then solved via descent direction $z = -DG^{-1}G$, and a line search strategy is applied to determine the step size.
In practice, we approximate the Hessian of the loss by 
\begin{equation}
\label{eq:opt_hessian_approx}
    \nabla^{2} \hat{\ell}(c, \omega) \approx J_{R}(c, \omega)^{T}W J_{R}(c, \omega),
\end{equation}
which is standard in Gauss-Newton type methods \citep{wright2006numerical}.

Note that, due to the introduction of Robinson variables, we need to set \(q_{N+1}\) for a newly inserted point in the greedy/boosting step. To obtain \(c_{N+1} = 0\) and a guarantee that the variable can be immediately activated with a small stepsize,  $q_{N+1}$ is set to be $-\alpha\text{sign}(p[u](\omega_{N+1}))$. Moreover, we adopt the convention that \(D\text{Prox}(\pm \alpha) = 1\).

We remark that the optimization problem is in general nonconvex, which results both from when $r_{k}$ is nonlinear in $l_{k}$, as is the case with nonlinear PDEs, and also the optimization in the inner weights. 
Moreover, the second order optimization algorithm exhibits better efficiency than first order methods. We refer readers to~\citet{pieper2015finite} for a complete discussion of semi-smooth calculus involved here, and additional technical details are discussed in \Cref{app:algo}.

\subsection{Computational cost}
The computational cost of our method primarily arises from evaluating the Jacobian $J_{R}$ as well as dual variables $p$ for sampled weights $\omega$. At each iteration, the computational cost for computing these scales as $\cO(NK)$ and $\cO(MK)$ respectively. Given the dependencies on number of collocation points $K$, this cost can be effectively reduced by subsampling a minibatch of collocation points at each iteration, which is widely adopted in the training of data-driven models. 

We note that the computational cost is well controlled by our adaptive kernel node insertion and deletion strategy, where the network width $N$ is increased only when necessary. However, $N$ may still become undesirably large when approximating complex solutions \revision{(e.g. high-dimensional solutions without localized features or low effective dimensions)}. This issue becomes more obvious in high-dimensional problems as the number of parameters to optimize also scales with the dimension $d$. To mitigate this, one may choose, based on their dual variables, only a subset of $\{\omega_n\}_{n=1}^{N}$ during each iteration to update.

In this work, however, we do not employ such acceleration techniques and instead use the most basic version of the algorithm.

\revision{
\subsection{Choice of hyperparameters: penalty weight $\lambda$ and regularization parameter $\alpha$}
In practice, the hyperparameters $\alpha$ and $\lambda$ have a significant impact on solution quality. The parameter $\alpha$ controls the strength of the $\ell_1$ regularization and thus influences the compactness of the learned representation. Larger values of $\alpha$ promote sparser kernel representations, whereas smaller values allow more kernels and may improve sparsity at the risk of overfitting. For this reason, we often adopt a continuation strategy in $\alpha$: we first train with a moderately large value to obtain a stable sparse representation, and then gradually decrease $\alpha$ while initializing from the previous solution (see \Cref{sec:experiments}).

The penalty weight $\lambda$ balances enforcement of boundary conditions against the interior PDE residual. If $\lambda$ is too small, boundary constraints may be under-enforced, resulting in noticeable boundary errors. Conversely, overly large $\lambda$ may bias the optimization toward boundary fitting, slowing convergence and potentially degrading overall accuracy. In practice, we choose $\lambda$ so that the boundary penalty term and the interior residual term remain comparable in magnitude during training. Additional treatments of boundary conditions are discussed in the next subsection.

We also note that both parameters are inherently related to the number of collocation points $K$. As $K$ increases, the empirical approximations of both the interior residual and boundary penalty terms become more accurate, and the sensitivity to the precise choice of $\alpha$ and $\lambda$ is typically reduced. While the analysis in \Cref{sec:theory} clarifies the roles of these parameters, we do not claim a universal scaling law with respect to $K$. Designing automated strategies for tuning $\alpha$ and $\lambda$ remains an interesting direction for future work.
}
% \xt{Please see the new Remark 20. How should we properly refer to the scaling?}
% \kp{I am not sure. The new Remark establishes an a~priori scaling, which clearly depends on the dimension \(d\). This could be a worst-case scaling law, which may be too pessimistic. On the other hand, if we use fixed quasi-uniform points, the form of the scaling and the dimension dependence seems quite realistic.}

\subsection{Treatment of boundary conditions}
\label{subsec:bnd_treat}
We note that our method does not strictly enforce the boundary condition but includes it as a penalty term in the objective. which is empirical estimation of $\|\cB[u_{c, \omega}]\big|_{\partial D}\|^{2}_{L^{2}(\partial D)}$. While widely adopted, the choice of $L^{2}$-norm is largely heuristic and not always theoretically appropriate. In particular, for problems with Dirichlet boundary conditions, it is more natural to consider a boundary norm consistent with the regularity of the exact solution. For example, if Dirichlet data $u\big|_{\partial D}\in H^{1/2} \subsetneq L^{2}(\partial D)$, as is the case when $u \in H^{1}(D)$, merely penalizing boundary misfit via $L^{2}(\partial D)$ may be insufficient to enforce the boundary condition with the appropriate level of regularity. A stricter penalty term such as $\|u\big|_{\partial D}\|^2_{H^1(\partial D)}$ can better incorporate theoretical requirements of the problem and lead to improved numerical performance.

Even with an appropriate choice of norm, the penalty-based approach remains suboptimal. Since it relaxes the boundary condition rather than enforcing it exactly, tuning the penalty parameter $\lambda$ becomes \revision{as mentioned in the above subsection.}\revision{
% \st{While a larger $\lambda$ can improve boundary accuracy and theoretically lead to a better solution, it often slows convergence and may be catastrophic to the obtained solution by overfitting the boundary condition.}
} Considering this, alternative approaches to enforcing Dirichlet boundary conditions beyond the penalty method are of great interest \citep{lagaris1998artificial,SUKUMAR:2022}.
For the homogeneous Dirichlet case, the solution can be parameterized as
\begin{equation}
\label{eq:homogenizor}
  u(x) \approx \gamma(x) \sum_{n=1}^N c_n \varphi(x; \omega_n) 
\end{equation}
where $\gamma \in C^2_0(D)$ with $\gamma(x) > 0$ for \(x \in D\) is a twice continuously differentiable function that vanishes on the boundary $\partial D$. This formulation preserves the theoretical guarantees discussed earlier, provided that  $u(x)/\gamma(x)$ satisfies the required regularity. 
More generally, the method works for any Dirichlet boundary function $f:\partial D \rightarrow \R$ under the assumption that $f$ admits an extension $\bar{f}: D\rightarrow \R, \bar{f}\big|_{\partial D} = f$ that is twice continuously differentiable. Since such an extension may only exist for smooth domains with $f$ at least twice continuously differentiable, this limits this approach in practice to functions $f$ where such an extension can be easily constructed.
% \kp{I have rewritten this to add the $C^2$ requirement.}\zs{You're right, thanks for pointing out. The extension needs to $C^{2}$ to be used. I think in practice this is not a problem since such extension is obtained by interpolation/data-driven methods anyway. We may consider moving these details into the appendix as a future research directions as we haven't done any implementation yet.}
We may then set $u(x) = f(x) + \gamma(x)\cN(x)$. The extension function $\bar{f}$ and the vanishing function $\gamma$ can often be constructed in low-dimensional domains with simple geometry. However, for complex or high-dimensional geometries, finding such functions becomes significantly more challenging. While there are interesting data-driven approaches, their investigation lies beyond the scope of this work and is left for future research.

% \xt{TODO: we may add discussions on $\lambda$ here}\zs{Moved this to the previous sub-section}.

\subsection{Additional implementation components}
We discuss a few additional components for practical convenience here.
% \subsubsection{Constraining $\omega$}
\subsubsection{\texorpdfstring{Constraining $\omega$}{Constraining omega}}
During each Gauss–Newton step, the parameter $\omega = (y, \sigma)$ may fall outside the prescribed parameter space $\Omega$. To ensure that $\sigma$ remains within  $\Omega_{\sigma} = (0, \sigma_{\max}]$, we introduce a parameterization 
\begin{equation}
\label{eq:sigma_param}
\begin{aligned}
\sigma(s) = \sigma_{\min} + (\sigma_{\max} - \sigma_{\min}) \tau(s),
\end{aligned}
\end{equation}
where $\tau$ is any increasing bijection from $\R$ to $(0, 1)$, typically chosen as the sigmoid function $\tau(s) = 1/ (1 + \exp(-s))$. The lower bound $\sigma_{\min}$ is usually set to a small positive value (e.g. $10^{-3}$) rather than $0$ to safeguard against numerical instability caused by excessively small $\sigma$, though in practice $\sigma$ remains well above this lower bound during training. Occasional out-of-bounds $y$ is not problematic and rarely observed in practice.

\subsubsection{Stopping criterion} The iteration in Algorithm \ref{alg:main} may be terminated when neither the optimization of existing kernel nodes nor the insertion of new nodes results in a significant descent of the loss function. Please see \Cref{app:algo} for more details.

\section{Numerical experiments}
\label{sec:experiments}

In this section, we conduct numerical experiments to demonstrate the effectiveness and versatility of the proposed method. We begin with a semilinear Poisson equation (example in \Cref{subsec:1.3}) as a simple testbed of the proposed method, including exploratory studies on higher-dimensional PDEs and treatment of boundary constraints. Next, we use our method to solve a regularized Eikonal equation, where we also investigate the convergence behavior of numerical solutions to the unique viscosity solution. Finally, we consider a spatial-temporal viscous Burgers’ equation, where our method is coupled with a time discretization scheme. 

We begin by outlining some common settings and evaluation metrics used across all experiments.

\subsection*{Settings and evaluation metrics}

In all numerical experiments, we use the following generic settings:
\begin{itemize}
    \item We set $s = d + 2 + 0.01$ in the feature function, i.e., 
    \[
        \varphi(x;\omega) = \frac{\sigma^{2.01}}{\left(\sqrt{2\pi}\right)^d} 
    \exp\left(- \frac{\norm{x - y}^2_2}{2 \sigma^2}\right).
    \]
    This fulfills the requirement for theory developed in \Cref{sec:theory}.
    \item We set $\tau(s) = 1/(1+\exp(-s))$ and $(\sigma_{\min}, \sigma_{\max}) = (10^{-3}, 1)$ in all experiments in the parameterization of $\sigma$ given by \eqref{eq:sigma_param}. 
    \item $M = 10^{4}$ candidates feature functions will be sampled from $\Omega$ in Phase I of each iteration. 
    \item We use an empty network as the initialization of Algorithm \ref{alg:main} unless otherwise specified.
\end{itemize}
In addition, all algorithmic components are implemented in Python using JAX\footnote{JAX is used here mainly for auto-differentiation, yet all derivatives can also be directly implemented from their analytical forms.}, executed on a 2023 Macbook Pro with Apple M2 Pro chip (16GB memory) using CPU-only computation. Training times typically range from 30 seconds to several minutes. In all the following experiments, double-precision (float64) arithmetic is used. However, the method is, in general, compatible with single-precision (float32) computations as well.

We evaluate the numerical solution $\hat{u}$ using several metrics (see \Cref{tab:metrics_summary}). These include the $L^2$ and $L^\infty$ errors with respect to the true solution $u$, both estimated on a finer test grid. To assess the impact of the regularization term in preventing overfitting, we also report the empirical loss $\hat{L}$, computed on both the training collocation points and the test grid. In addition, we record the number of kernels used in the solution to indicate the level of sparsity.

\begin{table}[t]
\centering
\caption{Evaluation Metrics}
\begin{tabular}{cl}
\toprule
\textbf{Metric} & \textbf{Definition} \\
\midrule

$L^2$ error  &
$L^{2}$ error between numerical and exact solutions estimated on test grid. \\
$L^\infty$ error  &
Maximum absolute error between numerical and exact solutions on test grid. \\
$\hat{L}$ (Train) &
Empirical loss evaluated on training collocation points. \\
$\hat{L}$ (Test) &
Empirical loss evaluated on test grid.  \\
\#Kernels &
Number of kernels/feature functions used in the numerical solution. \\
\bottomrule
\end{tabular}
\label{tab:metrics_summary}
\end{table}

% \footnote{Code and data that allow readers to reproduce the results in this paper are available at \href{https://github.com/EddyShao/SparseKernelPDE}{https://github.com/EddyShao/SparseKernelPDE} \xt{Zihan, I suggest we publish the code after this work is published/accepted.}}\zs{Ok. Thanks}. 

\subsection{Semilinear Poisson equation}
\label{subsec:twobumps}
We use a semilinear Poisson equation in \eqref{eq:semilinear} as a simple testbed for our method. We begin with a 2D equation where the exact solution features two steep bumps. We then test our approach in a 4-dimensional problem to explore its capability in handling higher-dimensional equations. Finally, we discuss treatment of boundary conditions, including the mask function technique introduced in the \Cref{subsec:bnd_treat}. In all following numerical experiments, we prescribe $D = [-1, 1]^{d} \subseteq [-2, 2]^{d} = \Omega$, where $d$ is $2$ or $4$.

\subsubsection{2D equation} 
% We prescribe exact solution $u: D\subseteq\R^{2}\rightarrow \R$ as
% \begin{equation}
%      u(x)  = v(x; \ c_1, k_1, R_1) + v(x; \ c_2, k_2,  R_2)
% \end{equation}
% where
% \begin{equation}
%     \begin{aligned}
%         v(x; \ c, k, R) &= \tanh\left( k \left( R_0 - \| x - c \|_2 \right) \right) + 1, \\
%         \text{with} \quad
%         \left\{
%         \begin{aligned}
%             c_1 = -c_2 &= (0.3,\ 0.3), \\
%             (k_1,\ k_2) &= (4,\ 12), \\
%             (R_1,\ R_2) &= (0.30,\ 0.15)
%         \end{aligned}
%         \right.
%     \end{aligned}
% \end{equation}
% and $f_\cE$ and $f_{\cB}$ in the formulation of Example in \Cref{sec:theory} are computed accordingly.

We prescribe the exact solution $u : D \subseteq \mathbb{R}^2 \rightarrow \mathbb{R}$ as
\begin{equation}
u(x) = v(x; c_1, k_1, R_1) + v(x; c_2, k_2, R_2),
\end{equation}
where 
% \begin{equation}
\begin{align*}
v(x; c, k, R) = \tanh\left( k \left( R_0 - | x - c | \right) \right) + 1, \\\\
\left\{
\begin{aligned}
c_1 &= (0.3,\ 0.3), \\
c_2 &= -c_1, \\
(k_1,\ k_2) &= (4,\ 12), \\
(R_1,\ R_2) &= (0.30,\ 0.15)
\end{aligned}
\right. .
\end{align*}
% \end{equation}
The source terms $f_{\mathcal{E}}$ and $f_{\mathcal{B}}$ in the formulation of the example in \Cref{sec:theory} are then computed accordingly. Here $v$ is constructed as a bump with center at $c\in\R^{2}$ and steepness parameterized jointly by $R$ and $k$. The exact solution $u$ then possesses two bumps with distinct steepness (see Figure \ref{fig:two_bump_exact}). In particular, the multi-scale steepness poses difficulties for the GP method, where the position of the RBF kernel is prescribed with a fixed shape parameter. To see this, we tested GP method with different shape parameter ($\sigma$ in $\kappa(x, x') = \exp(- \frac{\|x - x'\|_{2}^{2}}{2 \sigma^{2}})$) and observed significant volatility in the results (see Figure \ref{fig:two_bump_GP}). Moreover, different kernel widths might attain maximum absolute error at the center of $2$ steepness bumps, respectively, indicating that the single fixed kernel shape parameter $\sigma$ faces the challenges of handling distinct steepness simultaneously. 

\begin{figure}[t]
    \centering
    % First row (3 figures)
    \begin{subfigure}[t]{0.40\textwidth}
        \includegraphics[width=\textwidth]{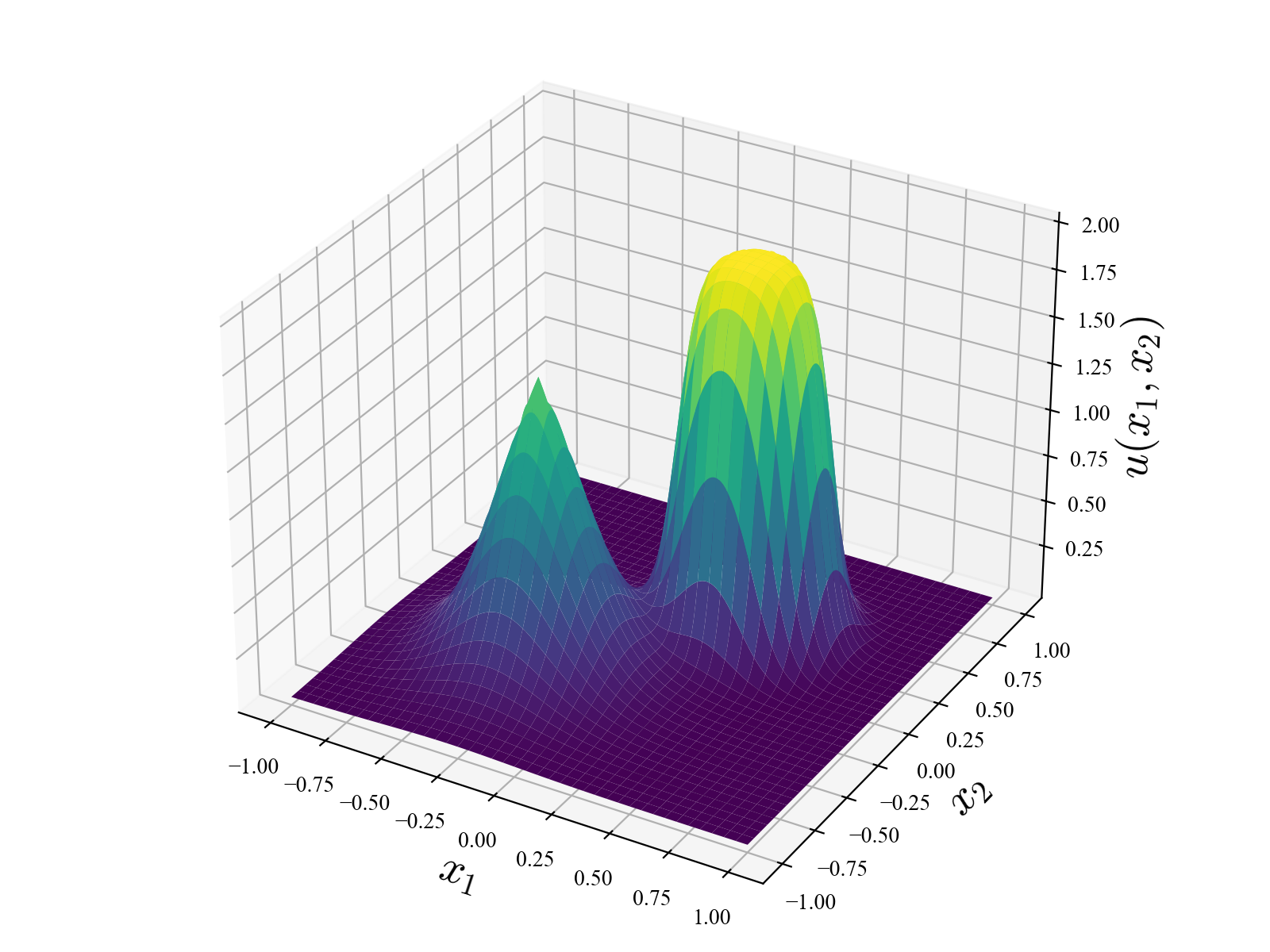}
        % \caption{Exact solution}
        \caption{}
        \label{fig:two_bump_exact}
    \end{subfigure}
    % \hfill
    \begin{subfigure}[t]{0.55\textwidth}
        \includegraphics[width=\textwidth]{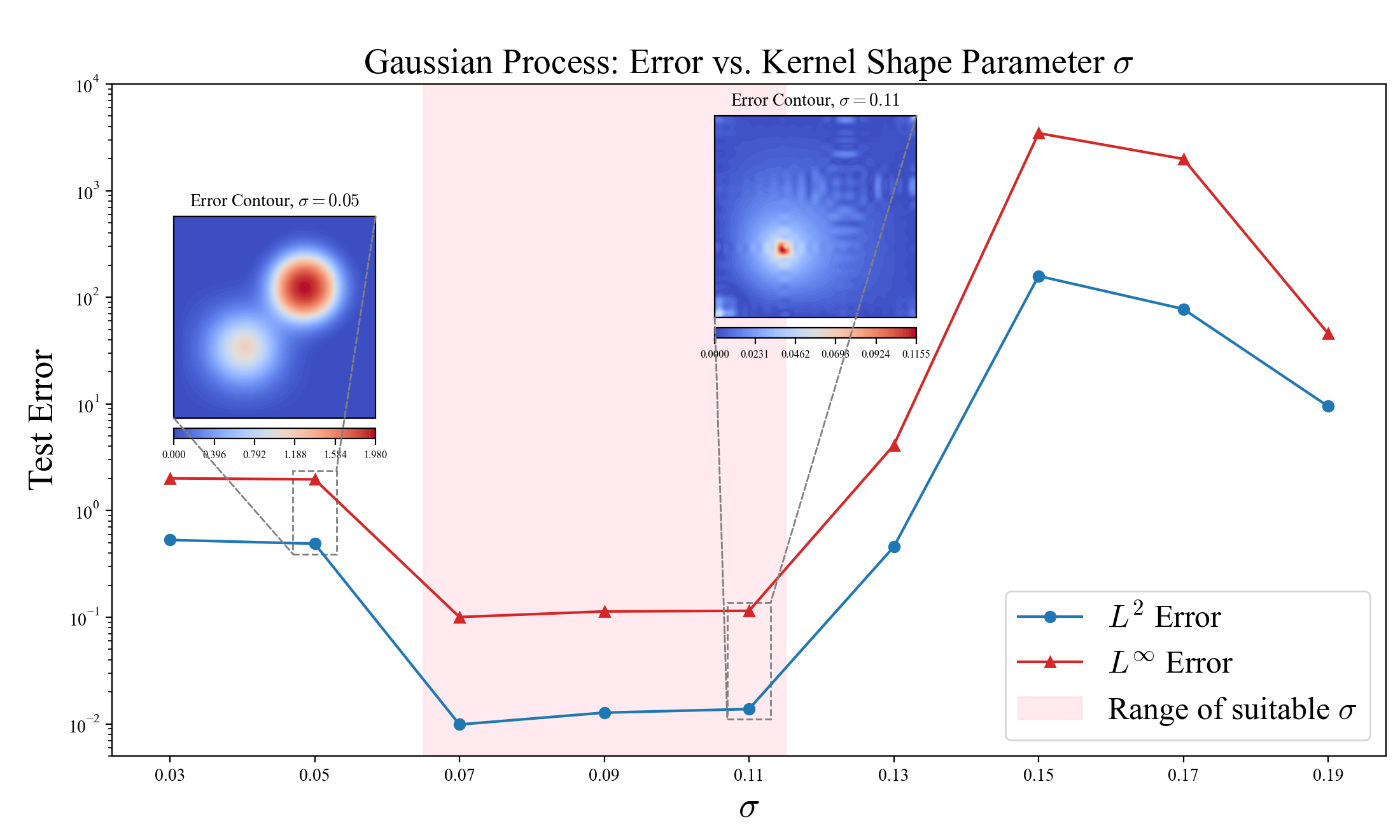}
        \caption{}
        % \caption{Error of solutions obtained by  the GP method with different kernel shape parameters. }
        \label{fig:two_bump_GP}
    \end{subfigure}
    % \hfill
    % \begin{subfigure}[b]{0.35\textwidth}
    %     \includegraphics[width=\textwidth]{plots/semilinear/linesearch_GP.png}
    %     \caption{Caption 3}
    % \end{subfigure}
    \vskip 1em
    % Second row (2 centered figures)
    \begin{subfigure}[t]{0.42\textwidth}
        \includegraphics[width=\textwidth]{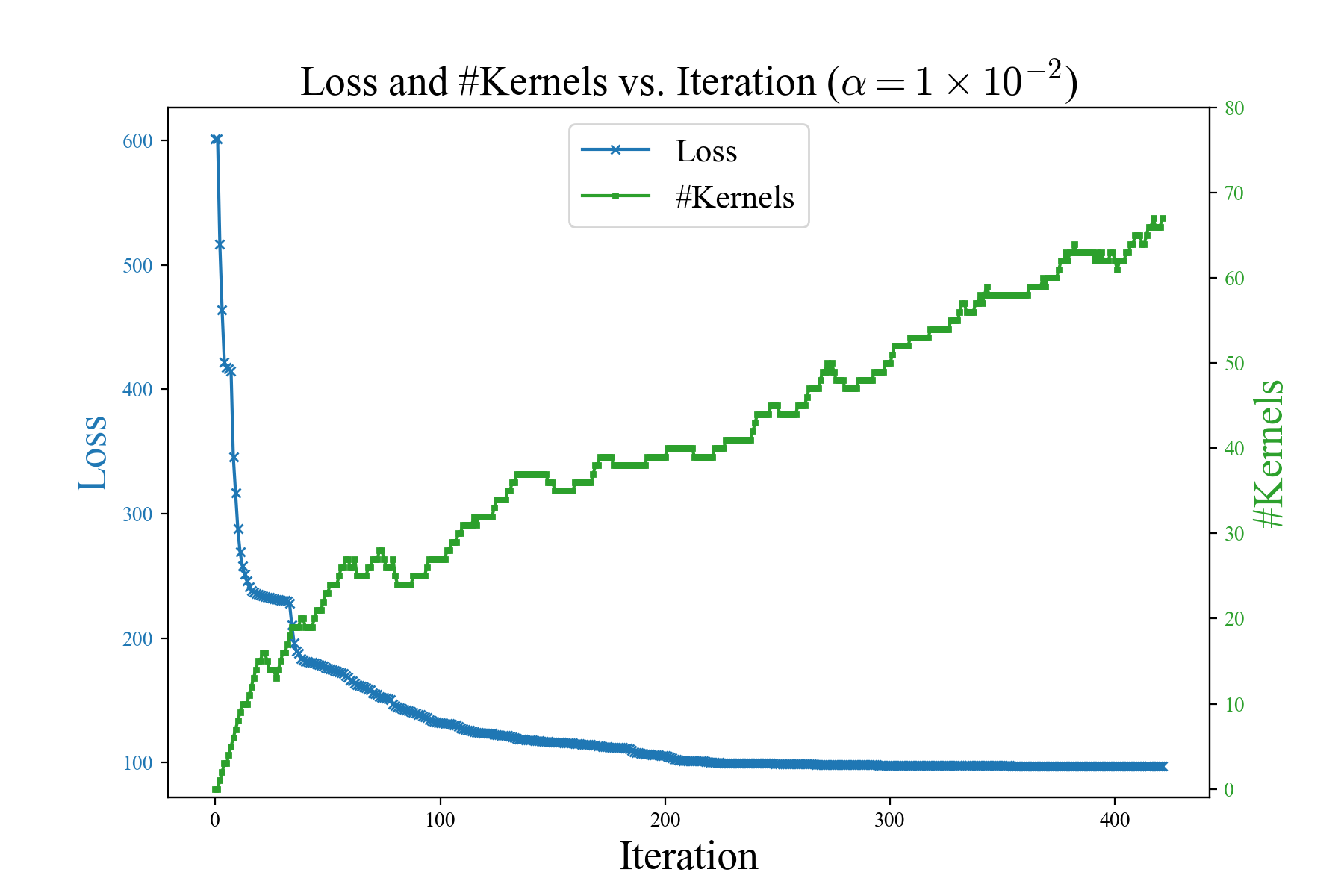}
        % \caption{Convergence of loss and growth of number of kernels ($\alpha = 10^{-2}$).}
        \caption{}
        \label{fig:two_bump_convergence}
    \end{subfigure}
    % \hfill
    \begin{subfigure}[t]{0.56\textwidth}
        \includegraphics[width=\textwidth]{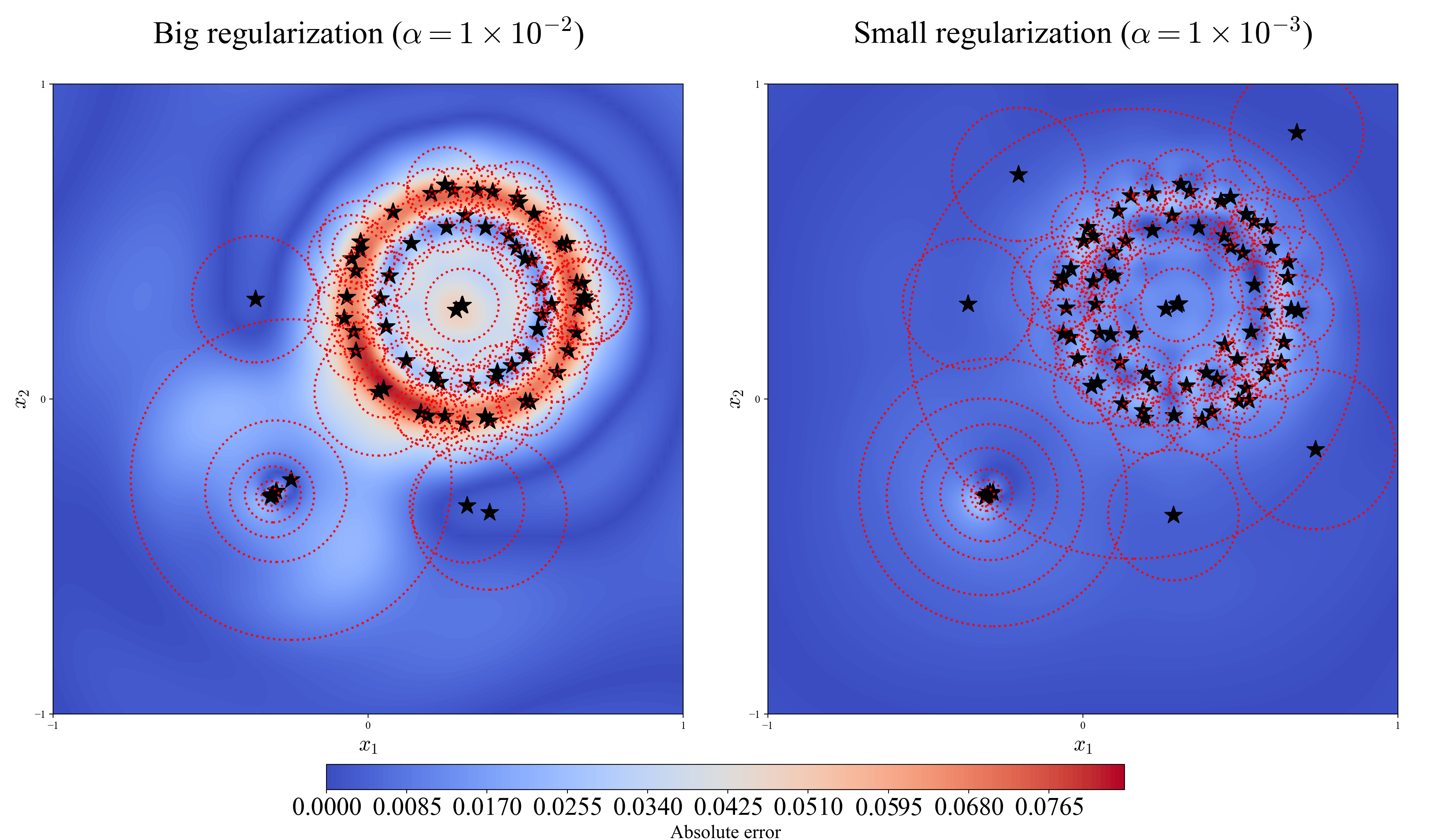}
        % \caption{Error contour of proposed method. Each $\star$ represents the location parameters $y_n$ of learned kernel nodes, and the dotted circle around which are of radius $\sigma_n$ (very few kernel nodes outside of $D$ are omitted).}
        \caption{}
        \label{fig:two_bump_contour}
    \end{subfigure}

    \caption{Numerical results using Sparse RBFNet (our) and Gaussian Process (GP) on solving a semilinear Poisson equation: (a) Exact solution; (b) Error of solutions obtained by  the GP method with different kernel shape parameters; (c) Convergence of loss and growth of number of kernels (\revision{regularization parameter} $\alpha = 10^{-2}$); (d) Error contours of our method. Each $\star$ represents the location parameter $y_n$ of a learned kernel node, and the dotted circle around has radius $\sigma_n$ (very few kernel nodes outside of $D$ are omitted). }
    \label{fig:semilinear_twobump}
\end{figure}

We then solve the PDE using proposed method with penalty parameter $\lambda = 10^3$ in the formula of $\hat{L}$ to enforce the boundary condition and regularization parameter $\alpha=10^{-2}, 10^{-3}$. Convergence of loss and growth of the number of kernels are shown in \Cref{fig:two_bump_convergence}. In practice, we solve case with $\alpha = 10^{-3}$ using the solution obtained with $\alpha=10^{-2}$ as initialization. This strategy largely saves computational time and results in better performance occasionally (see \Cref{app:pretraining} for more details).
% \footnote{Case with $\alpha = 10^{-3}$ is initialized as the results of $\alpha = 10^{-2}$ case \zs{Do we need to say it in the main body? This is necessary for this example but not used in the rest two examples.}}. 
The error contours are shown in \Cref{fig:two_bump_contour}. We also visualized the positions and shape parameters of learned kernel nodes $\{\omega_{n}\}$, observing that most of the kernels cluster at the location of the sharp transition. This shows our algorithm can capture localized structure and sharp transitions around the bumps. 

We further test the proposed method with a wider range of $\alpha$ and number of collocation grid points $K$. We compare the results with the GP method with a range of kernel shape parameter $\sigma$, evaluated with metrics listed in \Cref{tab:metrics_summary}. Results are shown in \Cref{tab:two_bumps}. We see that our method effectively finds a sparse and accurate solution without searching optimal $\sigma$ as necessary in Gaussian Process method. Meanwhile, decreasing $\alpha$ does not always lead to improved solution accuracy, especially when collocation points are relatively sparse in the domain. In such cases, the empirical loss $\hat{L}$ on the test grid is drastically larger than its estimate on the collocation points, suggesting overfitting. This shows the significant role of the regularization term in preventing overfitting, in addition to obtaining a sparse representation of solutions.

\begin{table}[t]
\centering
\caption{Errors, residual loss, and number of kernels for Sparse RBFNet and GP using grid collocation points. Case with $\alpha = 10^{-3}$($10^{-4}$) is solved with solution obtained by $\alpha=10^{-2}$($10^{-3}$) as initialization. Test results are evaluated on a $100\times 100$ grid in $D$. All results are averaged over 10 runs.}
\setlength{\tabcolsep}{4pt}
\begin{tabular}{ccccc@{\hskip 18pt}ccc}
\toprule
& & \multicolumn{3}{c}{\textbf{Sparse RBFNet}} & \multicolumn{3}{c}{\textbf{Gaussian Process}} \\
\cmidrule(lr{2em}){3-5} \cmidrule(lr{0.15em}){6-8}
\shortstack{$K$ \\ $(K_1, K_2)$} & Metric
& $\alpha=1\text{e}{-2}$ & $\alpha=1\text{e}{-3}$ & $\alpha=1\text{e}{-4}$
& $\sigma=0.05$ & $\sigma=0.10$ & $\sigma=0.15$ \\
\specialrule{0.9pt}{1pt}{1pt}
\multirow{6}{*}{\shortstack{$20$ \\ $(324, 76)$}} 
& $L^2$ error        & $1.05\text{e}{-1}$ & $8.94\text{e}{-2}$ & $8.42\text{e}{-2}$ & $5.29\text{e}{-1}$ & $3.89\text{e}{-2}$ & $3.62\text{e}{-2}$\\
& $L^\infty$ error   & $2.52\text{e}{-1}$ & $2.23\text{e}{-1}$ & $2.24\text{e}{-1}$ & $2.02\text{e}{+0}$ & $2.64\text{e}{-1}$ & $2.41\text{e}{-1}$ \\
& $\hat{L}$(train)   & $4.78\text{e}{+0}$ & $7.18\text{e}{-2}$ & $1.69\text{e}{-3}$ & -- & -- & -- \\
& $\hat{L}$(test)    & $6.56\text{e}{+1}$ & $5.89\text{e}{+1}$ & $5.82\text{e}{+1}$ & -- & -- & -- \\
& \#Kernels          & 46  & 65  & 86  & 400 & 400 & 400 \\
\midrule
\multirow{6}{*}{\shortstack{$30$ \\ $(784, 116)$}} 
& $L^2$ error        & $4.99\text{e}{-2}$ & $1.98\text{e}{-2}$ & $1.77\text{e}{-2}$ & $4.89\text{e}{-1}$ & $1.32\text{e}{-2}$ & nan \\
& $L^\infty$ error   & $9.04\text{e}{-2}$ & $5.35\text{e}{-2}$ & $4.87\text{e}{-2}$ & $1.96\text{e}{+0}$ & $1.14\text{e}{-1}$ & nan \\
& $\hat{L}$(train)   & $1.50\text{e}{+1}$ & $2.41\text{e}{-1}$ & $6.00\text{e}{-3}$ & -- & -- & -- \\
& $\hat{L}$(test)    & $1.77\text{e}{+1}$ & $6.81\text{e}{+0}$ & $6.88\text{e}{+0}$ & -- & -- & -- \\
& \#Kernels          & 68  & 83  & 102  & 900 & 900 & 900 \\
\midrule
\multirow{6}{*}{\shortstack{$50$ \\ $(2304, 196)$}} 
& $L^2$ error        & $4.37\text{e}{-2}$ & $1.22\text{e}{-2}$ & $1.08\text{e}{-2}$ & $6.40\text{e}{-3}$ & nan & $6.53\text{e}{-1}$  \\
& $L^\infty$ error   & $8.47\text{e}{-2}$ & $4.51\text{e}{-2}$ & $4.52\text{e}{-2}$ & $7.07\text{e}{-2}$ & nan & $1.81\text{e}{+0}$ \\
& $\hat{L}$(train)   & $1.82\text{e}{+1}$ & $5.57\text{e}{-1}$ & $3.73\text{e}{-2}$ & -- & -- & -- \\
& $\hat{L}$(test)    & $1.75\text{e}{+1}$ & $2.15\text{e}{+0}$ & $1.59\text{e}{+0}$ & -- & -- & -- \\
& \#Kernels          & 72  & 104 & 125 & 2500 & 2500 & 2500\\
\specialrule{1pt}{1pt}{1pt}
\end{tabular}
\label{tab:two_bumps}
\end{table}

We remark that grid collocation points are used in the above experiments. Using random collocation points negatively affect the performance of both methods (See \Cref{tab:two_bumps_uniform} in \Cref{app:collocation}). In general, the performance of our method is unfortunately sensitive to the uniformity of the collocation points. This negative impact is attributed to the increased gap between $\nu_{D}^{K_1}$, $\nu_{\partial D}^{K_2}$ and  $\nu_{D}$, $\nu_{\partial D}$, which further corrupts the empirical residual loss.

% \kp{We explain this by the smaller size of gaps in between the residual collocation points in the regular grids.}

\subsubsection{4D equation}
We now consider the same equation in 4D as an exploratory example of solving higher-dimensional problems with the proposed method. We set $D=[-1, 1]^{4}\subseteq [-2, 2]^{4}$ and exact solution $u$ be prescribed as 
\begin{equation}
    u(x) = \prod_{d=1}^{4} \sin(\pi x_d)
\end{equation}
with source term and boundary conditions computed accordingly. 

We fix $\lambda=3000$ use grid collocation points $K = 6^4, 8^4$ with $\alpha = 10^{-3}, 10^{-4}$ respectively. The results are shown in Table \ref{tab:highdim}. While obtaining sparse and accurate solutions, our method effectively prevents overfitting. When there is not a sufficient number of collocation points (e.g., $K = 1296$), decreasing regularization parameter $\alpha$ might lead to severe overfitting ($\hat{L}_{\text{test}} \gg \hat{L}_{\text{train}}$). This results in degraded accuracy even though significantly more kernel nodes are inserted compared to the case with a larger $\alpha$.

We highlight several key advantages of our method for solving equations in high-di\-mensional spaces compared with other methods.  First, by employing a shallow network with smooth feature functions instead of a multi-layer perception (MLP) as is often used in other PINN-based methods, all required derivatives can be computed analytically. This is especially beneficial when computing higher-order derivatives, as it eliminates the need to construct large computational graphs for automatic differentiation or to rely on Monte Carlo-based estimations, thereby reducing both memory usage and computational cost \citep{sirignano2018dgm}. Second, our method offers greater flexibility compared to GP methods. In a GP formulation, the number of trainable/optimizable parameters scales as $\cO(K)$, resulting in substantial computational cost for operations such as Cholesky factorization of a matrix of size $\cO(K) \times \cO(K)$. Since a large number of collocation points $K$ is typically required to learn complex solutions in high-dimensional spaces, this poses a significant computational bottleneck and requires the use of mini-batch techniques~\citep{yang2023mini, tran2020stochastic}. In contrast, our method inserts new kernel nodes/feature functions only when necessary, thereby keeping the number of trainable (optimizable) parameters manageable, usually much fewer than collocation points. Lastly, the regularization term in our method effectively prevents overfitting, which is particularly important when collocation points are relatively sparse in the domain, as this is often the case in high-dimensional problems. \revision{We emphasize that, however, these advantages are most pronounced  where the solution admits additional low-complexity structure (e.g., localized features or low effective dimensionality), so that accurate approximations can be achieved with relatively few active kernels. As with other kernel-based methods, we do not claim to overcome the curse of dimensionality in fully generic unstructured settings.}

% \begin{table}[h]
% \centering
% \caption{Training and testing errors for a high-dimensional Poisson equation. Test results are evaluated on grid of $20^4$ in $D$.}
% \setlength{\tabcolsep}{4pt}
% \begin{tabular}{ccccc}
% \toprule
% $K$ $(K_1, K_{2})$ & \multicolumn{2}{c}{$1296$ $(256, 1040)$} & \multicolumn{2}{c}{$4096$ $(1296, 2800)$} \\
% \cmidrule(lr){2-3} \cmidrule(lr){4-5}
% $\alpha$ & $10^{-3}$ & $10^{-4}$ & $10^{-3}$ & $10^{-4}$ \\
% \midrule
% $L^2$ error      & $1.74\text{e}{-1}$ & $2.14\text{e}{-1}$ & $1.46\text{e}{-1}$ & $5.98\text{e}{-2}$ \\
% \addlinespace[4pt]
% $L^\infty$ error  & $2.50\text{e}{-1}$ & $4.48\text{e}{-1}$ & $1.87\text{e}{-1}$ & $8.17\text{e}{-2}$ \\
% $\hat{L}$(train)  & $4.10\text{e}{+0}$ & $3.99\text{e}{-1}$ & $7.01\text{e}{+0}$ & $9.83\text{e}{-1}$ \\
% $\hat{L}$(test)   & $9.89\text{e}{+0}$ & $1.47\text{e}{+1}$ & $6.19\text{e}{+0}$ & $1.40\text{e}{+0}$ \\
% \#Kernels          & $192$ & $296$ & $221$ & $379$ \\
% \bottomrule
% \label{tab:highdim}
% \end{tabular}
% \end{table}

\begin{table}[t]
\centering
\caption{Errors and residual loss of numerical solutions to a 4D semilinear Poisson equation using grid collocation points. Test results are evaluated on a grid of $20^4$ in $D\subseteq \R^{4}$. All results are averaged over 10 runs.}
\setlength{\tabcolsep}{4pt}
\begin{tabular}{ccccc}
\toprule
$K$ $(K_1, K_{2})$ & \multicolumn{2}{c}{$1296$ $(256, 1040)$} & \multicolumn{2}{c}{$4096$ $(1296, 2800)$} \\
\cmidrule(lr){2-3} \cmidrule(lr){4-5}
$\alpha$ & $10^{-3}$ & $10^{-4}$ & $10^{-3}$ & $10^{-4}$ \\
\midrule
$L^2$ error      & $1.75\text{e}{-1}$ & $2.31\text{e}{-1}$ & $1.45\text{e}{-1}$ & $5.71\text{e}{-2}$ \\
$L^\infty$ error & $2.39\text{e}{-1}$ & $4.55\text{e}{-1}$ & $1.90\text{e}{-1}$ & $7.69\text{e}{-2}$ \\
$\hat{L}$(train) & $3.93\text{e}{+0}$ & $3.60\text{e}{-1}$ & $7.06\text{e}{+0}$ & $9.19\text{e}{-1}$ \\
$\hat{L}$(test)  & $1.03\text{e}{+1}$ & $1.78\text{e}{+1}$ & $6.27\text{e}{+0}$ & $1.30\text{e}{+0}$ \\
\#Kernels        & $173$ & $239$ & $204$ & $365$ \\
\bottomrule
\label{tab:highdim}
\end{tabular}
\end{table}

\subsubsection{Boundary conditions treatments} 
\label{subsec:num_bnd_treat}
We now offer a numerical experiment on the discussion of the treatment of boundary conditions in \Cref{subsec:bnd_treat}.
\begin{figure}[t]
    \centering
    % First subfigure: 1/4 width
    \begin{subfigure}[t]{0.45\textwidth}
        \centering
        \includegraphics[width=\textwidth]{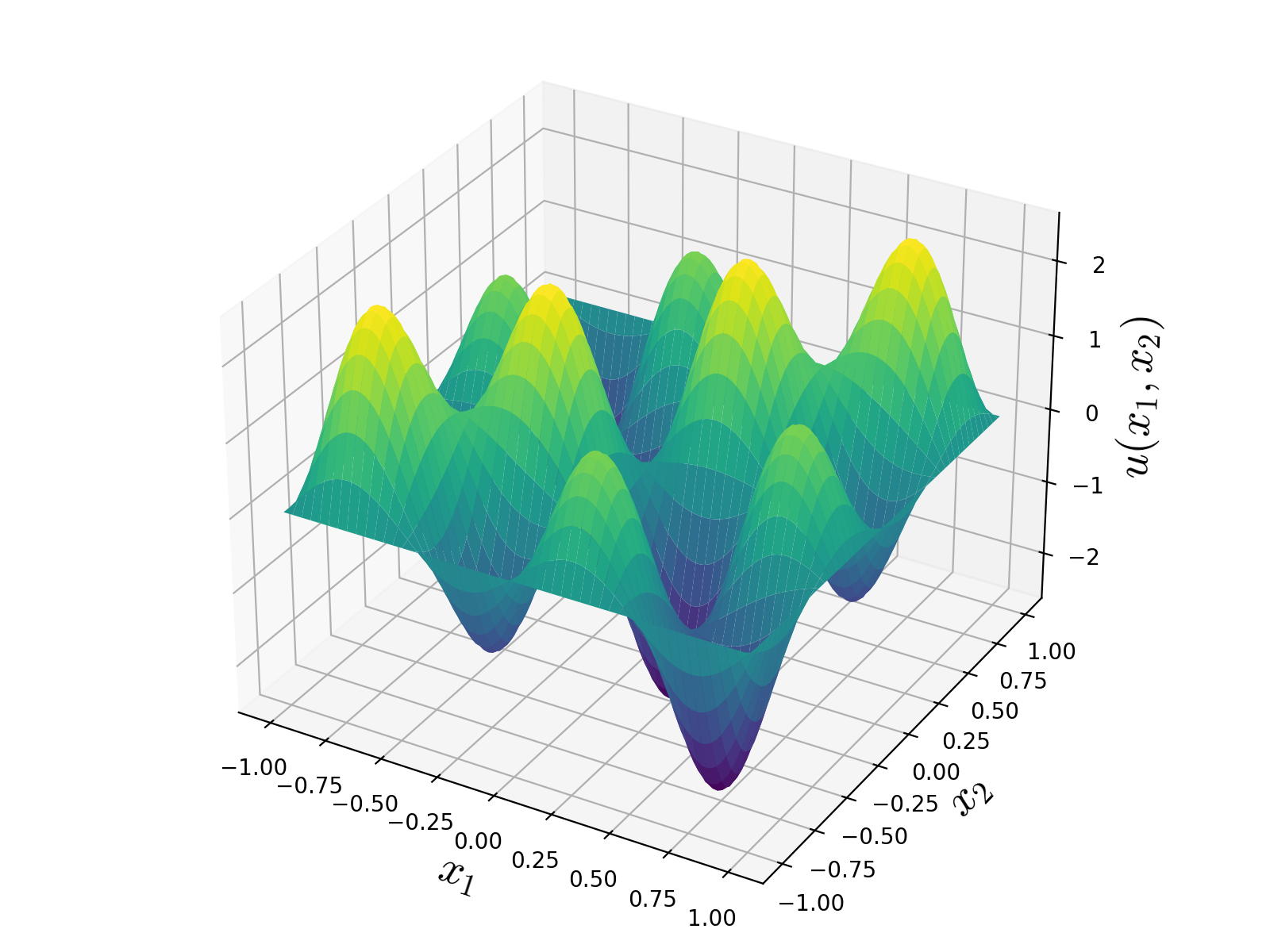}
        % \caption{Exact solution $u(x_1, x_2)$}
        \caption{}
        \label{fig:sines_exact}
    \end{subfigure}
    \hspace{0.03\textwidth}
    % Second subfigure: 1/4 width
    \begin{subfigure}[t]{0.50\textwidth}
        \centering
        \includegraphics[width=\textwidth]{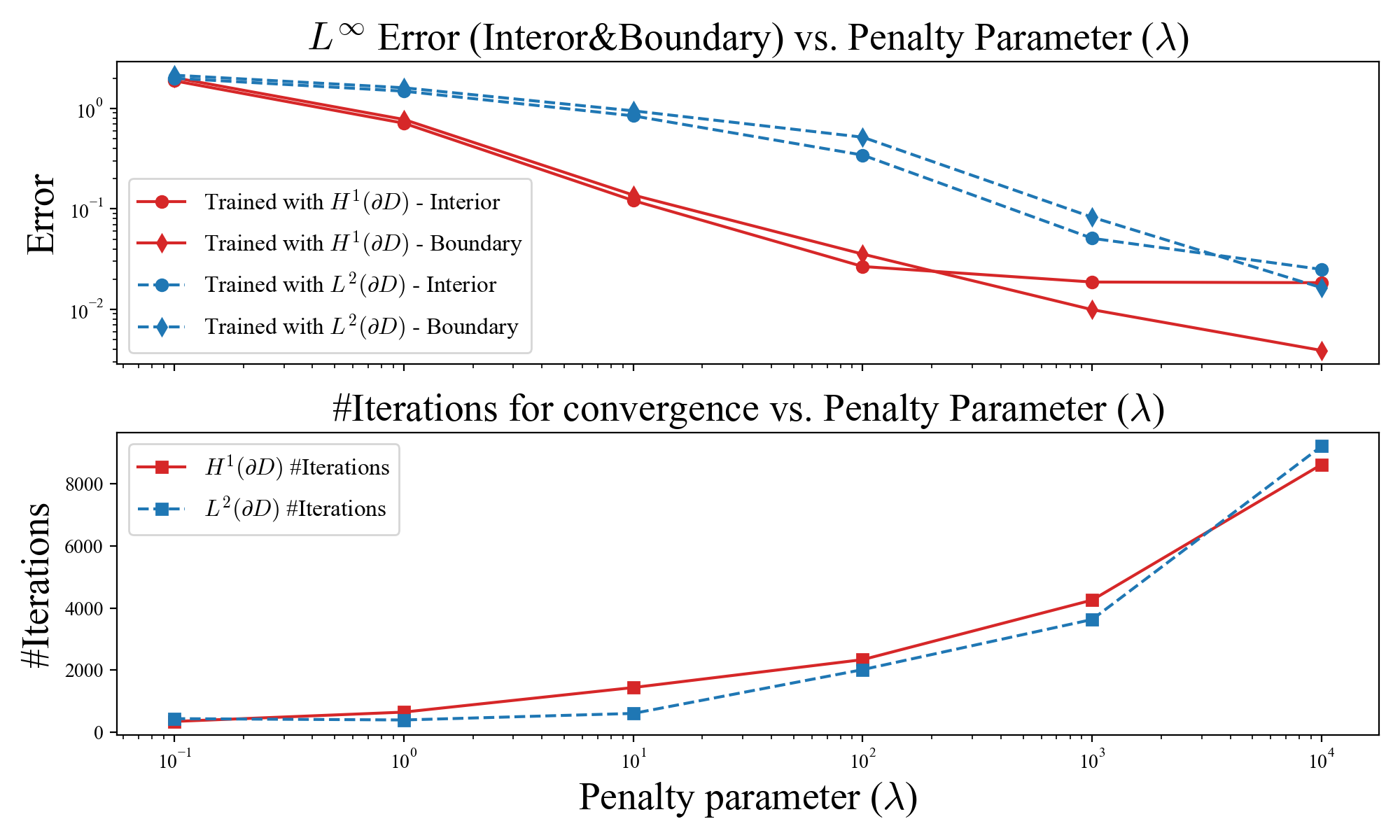}
        % \caption{Error and convergence speed - penalty coefficient $\lambda$ (results are averaged from 10 runs) }
        \caption{}
        \label{fig:sines_error_lambda}
    \end{subfigure}
    % Third subfigure: 1/2 width
    \vskip 1em
    \begin{subfigure}[t]{0.99\textwidth}
        \centering
        \includegraphics[width=\textwidth]{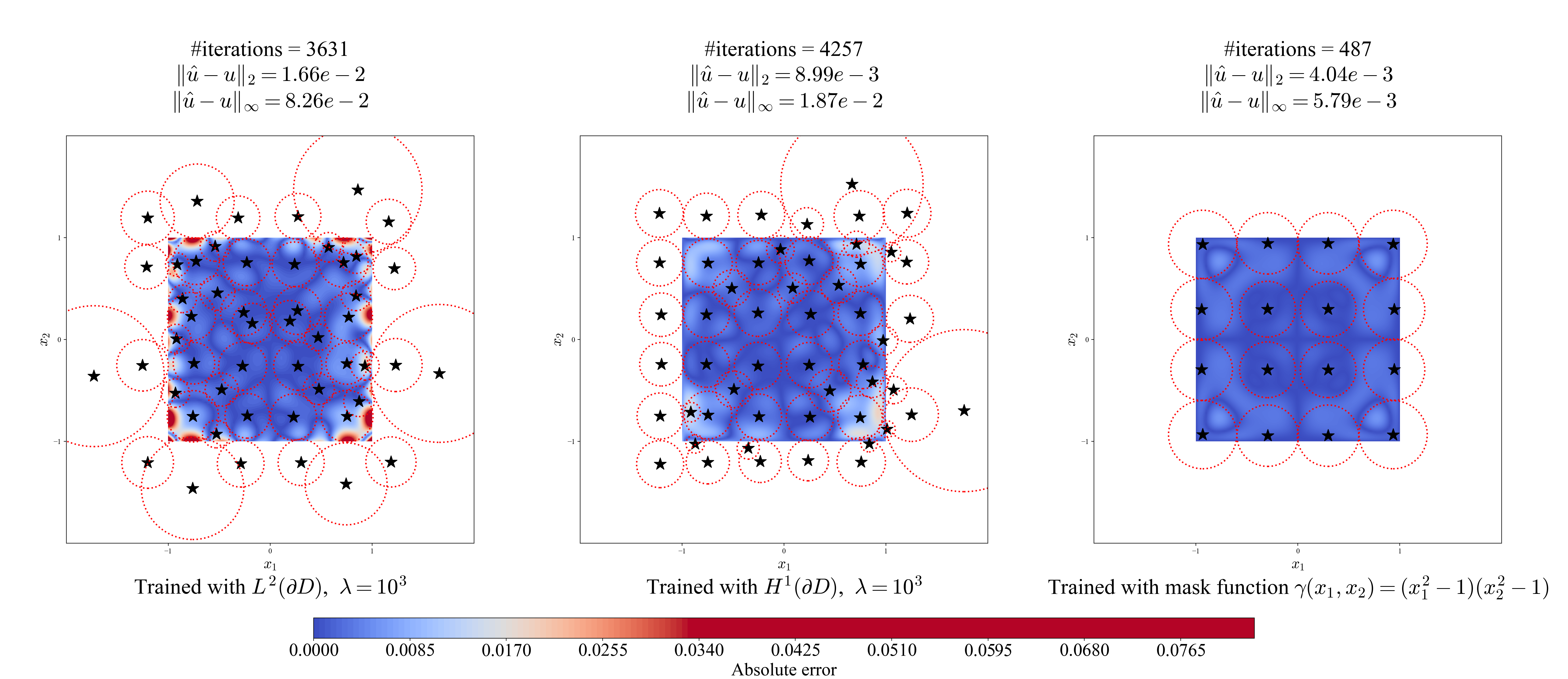}
        % \caption{Error contour of solutions with/without ``mask" function (one realization). Errors and iterations shown are averaged from 10 runs.}
        \caption{}
        \label{fig:error_contour_sines}
    \end{subfigure}
    \caption{Numerical results of different treatments of boundary conditions. (a) Exact solution $u(x_1, x_2)$; (b) Error and convergence speed -- penalty coefficient $\lambda$ (results are averaged from 10 runs); (c) Error contour of solutions with $L^{2}$-type loss, $H^{1}$-type loss, and ``mask" function for boundary conditions. Each $\star$ represents the location parameter $y_n$ of a learned kernel node, and the dotted circle around has radius $\sigma_n$. The contour plots are taken from a single run, while the reported number of iterations, $L^2$ error and $L^{\infty}$ error are averaged over 10 runs.}
    \label{fig:three-wide}
\end{figure}
We now consider the same equation in $2D$ but with exact solution prescribed as 
\[
 u(x_1, x_2) = \sin(\pi x_1) \sin(\pi x_2) + 2 \sin(2\pi x_1) \sin(2\pi x_2).
\]
This function was selected for its nontrivial boundary behavior (see \Cref{fig:sines_exact}), making it a suitable testbed for various boundary condition treatments. We set $\alpha = 10^{-3}$, $K = 400$ grid points and solve the equation with $\lambda  = 10^{-1}, 10^{0}, 10^{1}, 10^{2}, 10^{3}, 10^{4}$. We use both the $L^{2}(\partial D)$-type loss and $H^{1}(\partial  D)$-type loss mentioned in \Cref{subsec:bnd_treat} to enforce boundary conditions. As is shown in \Cref{fig:sines_error_lambda}, increasing $\lambda$ leads to higher-quality solutions. Notably, solutions trained with $H^{1}$-type loss achieve comparable accuracy to $L^{2}$-type loss with a much smaller $\lambda$, which is of practical significance since a larger $\lambda$ requires substantially more iterations for convergence. 

We remark that a larger $\lambda$ does not always lead to improved solutions in practice, even when sufficient iterations are allowed for convergence. A key contributing factor is the potential overfitting to boundary data, as the regularization effect of the $\alpha \|c\|_1$ term may be diminished when $\lambda$ becomes excessively large. This suggests that the number of boundary collocation points $K_2$ may need to be increased in proportion to $\lambda$ to maintain a proper balance between boundary data fitting and regularization. To address this challenge, the ``mask" function trick introduced in \Cref{subsec:bnd_treat} can solve this difficulty as it avoids using a penalty method at all.
% but directly enforce homogeneous boundary value. 
With $\gamma$ in \eqref{eq:homogenizor} set to be $\gamma(x) = (x_1^{2} - 1)(x_2^{2} - 1)$, we obtain a more sparse and more accurate solutions within much fewer iterations (see Figure~\ref{fig:error_contour_sines}).

\subsection{Eikonal equation}
We consider the regularized Eikonal equation with homogeneous Dirichlet boundary condition
\begin{equation*}
\begin{aligned}
      |\nabla u(x)|^{2} + \epsilon \upDelta u(x) &= f^{2}(x) \quad&\forall x\in D \\
    u(x) &= 0 \quad&\forall x\in \partial D
\end{aligned}
\end{equation*}
where $D = [-1, 1]^{2} \subseteq[-2, 2]^{2} = \Omega$, $f\equiv 1$,  and $\epsilon \in \R_{+}$. In this case, $\cE[u] = |\nabla u| + \epsilon \upDelta u - f^{2}$ and $\cB[u] = u$. We first set $\epsilon = 0.1$ and use $K = 30^{2}$ ($K_1 = 784$, $K_2 = 116$) grid collocation points, with $\alpha = 10^{-4}$ and $\alpha = 10^{-6}$ respectively.
% The convergence of Algorithm \ref{alg:main} is shown in \Cref{fig:eikonal_convergence} and 
Resulting error contours\footnote{Exact solution is obtained using finite difference method with transformation $u = -\epsilon \log v$~\citep[Same as][]{chen2021solving} with $2000$ uniform grids in each dimension of $D$.} are shown in \Cref{fig:error_contour_eikonal}. Our method obtains a highly sparse representation of the solution, with kernel nodes concentrated along the diagonals and the center where the solution undergoes sharp transitions. Results with a wider range of $K$ and $\alpha$ are shown in \Cref{tab:Eikonal}. 
\begin{table}[t]
\centering
\small
\setlength{\tabcolsep}{4pt}
\caption{Errors and residual loss of numerical solutions to the regularized Eikonal equation ($\epsilon = 0.10$) using grid collocation points with spatial mask $\gamma(x)=(1 - x_1^2)(1 - x_2^2)$. Test results are evaluated on a $100\times 100$ grid in $D$. All results are averaged over 10 runs.}
\begin{tabular}{c@{\hskip 6pt}ccc@{\hskip 8pt}ccc@{\hskip 8pt}ccc}
\toprule
& \multicolumn{3}{c}{$K = 400$ $(324, 76)$}
& \multicolumn{3}{c}{$K = 900$ $(784, 116)$}
& \multicolumn{3}{c}{$K = 2500$ $(2304, 196)$} \\
\cmidrule(lr){2-4} \cmidrule(lr){5-7} \cmidrule(lr){8-10}
Metric & $10^{-4}$ & $10^{-6}$ & $10^{-8}$
       & $10^{-4}$ & $10^{-6}$ & $10^{-8}$
       & $10^{-4}$ & $10^{-6}$ & $10^{-8}$ \\
\midrule
$L^2$ error
& $1.94\text{e}{-2}$ & $2.97\text{e}{-3}$ & $3.99\text{e}{-3}$
& $1.94\text{e}{-2}$ & $1.43\text{e}{-3}$ & $8.71\text{e}{-4}$
& $1.95\text{e}{-2}$ & $1.20\text{e}{-3}$ & $1.91\text{e}{-4}$ \\
$L^\infty$ error
& $3.04\text{e}{-2}$ & $6.50\text{e}{-3}$ & $7.43\text{e}{-3}$
& $2.98\text{e}{-2}$ & $5.71\text{e}{-3}$ & $1.90\text{e}{-3}$
& $2.97\text{e}{-2}$ & $5.28\text{e}{-3}$ & $1.28\text{e}{-3}$ \\
$\hat{L}$(train)
& $9.40\text{e}{-3}$ & $2.66\text{e}{-4}$ & $1.30\text{e}{-6}$
& $1.10\text{e}{-2}$ & $5.13\text{e}{-4}$ & $1.35\text{e}{-5}$
& $1.24\text{e}{-2}$ & $7.85\text{e}{-4}$ & $6.01\text{e}{-5}$ \\
$\hat{L}$(test)
& $1.39\text{e}{-2}$ & $3.07\text{e}{-3}$ & $3.91\text{e}{-3}$
& $1.36\text{e}{-2}$ & $1.67\text{e}{-3}$ & $5.70\text{e}{-4}$
& $1.36\text{e}{-2}$ & $1.23\text{e}{-3}$ & $1.94\text{e}{-4}$ \\
\#Kernels
& 15 & 61 & 81
& 14 & 71 & 89
& 13 & 88 & 115 \\
\bottomrule
\end{tabular}
\label{tab:Eikonal}
\end{table}

We then check the convergence towards the unique viscosity solution as $\epsilon \rightarrow 0$. The unique viscosity solution in this case is $u_{\text{visc}} = \dist(x, \partial D) = \min(1 - |x_1|, 1 -|x_2|)$. We solve the Eikonal equation with $\epsilon = 0.5, 0.1, 0.05, 0.01$ respectively with $K = 30^2$ grid collocation points with mask function $\gamma(x) = (1-x_1^{2})(1 - x_2^{2})$ and $\alpha=10^{-6}$. We observe that our method shows excellent convergence behavior as $\epsilon$ approaches $0$ (see \Cref{fig:eikonal_viscos}). As the solution becomes increasingly non-smooth near the diagonals, kernel nodes adaptively cluster in these regions, capturing the sharp transition of the solution.

\begin{figure}[tb]
    \centering
    % First subfigure: 1/4 width
    \begin{subfigure}[t]{0.38\textwidth}
        \centering
        \includegraphics[width=\textwidth]{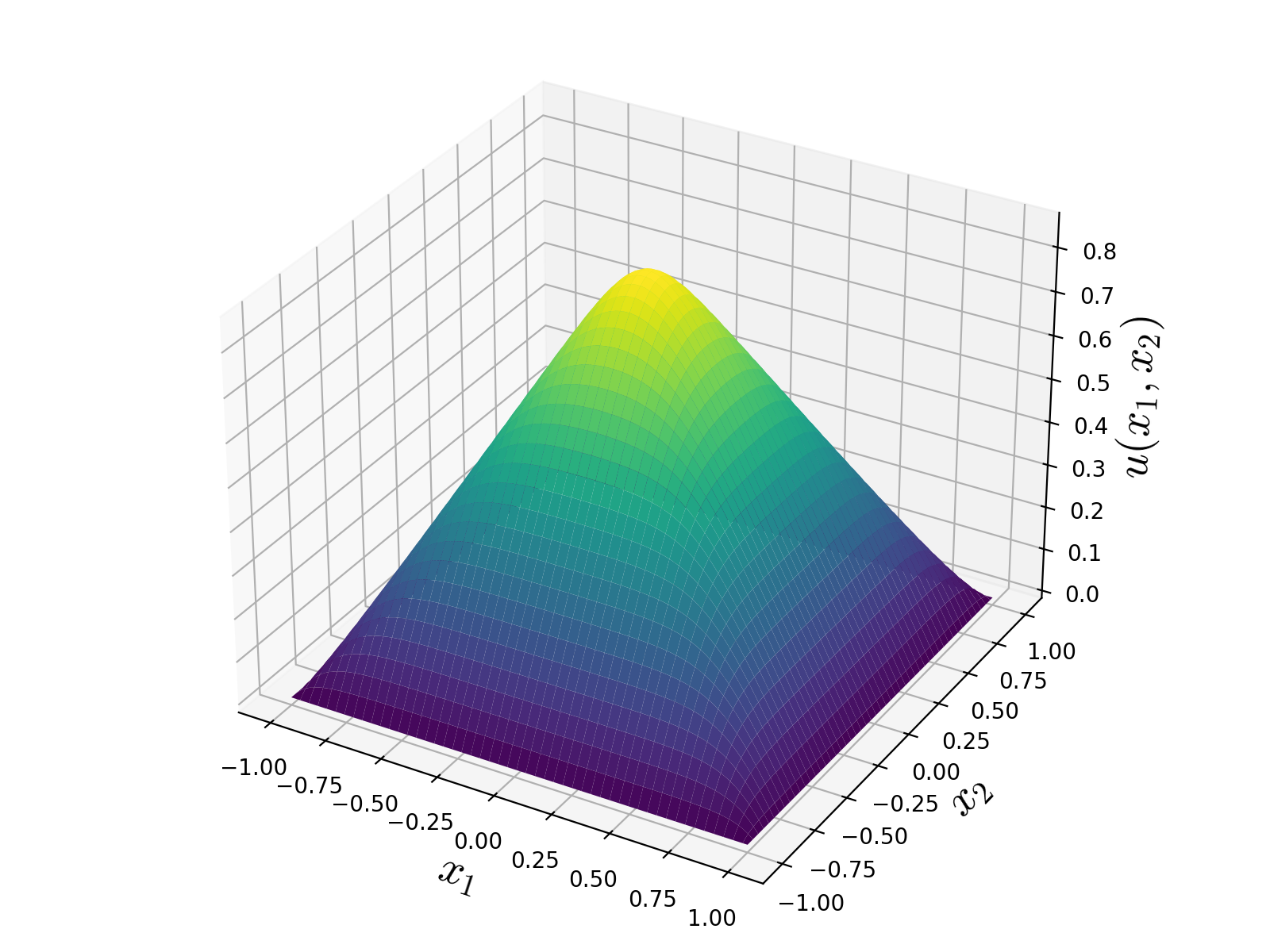}
        \caption{}
        % \caption{Exact solution $u(x_1, x_2)$, $\epsilon = 0.10$}
        \label{fig:eikonal_exact}
    \end{subfigure}
    \hspace{0.01\textwidth}
    % Second subfigure: 1/4 width
    % \begin{subfigure}[t]{0.36\textwidth}
    %     \centering
    %     \includegraphics[width=\textwidth]{plots/eikonal/eikonal_regularized_convergence.png}
    %     \caption{Convergence of loss (up) and growth of number of kernels (bottom) with $\alpha = 10^{-4}, 10^{-6}$.}
    %     \label{fig:eikonal_convergence}
    % \end{subfigure}
    % Third subfigure: 1/2 width
    % \vskip 1em
    \begin{subfigure}[t]{0.58\textwidth}
        \centering
        \includegraphics[width=1.0\textwidth]{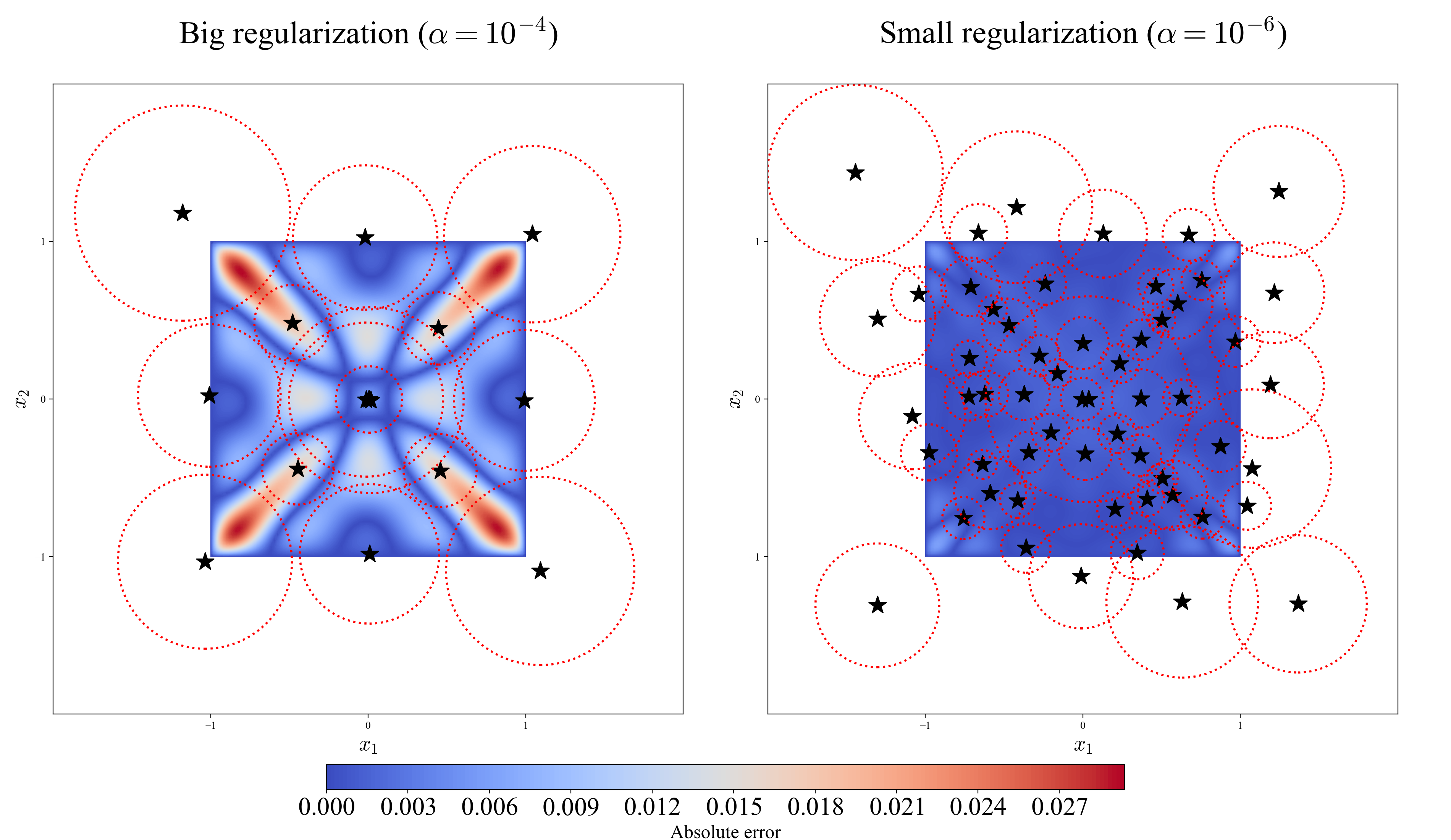}
        \caption{}
        % \caption{Error contour of numerical solutions with regularization parameter $\alpha = 10^{-4}, 10^{-6}$.}
        \label{fig:error_contour_eikonal}
    \end{subfigure}
    \caption{Numerical results of regularized Eikonal equation ($\epsilon = 0.10$). (a) Exact solution $u(x_1, x_2)$, $\epsilon = 0.10$; (b) Error contour of numerical solutions with regularization parameter $\alpha = 10^{-4}, 10^{-6}$. Each $\star$ represents the location parameters $y_n$ of learned kernel nodes, and the dotted circle around are of radius $\sigma_n$. }
    \label{fig:eikonal_regularized}
\end{figure}

\begin{figure}[t]
    \centering
    \begin{subfigure}[t]{1.0\textwidth}
        \includegraphics[width=\textwidth]{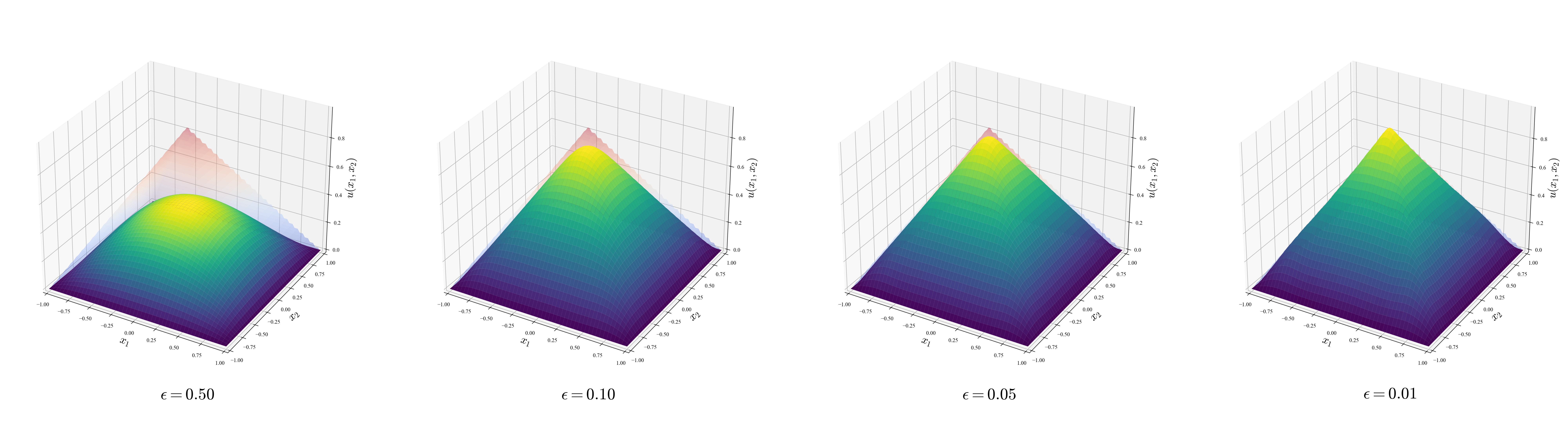}
        % \caption{Solutiuon surfaces (shadow in the background represents viscosity solution).}
        \caption{}
    \end{subfigure}
    % \vskip 1em
    \begin{subfigure}[t]{1.0\textwidth}
        \includegraphics[width=\textwidth]{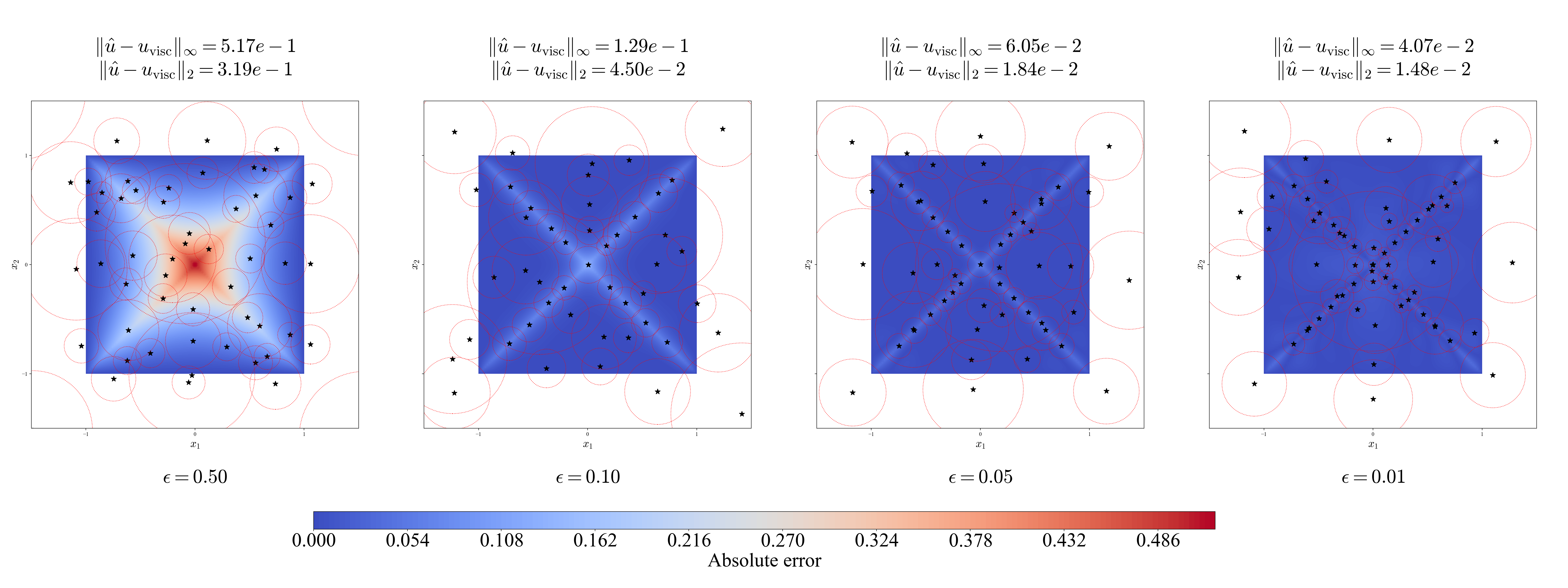}
        % \caption{Error contour with shrinking $\epsilon$. Each $\star$ represents the location parameters $y_n$ of learned kernel nodes, and the dotted circle around are of radius $\sigma_n$.  $L^{2}$ error and $L^{\infty}$ error listed are averaged from $10$ runs. \zs{Single realizations}}
        \caption{}
    \end{subfigure}
    \caption{Convergence of numerical solution to viscosity solution $u_{\text{visc}}(x) = \dist(x, \partial D)= 1 - \max(|x_{1}|, |x_{2}|)$ as $\epsilon \rightarrow 0$. (a) Solution surfaces (shadow in the background represents viscosity solution); (b) Error contours with shrinking $\epsilon$. Each $\star$ represents the location parameters $y_n$ of learned kernel nodes, and the dotted circle around are of radius $\sigma_n$. The contour plots are taken from a single run, while the reported $L^2$ and $L^{\infty}$ errors are averaged over 10 runs.}
    \label{fig:eikonal_viscos}
\end{figure}

\subsection{Viscous Burgers' equation}
\label{subsec:burgers}
We now consider the viscous Burgers' equation 
 \begin{align*}
        \partial_{t}u + u \partial_x u - 0.02 \partial_{x}^{2} u &= 0, \ \forall (t, x) \in (0, 1]\times(-1, 1)\\
        u(0, x) &= -\sin(\pi x)\\
        u(t, -1) &= u(t, 1) = 0
    \end{align*}

% In this method, we choose problems yet the loss function \eqref{eq:loss} treats them collectively through a single penalty parameter $\lambda$. As a result, our method faces challenges in selecting an appropriate value of $\lambda$ that simultaneously balances the contributions of both the initial and boundary conditions. In general, it is observed that methods where boundary/initial condition is forced as a penalty to the loss (e.g. PINN) struggles to solve dissipative equations and some remedies have been made 
%  ~\cite{biesek2023burgers,wight2020solving, mattey2022novel}. 

% To address this problem
 
Instead of using a spatial-temporal formulation as in~\cite{chen2021solving, raissi2019physics}, we employ a simple (fully) implicit backward Euler method for time discretization, which not only reduces the dimension of the problem but also alleviates the difficulty of selecting appropriate penalty parameter $\lambda$ that simultaneously balances initial and boundary conditions. The resulting algorithm solves an array of one-dimensional equations for $\{u^{n}\}_{n=1}^{N}$ ($N= \lceil\frac{1}{\Delta t} \rceil$) sequentially starting from $u^{0} = u(0, \cdot)$. To specify, we use Algorithm~\ref{alg:main} at each time step $n$ with the residuals
\begin{align*} 
    \cE^{n}[u^{n}]&:= \frac{u^{n} - u^{n-1}}{\Delta t} + u^{n}\partial_x u^{n} - \nu \partial_{x}^{2} u^n  \\
    \cB^{n}[u^{n}] &:= u^{n} .
\end{align*}

% (see Algorithm \ref{alg:burgers}).

% \begin{algorithm}[H]
%     \SetAlgoLined
%     \KwIn{Initial condition $u^{0} = u(0, \cdot)$, step size $\Delta t$, regulairzation parameter $\alpha$}
%     \KwOut{Solution $\{u^n\}_{n=1}^{N}$ approximating $u(t,x)$ on $(0, T]$}
%     Set $n=0$;\\
%     \While{$ n <  \frac{T}{\Delta t} $ }{
%         Initialize $u^n$ as $u^{n-1}$ if $n > 1$ otherwise an empty network.
        
%        Use Algorithm \ref{alg:main} to solve for $u[n]$ such that 
%        \begin{align*}
%            \frac{u^{n} - u^{n-1}}{\Delta t} + u^{n}\partial_x u^{n} - \nu \partial_{x}^{2} u^n = 0\\
%            u^n(1) = u^n(-1) = 0
%        \end{align*}
        
%         $n = n + 1$
%     }
%     \caption{Solving Burger's equation with time discretization}
%     \label{alg:burgers}
% \end{algorithm}

We set $\Delta t = 0.01$, $\alpha = 10^{-4}$ and fix $K = 40$ grid collocation points in space domain $D_x = [-1, 1] \subseteq[-2, 2] = \Omega_{x}$ for each solve of $u^{n}$. The mask function technique introduced in \Cref{subsec:bnd_treat} is used to enforce $u^{n}(\pm 1) = 0$ (with $\gamma(x) = (x+1)(x-1)$). $\{u^n\}$ is then solved sequentially, with $u^{n}$ initialized as $u^{n-1}$ for $n \geq 2$ \footnote{In practice, we solve scaled equation i.e.\ $u^{n} - u^{n-1}  + \Delta t ( u^{n}\partial_x u^{n} - \nu \partial_{x}^{2} u^{n}) = 0$ for some implementation convenience, in which $\alpha$ needs to be scaled as $\Delta t^{2}\alpha$.}. Convergence plots of $u^{n}$ for $n = 1, 20, 50, 80$ are shown in Figure \ref{fig:burgers_loss_convergence}. We remark that our sequential solving of $u^{n}$ benefits tremendously from transfer learning: since the solution is continuous in time, it is natural to initialize $u^{n}$ as $u^{n-1}$ trained in the last iteration, which effectively reduces number of iterations required for convergence. After obtaining $\{u_n\}$, the numerical solution is extended to the entire domain $D$ via linear interpolation. The error contour, together with three slices of exact and numerical solution, is shown in Figure \ref{fig:burgers_error_contour} and Figure \ref{fig:burgers_slices}. We also observe a rapid increase in the number of kernel nodes between $t = 0.2$ and $t = 0.5$, corresponding to the formation of a shock wave. During this period, the network faces increased difficulty in approximating the steepening solution, prompting the insertion of more kernel nodes. This behavior further highlights the adaptivity of our method. 

\begin{figure}[t]
\centering
\begin{subfigure}[t]{0.325\textwidth}
    \centering
    \includegraphics[width=\linewidth]{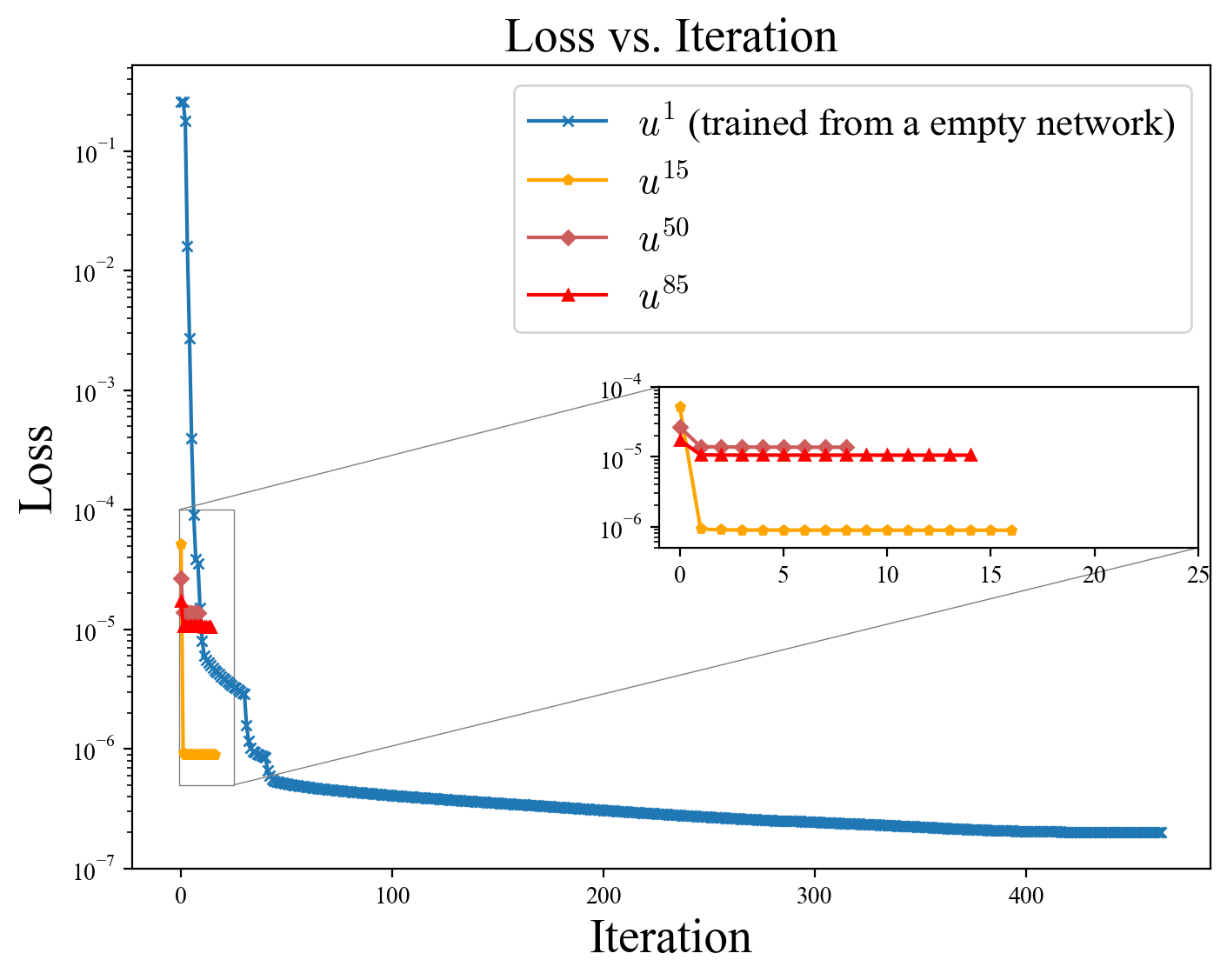}
    % \caption{Convergence of loss for selected $u^{n}$ ($\alpha = 10^{-4}$).}
    \caption{}
    \label{fig:burgers_loss_convergence}
\end{subfigure}
\begin{subfigure}[t]{0.32\textwidth}
    \centering
    \includegraphics[width=\linewidth]{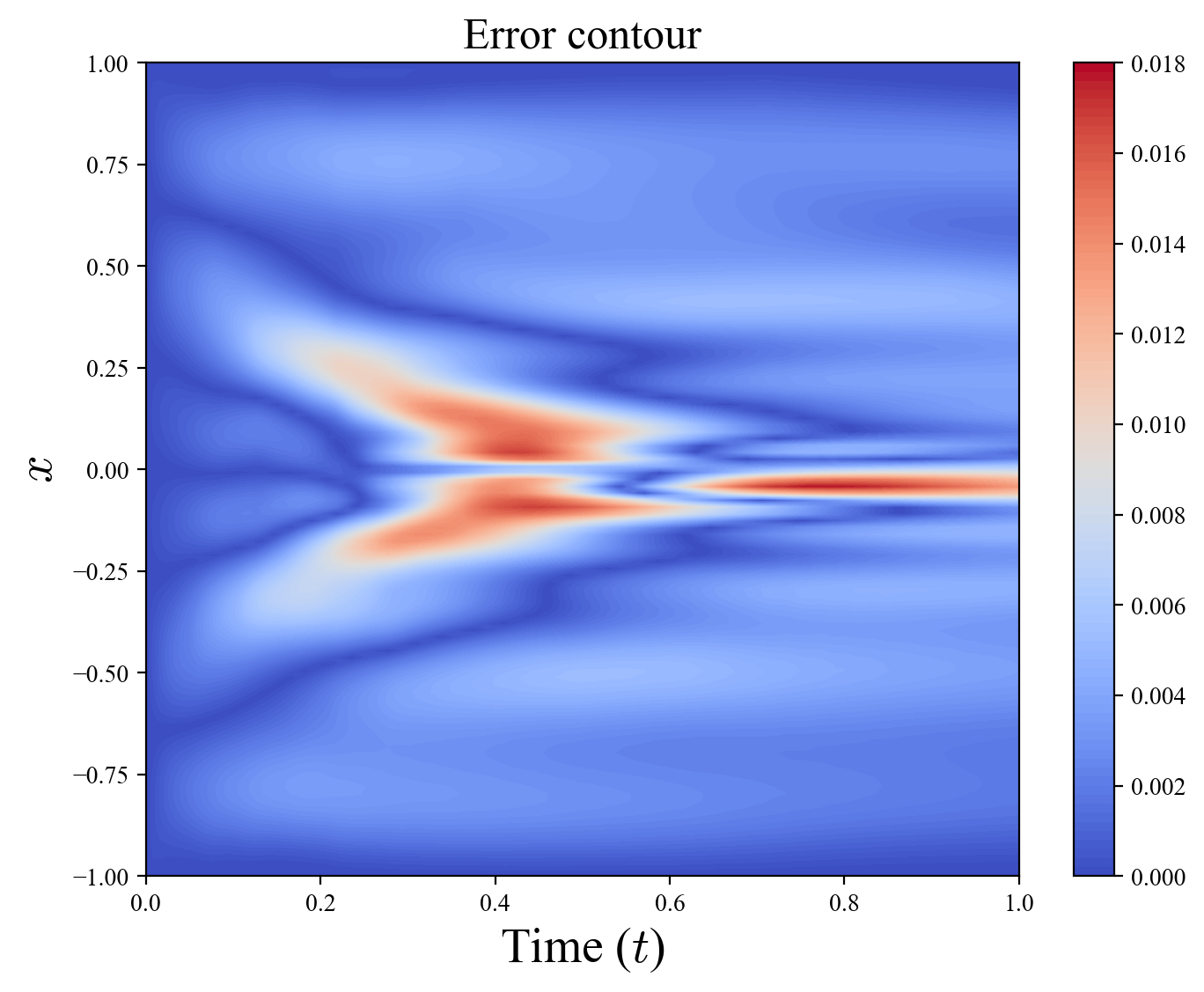}
    % \caption{Error contour ($\alpha = 10^{-4}$)}
    \caption{}
    \label{fig:burgers_error_contour}
\end{subfigure}
% \hfill[0.05\textwidth]
\begin{subfigure}[t]{0.32\textwidth}
    \centering
    \includegraphics[width=\linewidth]{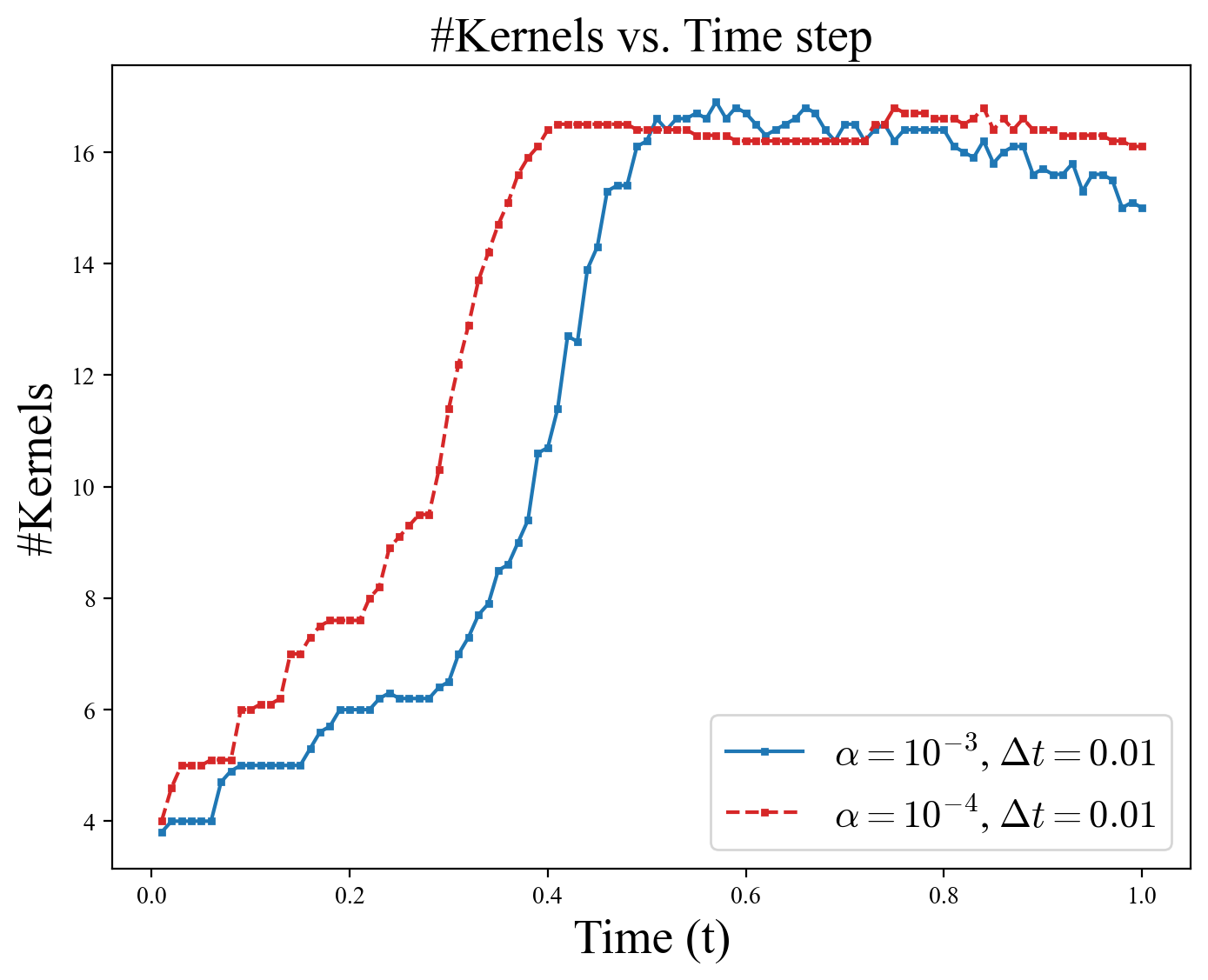}
    \caption{}
    % \caption{Growth of number of kernels in the network during time steps ($\Delta t = 0.01$) for $\alpha = 10^{-3}, 10^{-4}$ (results averaged from 10 runs). }
    % \caption{Increase of number of kernels in the network during time steps $\Delta t = 0.01$. Network with $\alpha = 10^{-3}$, though being more sparse at early time steps, inserted more kernels when error was accumulated in later. \kp{not sure I understand. I would say the main takeaway is that the method increases the number of kernels as the profile steepens (which is good). Why the difference between the two \(\alpha\) would need a better explanation, maybe a picture...}\zs{This picture is to give an explanation for similar \# kernels for big and smaller $\alpha$ in the result table. Though a big alpha results in a more compact network structure in early steps, this leads to inaccurate approximations. The error finally accumulated in $t=0.5$, where the transition started to become sharp, resulting in more kernel insertions to mitigate the error. }}
    % \label{fig:burgers_num_nodes}
\end{subfigure}

\vskip 1em
\begin{subfigure}[t]{1.0\textwidth}
    \centering
    \includegraphics[width=\linewidth]{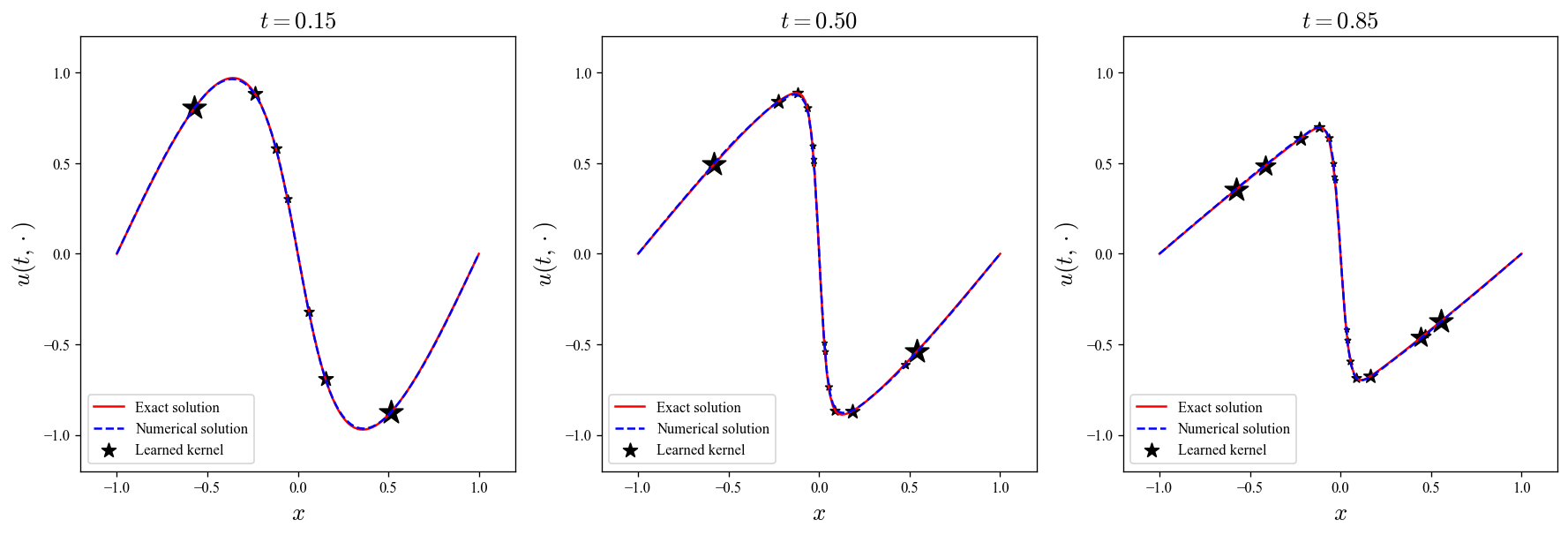}
    \caption{}
    % \caption{Time slices of exact and true solution at time $t= 0.15, 0.50, 0.85$ ($\alpha = 10^{-4}$). Size of $\star$ is proportional to the bandwidth of individual kernel. }
    \label{fig:burgers_slices}
\end{subfigure} 

\caption{Viscous Burgers equation solved by Sparse RBFNet coupled with fully implicit time discretization. (a) Convergence of loss for selected $u^{n}$ (\revision{regularization parameter}  $\alpha = 10^{-4}$); (b) Error contour (\revision{regularization parameter}  $\alpha = 10^{-4}$); (c) Growth of number of kernels in the network during time steps ($\Delta t = 0.01$) for \revision{regularization parameter} $\alpha = 10^{-3}, 10^{-4}$ (results averaged from 10 runs); (d) Time slices of exact and true solution at time $t= 0.15, 0.50, 0.85$ (\revision{regularization parameter}   $\alpha = 10^{-4}$). The size of $\star$ is proportional to the bandwidth of an individual kernel.}
\label{fig:burgers_one_run}
\end{figure}

We then test with regularization coefficient $\alpha =  10^{-3}, 10^{-4}$ and time stepsize $\Delta t= 0.1, 0.01$ (see Table \ref{tab:burgers}).
% \kp{To make sure that the regularization error is proportional to \(\Delta t\), it seems plausible to select \(\alpha = \alpha_0 \Delta t\), except for at \(t=0\).}\zs{I explained this in the footnote. In my implementation, I did use this scaling.}
% \kp{OK, good. I would suggest to move the footnote to the text, and clearly state how \(\alpha\) is chosen, to avoid confusion.}\zs{With the notation, there's no need to scale $\alpha$ accordingly. In implementation,  I would like it to keep this in the foonote.}
As is typical in backward Euler type methods, a smaller regularization parameter $\alpha$ is preferred to obtain a more accurate solve of each $u^{n}$, preventing significant error accumulation between time steps.

\begin{table}[t]
\centering
\caption{Errors and number of kernels for the viscous Burgers' equation with different combinations of time step $\Delta t$ and regularization parameter $\alpha$. Numerical solution $\{u^{n}\}$ is extended to the entire $D$ by linear interpolation. 
% \xt{Can you clarify this sentence?}. \zs{Since $\{u^{n}\}$ is obtained through time discretizations, extended to the entire domain $D$ via linear interpolation. } \xt{I didn't understand why $\hat{L}$ is relate to derivative information.}\zs{Here $\hat{L}$ refers to the one obtained by spatial-temporal formulation of PDE as we need to compute on the test grid. That one needs $\partial_x$, which is not accessible. I use $120\times 120$ test points, so that some of these do not lie on the discrete time slices.} 
% \xt{I see. Do you think it is worth including $\hat{L}$ for a final time step? I am fine either way.} \zs{I don't think we need it. I guess that will a bit misleading.}
Test results are evaluated on a grid of $120\times 120$ in $D$. The empirical residual loss $\hat{L}$ is not accessible because some test points do not lie on the discretized time steps. The number of kernels shown in the table is first averaged between time steps. All results are averaged over 10 runs.} 
\begin{tabular}{ccccc}
\toprule
$\Delta t$ & \multicolumn{2}{c}{$0.10$} & \multicolumn{2}{c}{$0.01$} \\
\cmidrule(lr){2-3} \cmidrule(lr){4-5}
$\alpha$ & $10^{-3}$ & $10^{-4}$ & $10^{-3}$ & $10^{-4}$ \\
\midrule
$L^2$ error          & $9.79\text{e}{-2}$ & $5.42\text{e}{-2}$ & $1.16\text{e}{-2}$ & $6.03\text{e}{-3}$ \\
$L^\infty$ error     & $3.39\text{e}{-1}$ & $1.73\text{e}{-1}$ & $5.55\text{e}{-2}$ & $1.93\text{e}{-2}$ \\
\#Kernels            & 9 & 17 & 15 & 16 \\
\bottomrule
\label{tab:burgers}
\end{tabular}
\end{table}
\revision{We briefly discuss the expected behavior in the small-viscosity regime. As the viscosity (set to $0.02$ in the above experiment) decreases, the solution develops a sharper shock wave, which requires more kernels to resolve the steep transition. The increased number of kernels typically leads to a more ill-conditioned Gauss–Newton system and a more challenging optimization, resulting in degraded performance. Moreover, since we employ a classical time discretization, smaller viscosity generally necessitates a finer time step, as is common in convection-dominated problems. In this regime, refined time discretization schemes (e.g., higher-order backward differentiation or adaptive stepping) could also be incorporated to further improve accuracy \citep{chen2025sparse}. Moreover, we observe that the small-viscosity regime requires a substantially larger number of collocation points ($K$) to resolve the sharper shock, analogous to the reduced mesh size requirement in classical numerical methods (see Appendix \ref{app:burgers_small_viscosity}). These observations suggest that, despite being mesh-free, the method still exhibits a resolution constraint when approximating highly singular solutions.} While valid, these points are beyond the scope of the present work and are left for future investigation.

% \begin{table}[h]
% \centering
% \caption{Error of solving $4$-dimensional Poisson equation under different numbers of collocation points $K$ and regularization parameters $\alpha$. Results are averaged from 10 runs.}
% \begin{tabular}{ccccc}
% \toprule
% $K$ $(K_D, K_{\partial D})$ & \multicolumn{2}{c}{$1296$ $(256, 1040)$} & \multicolumn{2}{c}{$4096$ $(1296, 2800)$} \\
% \cmidrule(lr){2-3} \cmidrule(lr){4-5}
% $\alpha$ & $10^{-3}$ & $10^{-4}$ & $10^{-3}$ & $10^{-4}$ \\
% \midrule
% \#Kernels          & $193$ & $287$ & $220$ & $370$ \\
% $L^2$ error      & $1.74\text{e}{-1}$ & $2.19\text{e}{-1}$ & $1.46\text{e}{-1}$ & $5.97\text{e}{-2}$ \\
% $L^\infty$ error & $2.24\text{e}{-1}$ & $4.67\text{e}{-1}$ & $1.88\text{e}{-1}$ & $8.19\text{e}{-2}$ \\
% \bottomrule
% \label{tab:highdim}
% \end{tabular}
% \end{table}

% \begin{figure}[h]
%     \centering
%     \includegraphics[width=0.5\linewidth]{plots/high_dim/high_dim_loss_and_kernels.png}
%     \caption{4d Poisson equation: Convergence of loss (up) and growth of number of kernels (bottom) with $\alpha = 10^{-3}, 10^{-4}$ and $K = 6^4, 8^4$.}
%     \label{fig:high_dim_convergence}
% \end{figure}

\section{Conclusion and discussion}
\label{sec:conclusion}

We have developed a sparse kernel/RBF network method for solving PDEs, in which the number of kernels, centers, and the kernel shape parameter are not prescribed but are instead treated as part of the optimization. Sparsity is promoted via a regularization term, added to the customary weighted quadratic $L^{2}$ loss on both interior and boundary. Theoretically, we extend the discrete network structure to a continuous integral neural network formulation, which constitutes a Reproducing Kernel Banach Space (RKBS). We establish the existence of a minimizer in the RKBS for the regularized optimization problem. Under additional assumptions, the minimizer has proven convergence properties towards the true PDE solution when the number of collocation points $K\rightarrow \infty$ and the regularization parameter $\alpha \rightarrow 0$. Notably, we prove a representer theorem that guarantees the existence of a minimizer expressible as a finite combination of feature functions (i.e., Gaussian RBFs) with a bounded number of terms. Computationally, we develop Algorithm \ref{alg:main} which effectively integrates the network width $N$ into the optimization process through an iterative three-phase framework that alternates between optimization $\omega = (y, \sigma)$ via a Gauss-Newton method and nodes insertion/deletion. Here, the insertion is guided by the dual variable derived from the first-order optimality condition. This dual variable, defined for each candidate kernel node $\omega$,  serves as an effective estimate of potential reduction in the loss function upon insertion into the network. We validate the proposed algorithm on a semilinear equation, an Eikonal equation, and a viscous Burgers’ equation, demonstrating its ability to produce sparse yet accurate solutions across a range of problem settings. Furthermore, the regularization term not only promotes sparsity in the learned solution but also plays a crucial role in preventing overfitting. Overall, our method lays a promising foundation for the development of flexible, scalable, efficient, and theoretically grounded neural network solvers capable of handling a wide range of PDEs with complex solution behaviors. 

We note that the analysis of existence and the representer theorem for minimizers is conducted within the RKBS $\cV$ associated to the kernel/feature function, and therefore depends critically on its characterization. Future research should revisit the analysis with refined understanding and characterization of the RKBS $\cV$. 
% \zs{Not sure if this true; more collocation points should result in refined solutions, and number of linear operators indicates complexity of PDE as well as the potential solution. We should only expect improvement of current linear relation. } 
Since each $\cV$ is completely defined by the associated kernel/feature function, it would also be interesting to explore alternative kernel families beyond Gaussian RBFs (e.g., Mat\'ern kernels). In particular, anisotropic kernels, where the shape parameter is given by a matrix rather than a scalar, offer more degrees of freedom to capture directional features in PDE solutions, making them a promising choice for achieving sparser representations. Moreover, current convergence analysis is based on the interplay between solution space $\cU$, source function space $\cF$, and $L^2$ space induced from the loss/objective function. To better reflect the underlying problem structure and thereby improve error analysis, it would be beneficial to develop PDE-specific loss functions that align more closely with the regularity or other properties of $\cU$ and $\cF$. Additionally, our theoretical estimate on the number of terms in the representer theorem is likely pessimistic; in practice, the sparse minimization yields significantly fewer active feature functions. Understanding this gap theoretically is an interesting direction. 
%Finally, exploring non-convex regularizers may further enhance sparsity and facilitate the characterization of local solutions to the sparse minimization problem \cite{pieper2022nonconvex}. 
% Beyond that, we may also considering using a non-convex regularization term for further sparser representation~\cite{pieper2022nonconvex}.  

% Future work on theoretical foundation include refined characterization of RKBS $\cV$. 

% \[
%      \varphi(x;\omega) = 
%     \frac{\det(\Sigma)^{s/2}}{\sqrt{2\pi}^d \det(\Sigma)} 
%     \exp\left(- \frac{(x-y)^{T}{\Sigma^{-1}}(x-y) }{2}\right)
%     \quad\text{where } \omega = (y,\Sigma) \in \R^{d+d^2}, \Sigma \geq 0
% \]

Further improvement of the computational algorithm can be achieved as well through the usage of kernels and loss functions that better capture features and regularity of both the solution and source function data. For instance, the effectiveness of a well-chosen loss function is exemplified in \Cref{subsec:num_bnd_treat} in which $H^{1}(\partial D)$ is used in place of the usual $L^{2}(\partial D)$ for boundary constraints. Additionally, as the computational cost of our method scales with the number of collocation points $K$, number of kernel functions $N$, and problem dimension $d$, it is natural to consider a doubly adaptive sub-sampling framework, where only a subset of the kernel parameters is updated with only a mini-batch of the collocation points at each iteration. While the use of mini-batches of training data is standard in the machine learning literature, incorporating it with partial kernel weights update (usually known as block coordinate descent method~\citep{wright2015coordinate, xu2013block}) requires further theoretical justification and algorithmic design. Nevertheless, this strategy can significantly improve the computational efficiency of the algorithm, particularly in cases where large $K$ and $N$ are required to capture complex solution geometries, which is common in high-dimensional PDE problems.
Finally, it is also of great interest to explore extensions of this approach to inverse problems and nonlocal or fractional equations \citep{burkardt2021unified,DDGG20,guo2022monte}.

% \xt{TODO: summarize the results and discuss future directions. 
% \begin{itemize}
%     \item Theoretical developments
%     \begin{itemize}
%         \item Refined characterization of $\cV$
%         \item A better estimate of atoms with not too small $\alpha$; appropriate function space for sparse approximation to work
%         \item Non-convex regularization and the associated theory
%         \item Suitable objective/loss functions for specific types of PDEs
%     \end{itemize}
%     \item Algorithmic developments 
%     \begin{itemize}
%         \item Anisotropic kernels
%         \item Adaptive sampling of collocation points 
%         \item High dimensional cases 
%         \item Multi-stage training? 
%     \end{itemize}
% \end{itemize}
% }

% Acknowledgements and Disclosure of Funding should go at the end, before appendices and references

\acks{Z.~Shao and X.~Tian were supported in part by NSF DMS-2240180 and the Alfred P. Sloan Fellowship.
 This manuscript has been co-authored by UT-Battelle, LLC, under contract DE-AC05-00OR22725 with the US Department of Energy (DOE). The US government retains and the publisher, by accepting the article for publication, acknowledges that the US government retains a nonexclusive, paid-up, irrevocable, worldwide license to publish or reproduce the published form of this manuscript, or allow others to do so, for US government purposes. DOE will provide public access to these results of federally sponsored research in accordance with the DOE Public Access Plan.
 
 The authors thank the anonymous referees for their helpful suggestions.}
\newpage

\appendix
\section{Optimization Algorithm}
\label{app:algo}
We provide additional implementation details for Algorithm \ref{alg:main}. 

\subsection{Phase I}
\label{app:algo_phase_1}
As mentioned earlier, we introduce several optional heuristic modifications to the insertion step.

\subsubsection{Insertion threshold} While the threshold for insertion, which is the dual variable, only characterizes the descent magnitude by optimizing $c$, we couple it with $\nabla_{y_{n}} L $. The threshold of insertion (right-hand side of \eqref{eq:insertion_threshold}) is then
\[
    \max_{1\leq n \leq N} \{ |p[u](\omega_{n})| + \eta \|\nabla_{y_{n}} L(u) \|_{2} \} 
\]
The additional term is based on the fact that we also gain descent by updating $y_n$ of preexisting nodes. In practice, this largely alleviates the possibility of having a cluster of nodes at a certain location (since the candidate nodes neighboring to preexisting ones are very likely to be inserted since their dual variables are high due to the algorithm).
The choice of $\eta$ is not critical and is typically fixed to be $0.01$. For the Burgers’ equation experiment in Section , however, we set $\eta = 0.1$ as we observed clustering of nodes otherwise.

\subsubsection{Loosening insertion criterion}
As mentioned earlier, the scheme of inserting new feature functions is essentially greedy. To accelerate convergence, we may need to relax the criteria for nodes insertion, especially during early iterations in which the learned solution is not close to the true solution. 
One option is to add a shrinking coefficient to the threshold, but tuning this coefficient is often nontrivial. We adopt a Metropolis-Hasting style criterion, in which a candidate node $\omega_{N+1}$ is accepted with probability
\begin{equation}
    \text{Pr}(\omega_{N+1}) = \min\{1, \; \exp( - \frac{1}{T\cdot \hat{R}[u]} (\max_{1\leq n \leq N}  |p[u](\omega_{n})| - |p[u](\omega_{N+1})|))\}
\end{equation}
where $T$ is a positive constant and $\hat{R}[u]$ is the relative error serving as an annealing term. This scheme admits more plausible points when the current approximation is far from the solution. In practice, we choose $T \propto -\log \alpha$ as a smaller $\alpha$ usually leads to more nodes in the end.

\subsection{Phase II}
\label{app:algo_phase_2}

\subsubsection{Levenberg–Marquardt type correction term for invertability}

To ensure invertability $DG$, we add the following correction term to approximated Hessian $\nabla^{2}\hat{\ell}$.
\[\operatorname{Cor} = 0.1 \times \|WR\|_{1} \operatorname{diag}\{\underbrace{\epsilon, \dots, \epsilon }_{N},\underbrace{|c_1|, \dots, |c_1|}_{d+1}, \underbrace{|c_2|, \dots, |c_2|}_{d+1}, \dots, \underbrace{|c_N|, \dots, |c_N|}_{d+1}\} \in \R^{N(d+2) \times N(d+2)}\]
and 
\[
    \nabla^{2}\hat{\ell}  \approx J_{R} W J_{R} + \operatorname{Cor} 
\]

\subsubsection{Line search}
As is customary in Gauss-Newton algorithm, we do a line search of optimal step size. We shrink step size $\theta$, by ${2}/{3}$ each time, starting from $\theta = 1$, until the actual descent is within trust-region, which is set to be
\[
  \text{Actual Descent} \leq h \times \text{Estimated Descent} =   h \theta G^{T}(\operatorname{DP} z).
\]
where $h\in[{1}/{5}, {1}/{3}]$. In all numerical experiments in \Cref{sec:experiments}, we set $h = {1}/{5}$.

% \subsubsection{Active points}
% To minimize the computational cost needed for assembling $J_R$ and dual variables $p$, we may consider only a subset of the collocation points. At each iteration, we select the collocation points with high $|p[u](\omega)|$ and only optimize inner weights $\omega$ of those. We call this subset of collocation points active points. However, in nonlinear PDE case, $|p[u](\omega)|$ needs to be computed at each iteration. 

\subsection{Stopping Criterion}
\label{app:algo_stopping}
We stop the iterations if

\[
    \text{Estimated Descent} < \epsilon  \text{   and   }  |p[u](\omega_{\omega_{N+1}})| -\alpha < \epsilon,
\]
which represent both insertion of new nodes and optimization of existing nodes cannot incur obvious descent to the loss function. $\epsilon$ is usually taken between $[10^{-5}, 10^{-3}]$.

\subsection{Pretraining}
\label{app:pretraining}
As mentioned earlier, we may initialize training using solutions obtained with a larger regularization parameter $\alpha$. This strategy often reduces computational cost and can occasionally yield higher-quality numerical solutions. For example, when solving the semilinear Poisson equation in \Cref{subsec:twobumps} with $K = 300$ and $\alpha = 10^{-3}$, we initialize the network using the solution computed with $\alpha = 10^{-2}$. In this case, such pretraining significantly reduces the overall computational cost (even when accounting for the training effort at $\alpha = 10^{-2}$) and achieves a smaller residual (see \Cref{fig:app_convergence}). This pretraining technique is of considerable practical importance and warrants further theoretical justification and algorithm development.

\begin{figure}[t]
    \centering
    \includegraphics[width=0.5\linewidth]{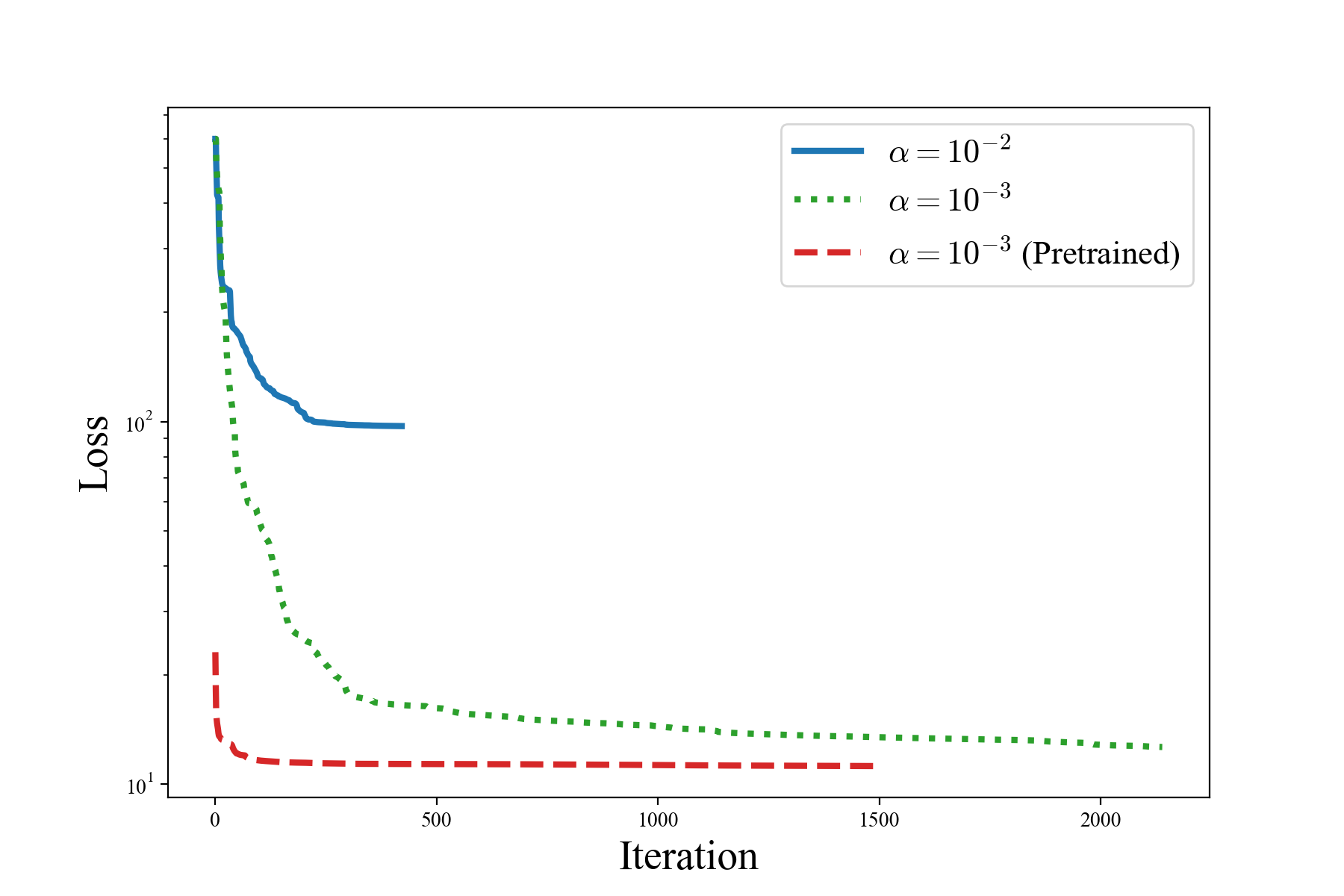}
    \caption{Convergence of residual loss for $\alpha=10^{-3}$ with/without pretraining.}
    \label{fig:app_convergence}
\end{figure}

\section{Choice of collocation points}
\label{app:collocation}
We observe that when randomly generated collocation points are used, the performance of our method degrades. Here we include the results for the semilinear Poisson equation in \Cref{subsec:twobumps}, also with comparison with the Gaussian Process method (see Table \ref{tab:two_bumps_uniform}). 
\begin{table}[t]
\centering
\caption{Errors, residual loss, and number of kernels for Sparse RBFNet and GP using \textbf{randomly} generated collocation points (in comparison to Table \ref{tab:two_bumps_uniform}). Case with $\alpha = 10^{-3}$($10^{-4}$) is solved with solution obtained by $\alpha=10^{-2}$($10^{-3}$) as initialization. Test results are evaluated on a $100\times 100$ grid in $D$. All results are averaged over 10 runs. For the Gaussian Process method, we retain the data even if some runs diverge due to random sampling.}
\setlength{\tabcolsep}{4pt}
\begin{tabular}{ccccc@{\hskip 18pt}ccc}
\toprule
& & \multicolumn{3}{c}{\textbf{Sparse RBFNet}} & \multicolumn{3}{c}{\textbf{Gaussian Process}} \\
\cmidrule(lr{2em}){3-5} \cmidrule(lr{0.2em}){6-8}
\shortstack{$K$ \\ $(K_1, K_2)$} & Metric
& $\alpha=1\text{e}{-2}$ & $\alpha=1\text{e}{-3}$ & $\alpha=1\text{e}{-4}$
& $\sigma=0.05$ & $\sigma=0.10$ & $\sigma=0.15$ \\
\specialrule{0.9pt}{1pt}{1pt}
\multirow{6}{*}{\shortstack{$400$ \\ $(324, 76)$}} 
& $L^2$ error      & $7.16\text{e}{-1}$ & $7.30\text{e}{-1}$ & $7.30\text{e}{-1}$ & $5.20\text{e}{-1}$ & $4.79\text{e}{-1}$ & $7.62\text{e}{-1}$ \\
& $L^\infty$ error  & $1.51\text{e}{+0}$ & $1.55\text{e}{+0}$ & $1.54\text{e}{+0}$ & $2.03\text{e}{+0}$ & $2.11\text{e}{+0}$ & $3.96\text{e}{+0}$ \\
& $\hat{L}$(train)        & $4.75\text{e}{+0}$ & $9.65\text{e}{-2}$ & $2.80\text{e}{-3}$ & -- & -- & -- \\
& $\hat{L}$(test)         & $2.63\text{e}{+2}$ & $2.56\text{e}{+2}$ & $2.55\text{e}{+2}$ & -- & -- & -- \\
& \#Kernels                & 46  & 75  & 87  & 400 & 400 & 400 \\
\midrule
\multirow{6}{*}{\shortstack{$900$ \\ $(784, 116)$}} 
& $L^2$ error      & $2.91\text{e}{-1}$ & $2.90\text{e}{-1}$ & $2.78\text{e}{-1}$ & $5.27\text{e}{-1}$ & $3.63\text{e}{-1}$ & $8.15\text{e}{-1}$ \\
& $L^\infty$ error   & $6.77\text{e}{-1}$ & $6.84\text{e}{-1}$ & $6.61\text{e}{-1}$ & $2.08\text{e}{+0}$ & $1.55\text{e}{+0}$ & $5.42\text{e}{+0}$ \\
& $\hat{L}$(train)        & $1.16\text{e}{+1}$ & $2.64\text{e}{-1}$ & $7.25\text{e}{-3}$ & -- & -- & -- \\
& $\hat{L}$(test)         & $9.65\text{e}{+1}$ & $7.84\text{e}{+1}$ & $7.52\text{e}{+1}$ & -- & -- & -- \\
& \#Kernels                & 58  & 108  & 135  & 900 & 900 & 900 \\
\midrule
\multirow{6}{*}{\shortstack{$2500$ \\ $(2304, 196)$}} 
& $L^2$ error      & $8.89\text{e}{-2}$ & $4.73\text{e}{-2}$ & $3.88\text{e}{-2}$ & $4.97\text{e}{-1}$ & $3.54\text{e}{-1}$ & $4.02\text{e}{-2}$ \\
& $L^\infty$ error    & $2.30\text{e}{-1}$ & $1.34\text{e}{-1}$ & $1.17\text{e}{-1}$ & $2.34\text{e}{+0}$ & $4.76\text{e}{+0}$ & $8.73\text{e}{-1}$ \\
& $\hat{L}$(train)        & $1.82\text{e}{+1}$ & $6.82\text{e}{-1}$ & $3.49\text{e}{-2}$ & -- & -- & -- \\
& $\hat{L}$(test)         & $3.48\text{e}{+1}$ & $7.13\text{e}{+0}$ & $3.93\text{e}{+0}$ & -- & -- & -- \\
& \#Kernels                & 68  & 117 & 155  & 2500 & 2500 & 2500 \\
\specialrule{1pt}{1pt}{1pt}
\end{tabular}
\label{tab:two_bumps_uniform}
\end{table}

There are several possible reasons for this decay in accuracy
\begin{enumerate}
    \item In this particular case, both our method and the Gaussian Process method have a significant drop in performance. This is likely due to the sampled collocation points failing to adequately capture the sharp transitions in the exact solution.
    \item 
    % Phase I in our algorithm is based on dual variable, which is further associated to the residue at collocation points. 
    With random points that cluster, overfitting might occur more rapidly at different parts of the domain, which may lead to non-useful kernel points being added in the insertion process.
    % \kp{I do not understand this. I would formulate that as: with random points that cluster, overfitting might occur more rapidly at different parts of the domain, which may lead to non-useful kernel points being added in the insertion process.}
\end{enumerate}

% In practice, it is crucial for a method to perform well under randomly selected collocation points. This highlights an important direction for future research.
% \kp{I would just assume that we need more random points than uniform to obtain the same coverage, which is well-known in different contexts. Often its close to a factor \(\sim 10\) compared to uniform or quasi uniform.}

In practice, it is crucial for a method to perform reliably when using randomly selected collocation points. However, achieving comparable coverage to uniform or quasi-uniform grids typically requires a substantially larger number of random points as observed in various numerical contexts. This underscores an important direction for future research on improving robustness and efficiency under random sampling.

\section{Besov space characterization}
\label{app:besov}
In this section, we characterize the RKBS norm in terms of a classical function space norm.
\begin{theorem}
Consider the unrestricted dictionary \(\Omega = \R^d \times (0, \sigma_{\max}]\)  with \(\sigma_{\max} = 1\).
Then, for the corresponding space
\[
\cV = \cV(\R^d) = \left\{ u(x) = \int_\Omega \varphi(x; \omega)\de\mu(\omega)  \;\big|\; \mu \in M(\Omega) \right\} 
\]
there is a continuous embedding of the Besov space \(B^{s}_{1,1}(\R^d)\)~\citep[as defined in, e.g., ][Section~2.1.1]{sawano2018theory} into $\cV$.
\end{theorem}
\begin{proof}
Let \(u \in B_{1,1}^s(\R^d)\). Thus, it can be decomposed as
\[
u = u_0 + \sum_{j=1}^\infty u_j
\text{ with }
u_0 = \psi(\sD)u, u_j = \phi_j(\sD)u
\text{ and } \norm{f}_{B_{1,1}^s} = \norm{u_0}_{L^1} + \sum_{j=1}^\infty 2^{j\,s} \norm{u_j}_{L^1} < \infty;
\]
~\citep[see][Section~2.1.1]{sawano2018theory}. Here, \(\psi \colon \R^d \to [0,1]\) is a function in \(C_c^\infty(\R^d)\) with
\[
\chi_{B_1(0)}(\xi) \leq \psi(\xi) \leq \chi_{B_2(0)}(\xi)
\]
and \(\phi_j = \psi_j - \psi_{j-1}\), where \(\psi_j(\xi) = \psi(2^{-j} \xi)\). Moreover, for any such function \(\psi\) (or \(\phi_j\)) the symbol \(\psi(\sD)\) denotes the multiplication in Fourier space \(\psi(\sD) u = \cF^{-1}[\psi \cdot \cF[u]]\), where \((\psi \cdot \cF[u])(\xi) = \psi(\xi) \cF[u](\xi)\).

Next, we show that every \(u_j\) can be represented as a convolution with a Gaussian \(\hat{\varphi}_j(x) = 2^{jd}\hat{\varphi}(2^{j} x)\) with scale \(\sigma_j = 2^{-j}\) in the form
\begin{equation}
\label{eq:rec_mu}
u_j = \hat{\varphi}_j * \mu_j \quad\text{with } \mu_j \in L^1(\R^d).
\end{equation}
We argue this by defining the Fourier transform of \(\hat{\varphi}_j\) as
\(g_j = \cF\hat{\varphi}_j\) with \(g_j(\xi) = g_0(2^{-j}\xi)\) and
\[
\mu_j = [\psi_{j+1}/g_j](\sD) u_j.
\]
The multiplier \(\psi_{j+1}/g_j\) is \(C^\infty\) and compactly supported, and thus its inverse Fourier transform \(\varphi^{-1}_j = \cF^{-1}(\psi_{j+1}/g_j)\) is in \(L^1(\R^d)\). Due to \(\varphi^{-1}_j(x) = 2^{j d} \varphi^{-1}_j(2^j x)\) the \(L^1\) norm is independent of \(j\) and equal to a constant \(c\), which implies
\[
\norm{\mu_j}_{L^1} = \norm{[\psi_{j+1}/g_j](\sD) u_j}_{L^1}
= \norm{\varphi^{-1}_j * u_j}_{L^1}
\leq c\norm{u_j}_{L^1}.
\]
Moreover, due to \(g_j \psi_{j+1}/g_j * \phi_j = \phi_j\) and \(g_0 \psi_{1}/g_0 * \psi = \psi\), we have~\eqref{eq:rec_mu} as claimed.
Now, we define the measure
\[
\de\mu(y,\sigma) = \sum_{j=0}^\infty\sigma^{-s}_{j}\delta_{\sigma_j}(\sigma) \mu_j(y) \de y
= \sum_{j=0}^\infty 2^{j\,s}\delta_{2^{-j}}(\sigma) \mu_j(y) \de y.
\]
We can directly verify that \(u = \cN \mu\) and 
\[
\norm{u}_{\cV(\R^d)} \leq
\norm{\mu}_{M(\Omega)} = \sum_{j=0}^\infty 2^{j\,s} \norm{\mu_j}_{L^1}
\leq c \sum_{j=0}^\infty 2^{j\,s} \norm{u_j}_{L^1}
= c \norm{u}_{B^s_{1,1}(\R^d)}.
\qedhere
\]
\end{proof}
\revision{
\section{Burgers equation with small viscosity}
\label{app:burgers_small_viscosity}
\begin{figure}[t]
    \centering
    
    \begin{subfigure}{0.495\textwidth}
        \centering
        \includegraphics[width=\textwidth]{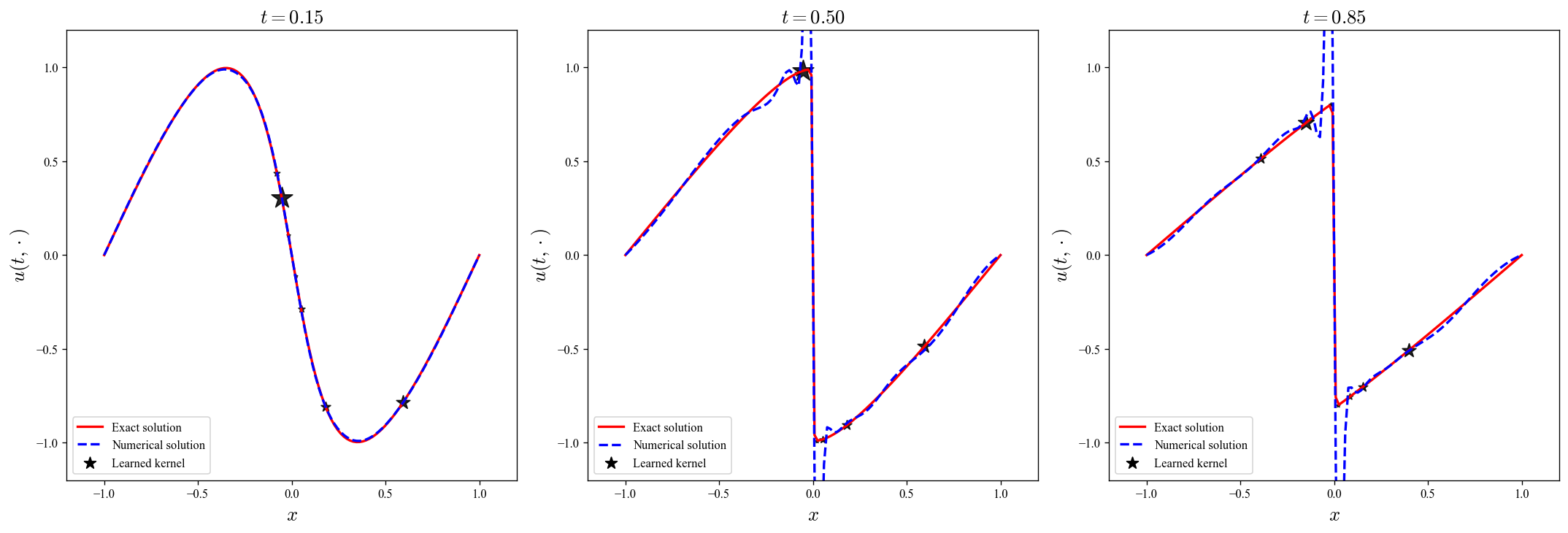}
        \caption{}
    \end{subfigure}
    % \hfill
    \begin{subfigure}{0.495\textwidth}
        \centering
        \includegraphics[width=\textwidth]{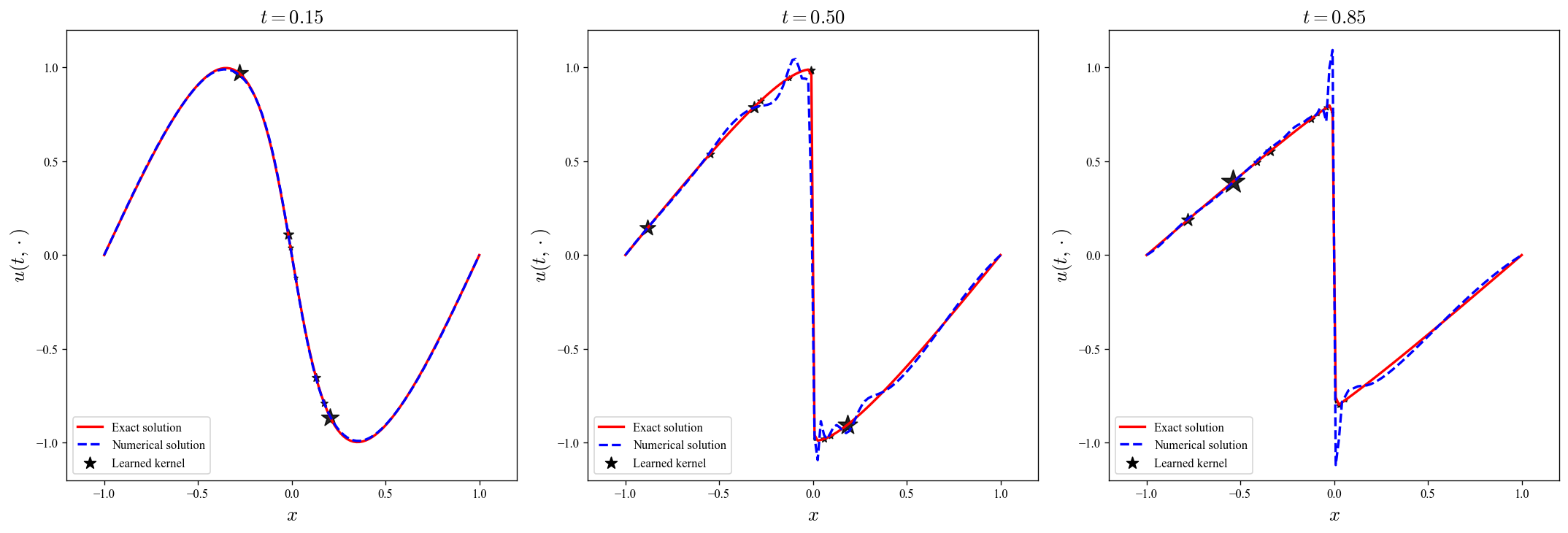}
        \caption{}
    \end{subfigure}

    \vspace{0.6em}

    \begin{subfigure}{0.495\textwidth}
        \centering
        \includegraphics[width=\textwidth]{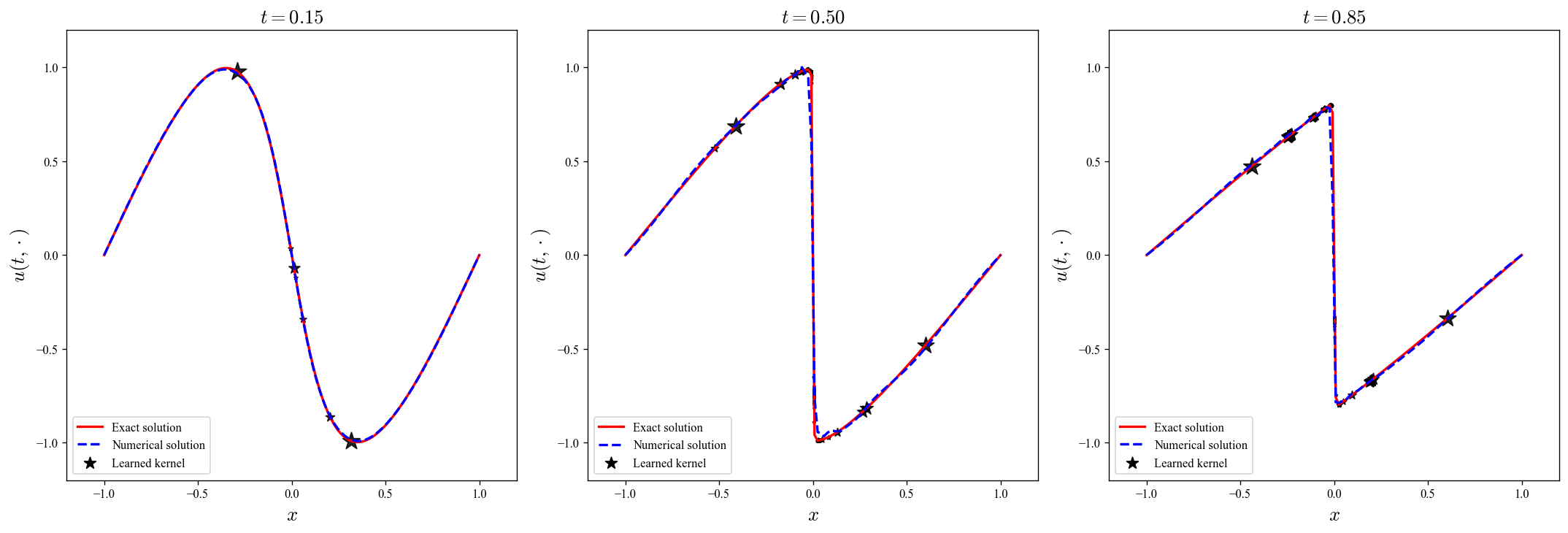}
        \caption{}
    \end{subfigure}
    % \hfill
    \begin{subfigure}{0.495\textwidth}
        \centering
        \includegraphics[width=\textwidth]{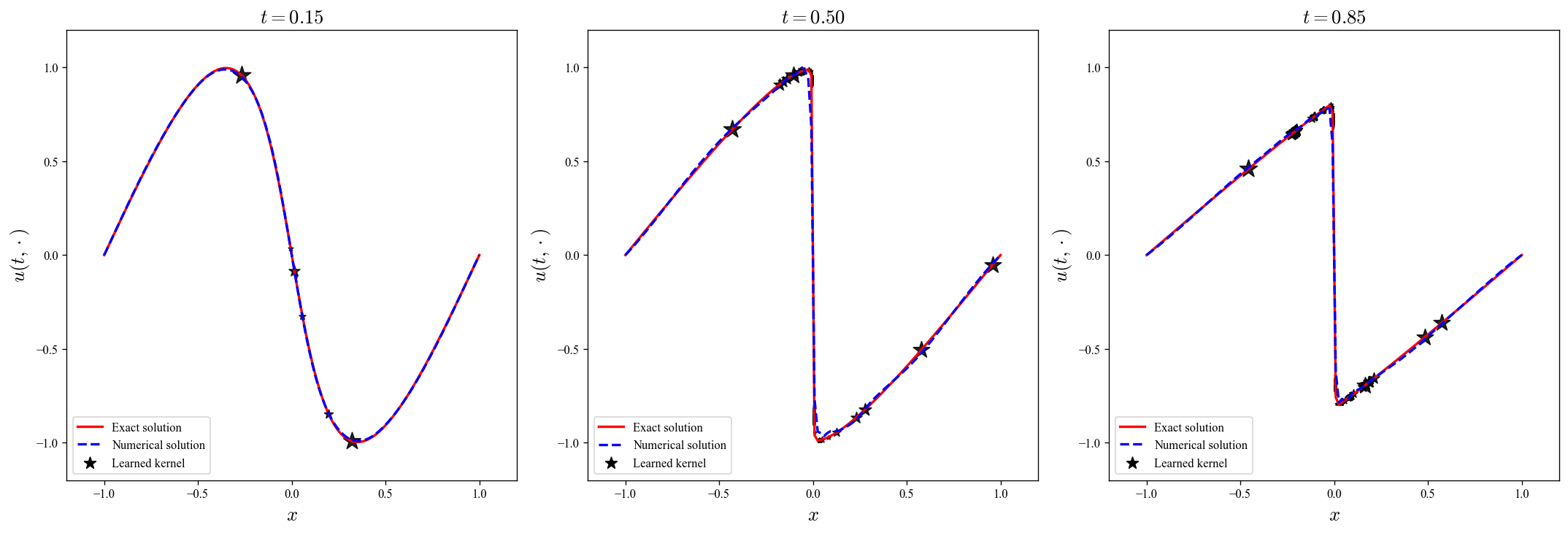}
        \caption{}
    \end{subfigure}
    
    \caption{Small-viscosity Burgers equation ($\nu = 0.002$): comparison of solution slices with different numbers of collocation points ($K$). All results are obtained using time step $\Delta t = 0.01$ and regularization parameter $\alpha = 10^{-4}$. (a) $K = 40$; (b) $K = 80$; (c) $K = 200$; (d) $K = 400$. The size of each $\star$ indicates the bandwidth of the corresponding kernel.}
    \label{fig:burgers_small_viscosity}
\end{figure}

We briefly discuss the behavior of the proposed method for the viscous Burgers equation in a more challenging small-viscosity regime. We consider the same problem as in \Cref{subsec:burgers} with a general viscosity parameter $\nu$:
\begin{align*}
    \partial_{t}u + u \partial_x u -\nu \partial_{x}^{2} u &= 0, \ \forall (t, x) \in (0, 1]\times(-1, 1)\\
    u(0, x) &= -\sin(\pi x)\\
    u(t, -1) &= u(t, 1) = 0.
\end{align*}
In \Cref{subsec:burgers}, the viscosity is set to $\nu = 0.02$.
As $\nu$ decreases, the solution develops a much steeper shock, making the approximation significantly more challenging. 
We solve the equation with $\nu = 0.002$ using the same time step $\Delta t = 0.01$ and regularization parameter $\alpha = 10^{-4}$. All other settings remain the same as in \Cref{subsec:burgers}. Figure~\ref{fig:burgers_small_viscosity} compares results with different numbers of collocation points $K = 40, 80, 200, 400$. We observe that, as the viscosity decreases, substantially more collocation points are required to resolve the sharper shock. In particular, the solution improves significantly when increasing $K$ from $40$ to $200$, while the additional gain from $K=200$ to $K=400$ is relatively modest, indicating that a sufficient spatial “resolution” has already been reached.

}

\vskip 0.2in
\bibliography{ref}

\end{document}

%% file: definitions.tex
\newtheorem{theorem}{Theorem}
\newtheorem{lemma}[theorem]{Lemma} 
\newtheorem{proposition}[theorem]{Proposition} 
\newtheorem{remark}[theorem]{Remark}

\newtheorem{assumption}[theorem]{Assumption}

\newcommand{\de}{\mathop{}\!\mathrm{d}}

\DeclareMathOperator{\tr}{tr}
\DeclareMathOperator{\supp}{supp}

\DeclarePairedDelimiter{\norm}{\lVert}{\rVert}
\DeclarePairedDelimiter{\abs}{\lvert}{\rvert}
\DeclarePairedDelimiter{\pair}{\langle}{\rangle}

\newcommand{\beq}{\begin{equation}}
\newcommand{\eeq}{\end{equation}}
\newcommand{\sgn}{\operatorname{sgn}}
\newcommand{\dist}{\text{dist}}
\newcommand{\argmin}{\text{argmin}}

\def\cB{\mathcal{B}}

\def\cE{\mathcal{E}}
\def\cF{\mathcal{F}}

\def\cH{\mathcal{H}}

\def\cL{\mathcal{L}}

\def\cN{\mathcal{N}}
\def\cO{\mathcal{O}}

\def\cR{\mathcal{R}}

\def\cU{\mathcal{U}}
\def\cV{\mathcal{V}}

\def\N{\mathbb{N}}

\def\R{\mathbb{R}}

\def\sD{\mathscr{D}}

\newcommand{\del}{\delta}